\pdfoutput=1
\documentclass[reqno,10pt]{amsart}
\usepackage{enumerate,amssymb,amsthm,hyperref}
\usepackage{amsmath,amscd}
\usepackage{mathtools}
\usepackage[nospace,noadjust]{cite}
\usepackage[margin=1in]{geometry}  
\usepackage{tikz}
\usetikzlibrary{decorations.markings,arrows,arrows.meta,cd,calc,decorations.pathmorphing,decorations.pathreplacing,backgrounds,
positioning,fit,snakes}
\usepackage{youngtab,genyoungtabtikz}
\definecolor{mygreen}{RGB}{0,128,0}

\tikzset{snake it/.style={decorate, decoration={snake, segment length=3mm, amplitude=0.5mm}}}

\tikzset{mapping/.style={decoration={
  markings,
  mark=at position #1 with {\arrow{Classical TikZ Rightarrow[length=3pt]}}},postaction={decorate}}}

\tikzset{->-i/.style={decoration={
  markings,
  mark=at position #1 with {\arrow{Classical TikZ Rightarrow[length=3pt]}}},postaction={decorate}}}

\tikzset{->-ii/.style={decoration={
  markings,
  mark=at position #1 with {\arrow{Classical TikZ Rightarrow[length=3pt]Classical TikZ Rightarrow[length=3pt]}}},postaction={decorate}}}

\tikzset{->-iii/.style={decoration={
  markings,
  mark=at position #1 with {\arrow{Classical TikZ Rightarrow[length=3pt]Classical TikZ Rightarrow[length=3pt]Classical TikZ Rightarrow[length=3pt]}}},postaction={decorate}}}
	
\tikzset{->-/.style={decoration={
  markings,
  mark=at position #1 with {\arrow{Latex[length=6pt,width=6pt]}}},postaction={decorate}}}

\tikzset{->-o/.style={decoration={
  markings,
  mark=at position #1 with {\arrow{Latex[length=6pt,width=6pt,open]}}},postaction={decorate}}}

\newtheorem{thm}{Theorem}[section]
\newtheorem{cnj}[thm]{Conjecture}

\newtheorem*{thm*}{Theorem}
\newtheorem*{thm11}{Theorem 1.1}

\newtheorem*{cor15}{Corollary 1.6}
\newtheorem*{cor16}{Corollary 1.7}

\newtheorem*{menger}{Menger's Theorem}
\newtheorem*{cnj*}{Conjecture}
\newtheorem*{que}{Question}
\newtheorem{cor}[thm]{Corollary}
\newtheorem{prop}[thm]{Proposition}
\newtheorem{lem}[thm]{Lemma}
\theoremstyle{definition}

\newtheorem{rem}[thm]{Remark}
\newtheorem{example}[thm]{Example}

\newcommand{\rk}{\operatorname{rk}}
\newcommand{\rr}{\operatorname{rr}}
\newcommand{\Cn}{\mathcal C_n}
\newcommand{\G}{\Gamma}
\newcommand{\Z}{\mathbb Z}
\newcommand{\N}{\mathbb N}
\newcommand{\T}{\mathcal T}
\newcommand{\A}{\mathcal A}
\newcommand{\val}{\operatorname{val}}
\newcommand{\im}{\operatorname{im}}
\newcommand{\valmy}{\val_{my}}
\newcommand{\valyc}{\val_{yc}}
\newcommand{\valmc}{\val_{mc}}
\newcommand{\SIG}{\operatorname{SIG}}

\begin{document}
\date{\today}

\title{Realizable ranks of joins and intersections of subgroups in free groups}

\author{Ignat Soroko}
\address{Department of Mathematics\\
Louisiana State University\\
Baton Rouge\\ LA 70803\\ USA}
\email{ignatsoroko@lsu.edu}


\begin{abstract}
The famous Hanna Neumann Conjecture (now the Friedman--Mineyev theorem) gives an upper bound for the ranks of the intersection of arbitrary subgroups $H$ and $K$ of a non-abelian free  group. It is an interesting question to `quantify' this bound with respect to the rank of $H\vee K$, the subgroup generated by $H$ and $K$. We describe a set of realizable values $\big(\!\rk(H\vee K),\rk(H\cap K)\big)$ for arbitrary $H$, $K$, and conjecture that this locus is complete. We study the combinatorial structure of the topological pushout of the core graphs for $H$ and $K$ with the help of graphs introduced by Dicks in the context of his Amalgamated Graph Conjecture. This allows us to show that certain conditions on ranks of $H\vee K$, $H\cap K$
are not realizable, thus resolving the remaining open case $m=4$ of Guzman's ``Group-Theoretic Conjecture'' in the affirmative. This in turn implies the validity of the corresponding ``Geometric Conjecture'' on hyperbolic $3$--manifolds with a $6$--free fundamental group. Finally, we prove the main conjecture describing the locus of realizable values for the case when $\rk(H)=2$.
\end{abstract}

\maketitle	

\section{Introduction}

Let $F$ be a free group and $H,K\le F$ finitely generated subgroups. Denote $H\vee K$ the subgroup of $F$ generated by $H$ and $K$. Define the \emph{reduced rank} of $H$ by  
\[
\rr(H)=\max(0,\rk(H)-1).
\]
The famous Hanna Neumann Conjecture (now the Friedman--Mineyev theorem~\cite{Mi1}, \cite{DM}, \cite{Fr}, \cite{DF}; see also a recent proof of Jaikin-Zapirain~\cite{Jai}) states that 
\[
\rr(H\cap K)\le \rr(H)\rr(K).
\]

It is an interesting problem to try to `quantify' the possible ranks of $H\cap K$ with respect to the rank of the join $H\vee K$ of $H$ and $K$ (i.e.\ the subgroup generated by $H$ and $K$). Ideally, one wishes to determine the set of all realizable values for tuples
\[
\big(\!\rk (H\vee K),\,\rk (H\cap K)\big)
\]
for any given values of $\rk(H)$ and $\rk(K)$. It seems plausible, by the analogy with the linear algebra identity for vector spaces, $\dim(U\cap V)+\dim(U+V)=\dim(U)+\dim(V)$, that the bigger $\rk (H\cap K)$ is, the smaller $\rk (H\vee K)$ should be. Several partial results and conjectures have been made in this direction.

\medskip
In~\cite{IM} Imrich and M\"uller have proved the following: 
\begin{thm*}[Imrich--M\"uller, 1994] If $H,K$ are finitely generated subgroups of $F$ and either $H$ or $K$ is of finite index in $H\vee K$ then
\begin{equation}\label{eq1}
\rr(H\cap K)\rr(H\vee K)\le \rr(H)\rr(K).
\end{equation}
\end{thm*}
\noindent
In general, without the finite index assumption, this inequality does not hold, see~\cite[Ex.~3]{IM}. Moreover, Hunt~\cite{Hun} has shown that the ratio of the left-hand side of~\eqref{eq1} to its right-hand side can be made arbitrarily large. 
Recently, Sergei Ivanov \cite[p.\,826]{Iva17} 
has posed the following open question:
\begin{que}[Ivanov, 2017]
Does inequality~\eqref{eq1} hold true if $\rr(H\cap K)$ is the maximal possible in the Friedman--Mineyev theorem, i.e.\ if $\rr(H\cap K)=\rr(H)\rr(K)>0$? Equivalently, does this assumption imply that $\rr(H\vee K)=1$?
\end{que}

\medskip
Another circle of questions about the relationship between $\rk (H\cap K)$ and $\rk (H\vee K)$ was motivated by the study of hyperbolic $3$--manifolds. 
Continuing the program started by Agol, Culler and Shalen~\cite{ACS},~\cite{CS}, Guzman~\cite{Gu} formulated the following ``Group-Theoretic Conjecture'' (GTC):
\begin{cnj*}[GTC, Guzman, 2014]
If two subgroups $H,K\le F$ both have ranks equal to $m\ge 2$, and $\rk (H\cap K)\ge m$, then $\rk (H\vee K)\le m$. 
\end{cnj*}
\noindent
She proved that this conjecture, if true, implies the following ``Geometric Conjecture'' (GC), with $k=m+2$ (recall that a group is called \emph{$k$--free} if all of its $k$--generator subgroups are free):
\begin{cnj*}[GC, Guzman, 2014]
Let $M$ be a closed, orientable, hyperbolic $3$--manifold. If\/ $\pi_1(M)$ is $k$--free for $k\ge3$ then there exists a point $P$ in $M$ such that the set of all elements of\/ $\pi_1(M,P)$ represented by loops of length less than\/ $\log(2k-1)$ is contained in a (free) subgroup of\/ $\pi_1(M)$ of rank $\le k-3$. 
\end{cnj*}
\noindent
The case $m=2$ of the GTC appeared earlier as a question in a preprint of Culler and Shalen 
and was subsequently resolved in the affirmative by Kent~\cite{Ke2} and, independently, by Louder--McReynolds~\cite{LMR}. This was used in~\cite[Th.\,1.4]{CS} to prove a statement equivalent to the GC for $k=4$ and to obtain a lower bound on the volume of a closed orientable hyperbolic $3$--manifold with a $4$--free fundamental group. (The case $k=3$ of the GC and the corresponding lower bound on the volume for the case of $3$--free fundamental groups was proved earlier in~\cite[Cor.\,9.3]{ACS}.) Using results and techniques from~\cite{Ke2}, Guzman proves the GTC for $m=3$ and hence the GC for $k=5$. However, Hunt~\cite{Hun} has shown by example that the GTC is no longer true for $m=5$. Below we will show that the GTC is false for all values $m\ge 6$, but holds true for $m=4$.

It must be noted that, very recently, Guzman and Shalen~\cite{GuSha} proved the Geometric Conjecture in full generality, without dependence on the Group-Theoretic Conjecture.

\medskip
It is an easy consequence of the Hopfian property of finite rank free groups that the only possibility for $\rk(H\vee K)$ to equal the maximal possible value $\rk(H)+\rk(K)$ is to have $H\vee K\cong H*K$. Thus, in this case $\rk(H\cap K)$ must equal $0$.

\medskip
In~\cite{Ke2}, Kent proved the following inequality:
\begin{thm*}[Kent, 2009]
Let $H$ and $K$ be nontrivial finitely generated subgroups of $F$ with reduced ranks $h=\rr(H)$, $k=\rr(K)$, and $k\ge h$. Assume also that $H\cap K\ne 1$. Then 
\[
\rr(H\cap K)\le 2hk-h\rr(H\vee K).
\]
\end{thm*}

Recently, Sergei Ivanov has improved the above estimate of Kent, see~\cite[(4.2)]{Iva} (where he mentions that inequality~\eqref{eq2} was also obtained independently by Dicks):
\begin{thm*}[Ivanov, 2018]
Let $H$ and $K$ be nontrivial finitely generated subgroups of $F$ with reduced ranks $h=\rr(H)$, $k=\rr(K)$. Then 
\begin{equation}\label{eq2}
\rr(H\cap K)\le \frac12\big(h+k-\rr(H\vee K)\big)\big(h+k-\rr(H\vee K)+1\big).
\end{equation}
\end{thm*}

\sloppy
The last result suggests that it may be convenient to describe the locus of possible values $\big(\!\rk{(H\vee K)},\rk{(H\cap K)}\big)$ in terms of the difference between $\rr(H\vee K)$ and its largest possible value $h+k+1$. If we denote $i=h+k+1-\rr(H\vee K)$ then inequality~\eqref{eq2} reads: 
$
\rr(H\cap K)\le\frac{i(i-1)}{2}.
$

\medskip
The results mentioned so far do not guarantee that if certain numbers $(v,c)$ satisfy the respective inequalities, then there exist subgroups $H$, $K$ realizing them as $(v,c)=\big(\!\rk(H\vee K),\rk(H\cap K)\big)$. Hence these results are, in effect,  describing the regions of tuples which are \emph{non-realizable}. 
By contrast, in~\cite{Ke1}, Kent exhibited for arbitrary $h,k\ge2$ a family of subgroups $H=H(h,k,m)$ and $K=K(k)$ such that $\rk(H)=h$, $\rk(K)=k$, $\rk(H\vee K)=2$ and $\rk(H\cap K)$ takes on all possible ranks $m=0,\dots,(h-1)(k-1)+1$ allowed by the Friedman--Mineyev theorem. This answered a question of Myasnikov~\cite[(AUX1)]{BMS}.

\medskip
Our first contribution to this theme is the following addition to the region of known realizable values of $\big(\!\rk(H\vee K),\rk(H\cap K)\big)$.
\begin{thm}
\label{thm:real}
Let $F$ be a free group and let integers $h,k,c,v$ satisfy $2\le h\le k$, $2\le v\le h+k$, $0\le c\le (h-1)(k-1)+1$. Define a sequence $a_i$ as follows:
\begin{align*}
a_0&=0\,;\\
a_i&=\Big\lfloor \frac {i^2}{4}\Big\rfloor +1,\quad \text{for\quad $i=1,\dots, 2(h-1)\,;$}\\
a_i&=(h-1)(i-h+1)+1,\quad \text{for\quad $i=2(h-1),\dots,h+k-2$.}
\end{align*}
If we denote $i=h+k-v$, then for any $c\le a_i$ there exist subgroups $H,K\le F$ such that\/ $\rk(H)=h$, $\rk(K)=k$, $\rk(H\cap K)=c$ and\/ $\rk(H\vee K)=v$.
\end{thm}
\noindent
The sequence $a_i$ from this theorem is a splicing of a discrete quadratic function and a linear function. The linear part exists only if $h<k$. Written in terms of reduced ranks, the quadratic part implies: ${\rr(H\cap K)\le \big\lfloor \frac{i^2}{4}\big\rfloor}$, which is smaller than Ivanov's upper bound $\frac{i(i-1)}{2}$ above, 
and the gap between the two becomes unbounded as $i$ grows. If $h=2$, the quadratic part trivializes and the sequence $a_i$ becomes especially simple: $a_i=i$ for all $i$, see Figure~\ref{fig:2.10}. 

The realizable values from Theorem~\ref{thm:real} allow us to establish the following.
\begin{cor}
Guzman's ``Group-Theoretic Conjecture'' does not hold for any $m\ge 5$.
\end{cor}

\medskip
Figure~\ref{fig:intro} depicts all the regions described above for the case $\rk(H)=\rk(K)=6$.

\begin{figure}
\begin{center}
\begin{tikzpicture}[scale=0.9]
{\small
\draw (6.5cm,0.75cm) node {$\rk (H\cap K)$};
\draw (-1.3cm,-2cm) node [rotate=90] {$\rk (H\vee K)$};
\draw (-0.5cm,-4.82cm) node {{\tiny $(i=0)$}};
\draw (-0.5cm,-4.35cm) node {{\tiny $(i=1)$}};
\draw (-0.5cm,-3.88cm) node {{\tiny $(i=2)$}};
\draw (-0.5cm,-3.45cm) node {{\tiny $(i=3)$}};
\draw (-0.5cm,-0.25cm) node {{\tiny $(i=10)$}};
\draw (-0.5cm,-1.85cm) node {{\Large $\vdots$}};

}
{\tiny
\Ylinecolour{black!50}
\Yfillcolour{black!10}
\tgyoung(0cm,0cm,%
::::::::::::::::::::::::::::,%
::::::::::::::::::::::::::::,%
:::::::::::::::::::::::;;;;;,%
:::::::::::::::::::;;;;;;;;;,%
:::::::::::::::;;;;;;;;;;;;;,%
::::::::::::;;;;;;;;;;;;;;;;,%
:::::::::;;;;;;;;;;;;;;;;;;;,%
:::::::;;;;;;;;;;;;;;;;;;;;;,%
:::::;;;;;;;;;;;;;;;;;;;;;;;,%
::::;;;;;;;;;;;;;;;;;;;;;;;;,%
:::;;;;;;;;;;;;;;;;;;;;;;;;;,%
::;;;;;;;;;;;;;;;;;;;;;;;;;;%
)
\tgyoung(0cm,0cm,::<0>:1:2:3:4:5:6:7:8:9:<10>:<{1}1>:<12>:<13>:<14>:<15>:<16>:<17>:<18>:<19>:<20>:<21>:<{2}2>:<23>:<24>:<25>:<26>,%
:2;;;;;;;;;;;;;;;;;;;;;;;;;;;,%
:3;;;;;;;;;;;;;;;;;;;;;;,%
:4;;;;;;;;;;;;;;;;;;,%
:5;;;;;;;;;;;;;;,%
:6;;;;;;;;;;;,%
:7;;;;;;;;,%
:8;;;;;;,%
:9;;;;,%
:<10>;;;,%
:<{1}1>;;,%
:<12>;%
)
\Ylinecolour{black!25}
\Yfillcolour{black!50}
\Ynodecolour{white}
\tgyoung(0cm,0cm,::::::::::::::::::::::::::::,%
:;;;;;;;;;;;;;;;;;;;;;;;;;;;,%
:;;;;;;;;;;;;;;;;;;;;;;,%
:;;;;;;;;;;;;;;;;;;,%
:;;;;;;;;;;;;;;,%
:;;;;;;;;;;;,%
:;;;;;;;;,%
:;;;;;;,%
:;;;;,%
:;;;,%
:;;,%
:;%
)
\Ynodecolour{black}
\Yfillcolour{white}
\tgyoung(0cm,0cm,%
::::::::::::::::::::::::::::,%
::::::::::::::::::::::::::::,%
::::::::::::::::::::::::::::,%
::::::::::::::::::::::::::::,%
:::::::::::::::::::::::;;;;;,%
:::::::::::::::::;;;;;;;;;;;,%
::::::::::::;;;;;;;;;;;;;;;;,%
::::::::;;;;;;;;;;;;;;;;;;;;,%
:::::;;;;;;;;;;;;;;;;;;;;;;;,%
::::;;;;;;;;;;;;;;;;;;;;;;;;,%
:::;;;;;;;;;;;;;;;;;;;;;;;;;,%
::;;;;;;;;;;;;;;;;;;;;;;;;;;%
)
\begin{scope}[black,thick]
\draw (13pt,-143pt) -- ++(351pt,0) -- ++ (0,143pt) -- ++(-351pt,0) -- ++(0,-143pt);
\draw (26pt,-143pt) -- ++(0,13pt) -- ++(338pt,0); 
\draw (13pt,-13pt) -- ++(351pt,0);
\draw[mygreen] (364pt,-13pt) -- ++(-169pt,0) -- ++(0,-13pt) -- ++(-52pt,0)-- ++(0,-13pt) -- ++(-26pt,0) -- ++(0,-13pt) -- ++(-13pt,0)-- ++(0,-13pt)-- ++(-13pt,0)-- ++(0,-13pt)-- ++(-13pt,0)-- ++(0,-26pt) -- ++(-13pt,0) -- ++(0,-39pt);
\draw[mygreen] (323pt,-20pt) node {\small {\sl ineq.}\,(1)};
\draw (39pt,-130pt) -- ++(325pt,0) -- ++(0,65pt) -- ++(-65pt,0) -- ++(0,-13pt) -- ++(-65pt,0) -- ++(0,-13pt) -- ++(-65pt,0)
-- ++(0,-13pt) -- ++(-65pt,0) -- ++(0,-13pt) -- ++(-65pt,0) -- ++(0,-13pt) -- ++(-13pt,0);
\draw (364pt,-65pt) -- ++ (0,26pt) -- ++(-52pt,0) -- ++(0,-13pt) -- ++(-78pt,0) -- ++(0,-13pt) -- ++(-65pt,0) -- ++(0,-13pt)
-- ++(-52pt,0) -- ++(0,-13pt) -- ++(-39pt,0) -- ++(0,-13pt) -- ++(-26pt,0) -- ++(0,-13pt);
\draw (299pt,-13pt) -- ++(0,-13pt) -- ++(-52pt,0) -- ++(0,-13pt) -- ++(-52pt,0)-- ++(0,-13pt) -- ++(-39pt,0)-- ++(0,-13pt) -- ++(-39pt,0)-- ++(0,-13pt) -- ++(-26pt,0)-- ++(0,-13pt) -- ++(-26pt,0)-- ++(0,-13pt);
\draw[red] (364pt,-65pt) -- ++(-273pt,0) -- ++(0,-78pt);
\draw[red] (235pt,-71.5pt) node {\small \sl Guzman's GTC $(m=6)$};
\draw[blue] (364pt,-13pt) -- ++(-13pt,0) -- ++(0,-130pt); 
\draw[blue] (357.75pt,-64pt) node [rotate=90] {\small \sl Ivanov's\ \,question};
\draw (71.5pt,-97.5pt) node {\Large$*$};
\draw (110.5pt,-84.5pt) node {\Large$*$};
\draw (162.5pt,-71.5pt) node {\Large$*$};
\draw (227.5pt,-58.5pt) node {\Large$*$};
\draw (305.5pt,-45.5pt) node {\Large$*$};
\draw (182pt,-136.5pt) node {\small \sl Hopfian \ \ \ property};
\draw (260pt,-110.5pt) node {\small \sl Kent\,'09};
\draw (172pt,-84.5pt) node {\small \sl Ivanov\,'18};
\draw[white] (182pt,-6.5pt) node {\small \sl Kent\,'05};
\draw[white] (47pt,-58.5pt) node {\small \sl Th.\,1.1};
\end{scope}
}
\end{tikzpicture}
\caption{\label{fig:intro} The locus of realizable values for $\rk(H)=\rk(K)=6$. Dark gray area: known realizable values. White area: proved non-realizable values. Light gray area: conjecturally non-realizable values (Conjecture~\ref{cnj:real}). The cells marked with an asterisk are non-realizable by Theorem~\ref{thm:guzman4}. Guzman's GTC for $m=6$ claims that the red rectangle consists entirely of non-realizable values.}
\end{center}
\end{figure}
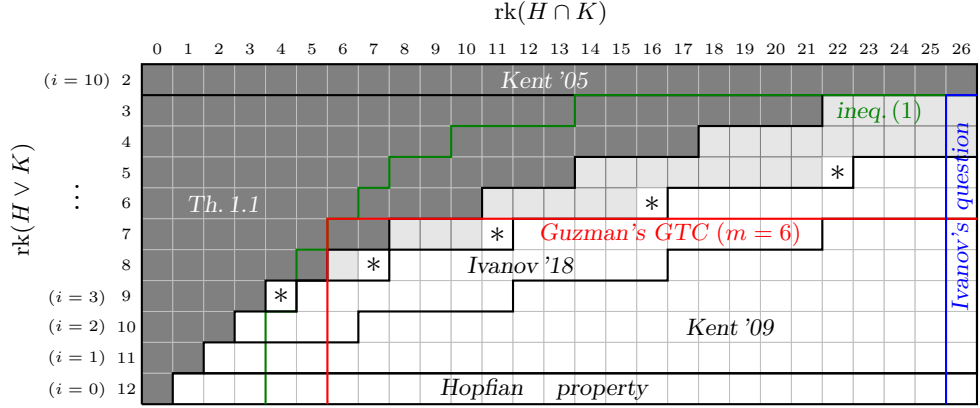

\medskip
We conjecture that the set of realizable values from Theorem~\ref{thm:real} is complete. 
\begin{cnj}\label{cnj:real}
Let $F$ be a free group and let integers $h,k,v,c$ satisfy: $2\le h\le k$, $2\le v\le h+k$ and $0\le c\le (h-1)(k-1)+1$. Then there exist subgroups $H,K\le F$ such that\/ $\rk(H)=h$, $\rk(K)=k$, $\rk(H\vee K)=v$, and\/ $\rk(H\cap K)=c$ if and only if $c\le a_i$ for $i=h+k-v$, where $a_i$ is the sequence defined in Theorem~\ref{thm:real}.
\end{cnj}

The author tested this conjecture on a computer (searching for a possible counterexample) using a Monte-Carlo type algorithm of Bassino, Nicaud and Weil~\cite{BNW}, which randomly generates core graphs on a given number of vertices with the uniform distribution for subgroups of the given size (of their core graph) in a free group. The author learned about this algorithm from the preprint of Hunt~\cite{Hun}, and also used the computer algebra system GAP~\cite{GAP} (with the package FGA~\cite{FGA} for methods dealing with free groups) to implement it. Testing about $5\cdot 10^8$ pairs of random core graphs did not produce any values outside of the conjectured locus. 

Note that Conjecture~\ref{cnj:real} subsumes the open question of Ivanov above. 

We prove this conjecture for the special case of $h=\rk(H)=2$ (when the sequence $a_i$ from Theorem~\ref{thm:real} becomes linear: $a_i=i$ for all $i$):
\begin{thm}\label{thm:h2}
Let $F$ be a free group and let integers $k,v,c$ satisfy: $k\ge 2$, $2\le v\le k+2$ and $0\le c\le k$. Then there exist subgroups $H,K\le F$ such that\/ $\rk(H)=2$, $\rk(K)=k$, $\rk(H\vee K)=v$, and\/ $\rk(H\cap K)=c$ if and only if $c+v\le k+2$.
\end{thm}

The diagram in Figure~\ref{fig:2.10} shows all realizable values for $\rk(H)=2$, $\rk(K)=10$. 

\begin{figure}[ht]
\begin{center}
\begin{tikzpicture}[scale=0.9]
{\small
\draw[xshift=-100pt] (6.5cm,0.75cm) node {$c=\rk (H\cap K)$};
\draw (-1.3cm,-2.3cm) node [rotate=90] {$v=\rk (H\vee K)$};
\draw (-0.4cm,-4.82cm) node {{\tiny $(i=0)$}};
\draw (-0.4cm,-4.35cm) node {{\tiny $(i=1)$}};
\draw (-0.4cm,-0.25cm) node {{\tiny $(i=10)$}};
\draw (-0.4cm,-2.30cm) node {{\Huge $\vdots$}};
}
{\tiny
\Ylinecolour{black!25}
\Yfillcolour{white}
\tgyoung(0cm,0cm,%
:::::::::::::::::::::::::::,%
::::::::::::,%
:::::::::::;,%
::::::::::;;,%
:::::::::;;;,%
::::::::;;;;,%
:::::::;;;;;,%
::::::;;;;;;,%
:::::;;;;;;;,%
::::;;;;;;;;,%
:::;;;;;;;;;,%
::;;;;;;;;;;%
)
\Ylinecolour{black!25}
\Yfillcolour{black!50}
\tgyoung(0cm,0cm,::0:1:2:3:4:5:6:7:8:9:<10>,%
:2;;;;;;;;;;;,%
:3;;;;;;;;;;,%
:4;;;;;;;;;,%
:5;;;;;;;;,%
:6;;;;;;;,%
:7;;;;;;,%
:8;;;;;,%
:9;;;;,%
:<10>;;;,%
:<{1}1>;;,%
:<12>;%
)
}
\end{tikzpicture}
\caption{All realizable values for $\rk(H)=2$, $\rk(K)=10$.\label{fig:2.10}}
\end{center}
\end{figure}

\medskip
Our next contribution is the resolution of a few extremal cases which are not covered by the realizable values from Theorem~\ref{thm:real} and are not eliminated by Ivanov's inequality~\eqref{eq2} (see the cells marked with an asterisk in Figure~\ref{fig:intro}).
\begin{thm}\label{thm:guzman4}
Let $F$ be a free group. Then there do not exist subgroups $H,K\le F$ such that\/ $\rk(H)$, $\rk(K)\ge 2$, $\rk(H\vee K)=\rk(H)+\rk(K)-i$ for some $i\ge 3$, and\/ $\rk(H\cap K)=\frac{i(i-1)}{2}+1$. 
\end{thm}

As an immediate consequence we conclude that Guzman's ``Group-Theoretic Conjecture'' holds true for the remaining unresolved case of $m=4$:
\begin{cor}
Let $F$ be a free group. If two subgroups $H,K\le F$ both have ranks equal to\/ $4$, and\/ ${\rk (H\cap K)\ge 4}$, then\/ $\rk (H\vee K)\le 4$. 
\end{cor}
Invoking the implication theorem from~\cite{Gu}, we obtain a proof of the ``Geometric Conjecture'' for $k=6$:
\begin{cor}
Let $M$ be a closed, orientable, hyperbolic $3$--manifold. If\/ $\pi_1(M)$ is\/ $6$--free then there exists a point $P$ in $M$ such that the set of all elements of\/ $\pi_1(M,P)$ represented by loops of length less than\/ $\log(11)$ is contained in a free subgroup of\/ $\pi_1(M)$ of rank\/ at most\/ $3$.
\end{cor}

\medskip
Our paper is organized as follows.

Section~\ref{sec:graphs} contains the basic definitions of graph related concepts necessary for our needs. We introduce Stallings' core graphs and describe the construction of the topological pushout of two core graphs, following Kent~\cite{Ke2}. The topological pushout is an intermediate object between the join of two core graphs $\G_H$, $\G_K$ and the core graph $\G_{H\vee K}$ of the join of the two subgroups, and the rank of the topological pushout serves as an upper bound for the rank of the join. 

In Section~\ref{sec:real} we show how to obtain all the values of $\big(\!\rk(H\vee K),\rk(H\cap K)\big)$ from Theorem~\ref{thm:real} by properly adding new generators to the family of examples exhibited by Kent in~\cite{Ke1}.

In Section~\ref{sec:dicks} we study the combinatorial structure of the topological pushout in terms of graphs introduced by Dicks in~\cite{Di} in the context of the Amalgamated Graph Conjecture. Our approach is motivated by the construction of graphs $\Upsilon$ and $Z$ in section 6 of~\cite{Di}. In Section~\ref{sec:cn} we establish a technical condition on Dicks graphs which specifies when the rank of the topological pushout is the maximal possible.

Finally, in Section~\ref{sec:guzman4}, we use the results obtained so far to prove Theorems~\ref{thm:h2},~\ref{thm:guzman4} and the consequences of the latter, the Guzman's GTC for the remaining case of $m=4$, and hence the GC for $k=6$. The key observation is that in the situation described in Theorems~\ref{thm:h2},~\ref{thm:guzman4}, the components of the corresponding Dicks graphs $\Omega_{abc}$ are `incompatible' in the sense that one component contains a highly connected subgraph (a complete bipartite graph $K_{i,i-1}$ in Theorem~\ref{thm:guzman4} and $K_{2,m}$ in Theorem~\ref{thm:h2}),  while others are singleton vertices. Dicks' duality implies the existence of an isomorphic copy of the highly connected subgraph in $\Omega$, which must be `spread' along two or more subgraphs $\Omega_{ab}$, $\Omega_{ac}$, $\Omega_{bc}$. This forces the rank of the join to be less than required, which makes the corresponding tuples from Theorems~\ref{thm:h2},~\ref{thm:guzman4} non-realizable.

\subsection*{Acknowledgments}
The author is very grateful to Jing Tao for bringing Guzman's ``Group-Theoretic Conjecture'' to his attention and for supporting him with the research assistantship from her grants (NSF grant DMS 1611758 and NSF Career grant DMS 1651963). The author extends his gratitude to Noel Brady for his interest in this project and numerous constructive discussions and to Sergei Ivanov who suggested to the author that Warren Dicks' methods may prove useful for this project. The author would also like to thank the anonymous referee for numerous suggestions which greatly improved the clarity of the exposition, and Pallavi Dani, Warren Dicks, Max Forester, Autumn Kent, and Lars Louder for their valuable remarks. Last but not least the author would like to thank Till Tantau and Matthew Fayers for creating \LaTeX{} packages {\tt Ti\emph{k}Z}~\cite{TikZ} and {\tt genyoungtabtikz}~\cite{gytt}, respectively, which proved to be extremely useful in typesetting this article.

\section{Graphs}\label{sec:graphs}

In this paper we will deal with two types of graphs: directed labeled graphs and undirected ones. Among the former are Stallings' core graphs that are used to represent finitely generated subgroups of a free group, see~\cite{Sta}, \cite{KM}. Among the latter ones, are bipartite graphs introduced by Dicks~\cite{Di} to study the structure of the intersection of two subgroups in a free group; they will be useful for the description of the topological pushout in the sense of Kent~\cite{Ke2}. We now remind the reader of the relevant definitions adapted to our needs from~\cite{Bog}, \cite{KM}, \cite{Di}.
 
\subsection{Basic definitions} A \emph{graph} $\G$ is a pair of sets $V(\G)$, $E(\G)$, where $V(\G)$ is a nonempty set of \emph{vertices} of $\G$ and $E(\G)$ is a set of \emph{(directed) edges} of $\G$ equipped with the three maps: $o\colon E(\G)\to V(\G)$, $t\colon E(\G)\to V(\G)$ and $\overline{\phantom{e}}\colon E(\G)\to E(\G)$ called the \emph{origin} map, the \emph{terminus} map and the map of taking the \emph{inverse} of an edge, respectively, with the following properties: for each $e\in E(\G)$, $\overline{\overline{e}}=e$, $\overline e\ne e$ and $o(\overline e)=t(e)$.

A \emph{morphism} between two graphs $\G$ and $\Delta$ is a map $\pi\colon \G\to\Delta$ 
that sends vertices to vertices and edges to edges and has the property that $o(\pi(e))=\pi(o(e))$, $t(\pi(e))=\pi(t(e))$ and $\pi(\overline e)=\overline{\pi(e)}$ for any edge $e\in E(\G)$.

Each graph $\G$ admits a geometric realization as a $1$--dimensional CW complex $X_\G$, with vertices of $\G$ being the $0$--cells of $X_\G$, and each pair of mutually inverse edges $e,\overline e$ of $\G$ corresponding to the two opposite orientations of the same open $1$--cell of $X_\G$.

A graph $\G$ is called \emph{directed} (or \emph{oriented}) if in each pair of its mutually inverse edges $e, \overline e$ one edge is chosen, which is called \emph{positively oriented}, and the other is called \emph{negatively oriented}. The set of all positively (negatively) oriented edges is denoted $E^+(\G)$ (respectively, $E^-(\G)$). A morphism of directed graphs $\pi\colon\G\to\Delta$ is required to send $E^+(\G)$ to $E^+(\Delta)$.

Let $\A$ be a finite alphabet, and $\A^{-1}$ be the set of formal inverses of $\A$. A \emph{directed $\A$--labeled graph} (or just a \emph{directed labeled graph}, if $\A$ is obvious from the context) is a directed graph $\G$ with a labeling $\mu\colon E(\G)\to \A\sqcup \A^{-1}$ such that $\mu(E^+(\G))\subseteq \A$ and $\mu(\overline e)=\mu(e)^{-1}$ for each $e\in E(\G)$. A morphism of directed $\A$--labeled graphs $\pi\colon\G\to\Delta$ is required to preserve the labeling, i.e.\ $\mu(\pi(e))=\mu(e)$ for each $e\in E(\G)$.

The \emph{star} of a vertex $v$ of $V(\G)$ is the set of all edges $e$ in $E(\G)$ such that $o(e)=v$. The star of $v$ can be thought of as the link of $v$ in the geometric realization of $\G$ in our context. 
The \emph{valence} $\val(v)$ of $v$ is the cardinality of the star of $v$. If $k=\val(v)$ we call the vertex $v$ \emph{$k$-valent}.

A morphism $\pi\colon \G\to\Delta$ is called 
an \emph{immersion} if its restriction to the star of each vertex of $\G$ is injective.

A \emph{path} $p$ in $\G$ is a sequence of edges $p=e_1,\dots,e_k$ of $E(\G)$ such that for each $i=2,\dots,k$, we have $o(e_i)=t(e_{i-1})$. The length of $p$ is set to be $k$ (with the case $k=0$ possible). In this situation we call vertex $x=o(e_1)$ the \emph{origin of $p$} and $y=t(e_k)$ the \emph{terminus of $p$}. We also say that $p$ is a path \emph{from $x$ to $y$}, and use notation $x{-}y$ to denote any such path.

If $\G$ is a directed labeled graph, then any path $p=e_1,\dots,e_k$ has a naturally defined label $\mu(p)=\mu(e_1)\dots\mu(e_k)$, which is a word in the alphabet $\A\sqcup\A^{-1}$. (If $k=0$ then $\mu(p)=1$, the empty word.)

The notion of the fundamental group of a graph is a combinatorial analog of the notion of the fundamental group of the geometric realization of the graph, see~\cite[Ch.~2.4]{Bog}.
In what follows, we will not distinguish graphs and their geometric realizations and will use these notions interchangeably. All directed labeled graphs will be labeled by the set $\A=\{a,b,c\}$ of free generators of a rank $3$ free group $F(a,b,c)$.

With each graph $\G$ we can associate an \emph{undirected graph} $\G_u$ which has the same set of vertices $V(\G)$ but whose set of \emph{undirected edges} is obtained by identifying each pair $\{e,\overline e\}$ of mutually inverse directed edges of $\G$ into a single equivalence class. Such an undirected edge has two vertices that are \emph{incident} to it, namely $\{o(e),t(e)\}=\{o(\overline e),t(\overline e)\}$, and these two vertices may coincide if $e$ is a loop. We also say that such vertices are \emph{adjacent} to each other. We will abuse the notation and denote an undirected edge $\{e,\overline e\}$ simply by $e$, and the set of all undirected edges of $\G_u$ also by $E(\G_u)$. A \emph{path} in an undirected graph is a sequence $p=x_1e_1x_2e_2\dots x_ke_kx_{k+1}$ of pairwise distinct vertices $x_i$ and undirected edges $e_i$ such that for all $i$, the vertices $x_i,x_{i+1}$ are incident to the edge $e_i$. We say that $p$ is a path \emph{from $x_1$ to $x_{k+1}$}, and we will also denote it as $x_1{-}x_{k+1}$. The length of it is $k$, as above, with $k=0$ possible. If we denote a subpath $e_1x_2e_2\dots x_ke_k$ as $q$, we can also write: $p=x_1qx_{k+1}$.
A \emph{cycle} in an undirected graph $\G$ is a union of a path $x_0e_0x_1e_1\dots x_k$ with an edge $e_k$ which is incident to both $x_k$ and $x_0$. We denote a cycle also as $x_0e_0x_1e_1\dots x_ke_kx_0$ and consider its length to be equal $k$.

\subsection{Core graphs represent subgroups} Let $F=F(a,b)$ be a free group of rank $2$ and let $X$ be a finite graph viewed as a $1$--dimensional CW complex such that $\pi_1(X)$ is isomorphic to $F$. Traditionally $X$ is identified with a wedge of two circles, but we will implement Dicks' approach~\cite{Di} and take $X$ to be the graph with two $0$--cells $u$, $v$ and three $1$--cells  $a$, $b$, $c$ all originating at $u$ and terminating at $v$, see Figure~\ref{fig:xgraph}. If we take $u$ as the basepoint for $X$, this realizes $F$ as a subgroup of a rank $3$ free group $F(a,b,c)$ on free generators $\{a,b,c\}$ with the inclusion $\theta\colon F(a,b)\xhookrightarrow{\quad} F(a,b,c)$ given by $a\mapsto ca^{-1}$, $b\mapsto cb^{-1}$. (Choosing $v$ as the basepoint of $X$ yields an inclusion given by $a\mapsto a^{-1}c$ and $b\mapsto b^{-1}c$.)

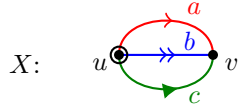
\begin{figure}
\begin{center}
\begin{tikzpicture}
\begin{scope}[scale=1.25,looseness=1.25,thick] 
\draw[->-i=0.55,color=red] (5.5,-0.5) to [out=90, in=90] (6.5,-0.5);
\draw[->-ii=0.6,color=blue] (5.5,-0.5) to [out=0, in=180] (6.5,-0.5);
\draw[->-=0.6,color=mygreen] (5.5,-0.5) to [out=-90, in=-90] (6.5,-0.5);
\fill (5.5,-0.5) circle (1.5pt); \draw (5.5,-0.5) circle (2.5pt);
\fill (6.5,-0.5) circle (1.5pt);
\draw (4.5,-0.6) node {$X$:};
\draw (5.3,-0.6) node {$u$};
\draw (6.7,-0.6) node {$v$};
\draw[color=red] (6.3,0) node {$a$};
\draw[color=blue] (6.25,-0.35) node {$b$};
\draw[color=mygreen] (6.3,-0.95) node {$c$};
\end{scope}
\end{tikzpicture}
\end{center}
\caption{The graph $X$ with $\pi_1(X,u)$ equal to $F(ca^{-1},cb^{-1})\le F(a,b,c)$.\label{fig:xgraph}}
\end{figure}

For any subgroup $H\le F$ there is a covering $\widetilde X_H\to X$ corresponding to $H$. If we fix the vertex $u$ (or $v$, for that matter) as the basepoint of $X$, there is a choice of the basepoint $x_H$ in $\widetilde X_H$ such that $\pi_1(\widetilde X_H, x_H)$ is identical to $H$. (Such choice is not unique if $\widetilde X_H$ has a nontrivial deck transformation.) Let $\G_H$ be the smallest subgraph of $\widetilde X_H$ containing $x_H$ that carries $\pi_1(\widetilde X_H,x_H)=H$. We call $(\G_H,x_H)$ the \emph{core graph} for~$H$. The vertices of $\G_H$ fall into two classes: `sources' (preimages of $u\in X$) and `sinks' (preimages of $v\in X$). Every edge of $\G_H$ is oriented from a source to a sink and inherits a unique label $a$, $b$, or $c$ induced by the covering map $\widetilde X_H\to X$. Notice that this labeling is \emph{proper} in the sense that for every vertex $x$ of $\G_H$ and each letter $\eta\in\{a,b,c\}$ there is at most one edge in $\G_H$ with the origin $x$ labeled $\eta$ and there is at most one edge in $\G_H$ with the terminus $x$ labeled $\eta$. Notice also that any vertex of $\G_H$ is at most $3$--valent, and the only vertex that may have valence $1$ is the basepoint $x_H$. 

In what follows, we will call the edges of $\G_H$ which map to the edge $a$ (respectively, $b$, $c$) of $X$, as $a$--edges (resp., $b$--edges, $c$--edges), and paint them in diagrams with red (resp., blue, green) color. (We also depict $a$--edges with a single arrow, $b$--edges with a double arrow, and $c$--edges with a solid triangular arrow in all diagrams.)

\subsection{Intersection of subgroups is represented by pullback} Let $H$, $K\le F$ be two finitely generated subgroups of $F$ and $(\G_H, x_H)$, $(\G_K,x_K)$ be the corresponding core graphs, with the natural maps $p_H\colon(\G_H, x_H)\to (X,u)$,\ \  $p_K\colon(\G_K, x_K)\to (X,u)$, which are injective on links of vertices, i.e.\ are immersions. Let $G_{H\cap K}$ be the pullback of these maps, defined as follows. The vertex set of $G_{H\cap K}$ is $V(\G_H)\times V(\G_K)$ and there is an oriented edge labeled $\eta$ ($\eta\in\{a,b,c\}$) from the vertex $(p,q)$ to the vertex $(r,s)$ in $G_{H\cap K}$ if and only if there is an edge labeled $\eta$ from $p$ to $r$ in $\G_H$ and an edge labeled $\eta$ from $q$ to $s$ in $\G_K$. The natural projections $V(\G_H)\times V(\G_K)\to V(\G_H)$ and $V(\G_H)\times V(\G_K)\to V(\G_K)$ give rise to immersions $\Pi_H\colon G_{H\cap K}\to \G_H$, $\Pi_K\colon G_{H\cap K}\to \G_K$, and the fundamental group of the component of $G_{H\cap K}$ containing the basepoint $(x_H,x_K)$ is equal to $H\cap K$, see~\cite[Th.~5.5]{Sta}. Denote by $\G_{H\cap K}$ the minimal subgraph of $G_{H\cap K}$ that contains $(x_H,x_K)$ and carries the fundamental group of the connected component of $(x_H,x_K)$. Then $\big(\G_{H\cap K},(x_H,x_K)\big)$ is the core graph for $H\cap K$. The construction guarantees that the two compositions of immersions $\big(\G_{H\cap K},(x_H,x_K)\big)\xrightarrow{\,\Pi_H\,}(\G_H,x_H)\xrightarrow{\,p_H\,} (X,u)$ and $\big(\G_{H\cap K},(x_H,x_K)\big)\xrightarrow{\,\Pi_K\,}(\G_K,x_K)\xrightarrow{\,p_K\,} (X,u)$ commute and thus define the canonical immersion $\big(\G_{H\cap K},(x_H,x_K)\big)\to (X,u)$.

\begin{example}\label{ex:main}
Figure~\ref{fig:mainpic} shows the core graphs $\G_H$, $\G_K$, $\G_{H\cap K}$ for the subgroups $H,K\le \theta(F)\le F(a,b,c)$ given by $H=\theta\big(\langle a,bab^{-1}\rangle\big)=\langle ca^{-1}, cb^{-1}ca^{-1}bc^{-1}\rangle$, $K=\theta\big(\langle b^{-1}a,ba\rangle\big)=\langle ba^{-1},cb^{-1}ca^{-1}\rangle$, and for their intersection $H\cap K=\theta\big(\langle bab^{-1}a\rangle\big)=\langle cb^{-1}ca^{-1}ba^{-1}\rangle$.
\end{example}

\begin{figure}[ht]
\begin{center}
\begin{tikzpicture}
\begin{scope}[scale=1.25]
\begin{scope}[thick]

\begin{scope}[yshift=-0.5cm]
\draw[->-i=0.55,color=red] (1,-0.5) to [out=90, in=90,looseness=0.75] (2,-0.5);
\draw[->-=0.6,color=mygreen] (1,-0.5) to [out=-90, in=-90,looseness=0.75] (2,-0.5);
\draw[->-ii=0.6,color=blue] (3,-0.5) to [out=180, in=0,looseness=1] (2,-0.5);
\draw[->-i=0.55,color=red] (3,-0.5) to [out=90, in=90,looseness=0.75] (4,-0.5);
\draw[->-=0.6,color=mygreen] (3,-0.5) to [out=-90, in=-90,looseness=0.75] (4,-0.5);
\draw (1,-0.9) node {$1$};
\draw (2,-0.9) node {$2$};
\draw (3,-0.9) node {$3$};
\draw (4,-0.9) node {$4$};
\draw (0.85,-0.25) node {$\scriptstyle x_H$};
\end{scope}

\begin{scope}[xshift=0.5cm]
\draw[->-ii=0.52,color=blue] (6,4) to [out=180, in=180,looseness=0.75] (6,1);
\draw[->-=0.65,color=mygreen] (6,4) to [out=-90, in=90,looseness=1] (6,3);
\draw[->-=0.65,color=mygreen] (6,2) to [out=-90, in=90,looseness=1] (6,1);
\draw[->-i=0.55,color=red] (6,2) to [out=180, in=180,looseness=0.75] (6,3);
\draw[->-ii=0.6,color=blue] (6,2) to [out=0, in=0,looseness=0.75] (6,3);
\draw (6.4,4) node {$5$};
\draw (6.4,3) node {$6$};
\draw (6.4,2) node {$7$};
\draw (6.4,1) node {$8$};
\draw (5.75,1.85) node {$\scriptstyle x_K$};
\end{scope}

\draw[->-=0.62,color=mygreen] (1,2) to [out=-45, in=135,looseness=1] (2,1);
\draw[->-=0.62,color=mygreen] (3,4) to [out=-45, in=135,looseness=1] (4,3);
\draw[->-o=0.6,thin,densely dashed,color=mygreen] (1,4) to [out=-45, in=135,looseness=1] (2,3);
\draw[->-o=0.6,thin,densely dashed,color=mygreen] (3,2) to [out=-45, in=135,looseness=1] (4,1);
\draw[->-i=0.55,color=red] (1,2) to [out=45, in=-135,looseness=1] (2,3);
\draw[->-i=0.55,color=red] (3,2) to [out=45, in=-135,looseness=1] (4,3);
\draw[->-ii=0.35,color=blue] (3,2) to [out=135, in=-45,looseness=1] (2,3);
\draw[->-ii=0.65,color=blue] (3,4) to (2,1);
\end{scope}

\foreach \x in {1,2,3,4}
{
		\fill (\x,-1) circle (1.5pt);
		\fill (6.5,\x) circle (1.5pt);
		
	\foreach \y in {1,2,3,4}
	{
			\fill (\x,\y) circle (1.5pt);
	}
}

\draw (1,-1) circle (2.5pt);
\draw (6.5,2) circle (2.5pt);
\draw (1,2) circle (2.5pt);

\begin{scope}[xshift=0.5cm, yshift=-0.5cm]
\draw[->-i=0.55,thick,color=red] (5.5,-0.5) to [out=90, in=90,looseness=0.75] (6.5,-0.5);
\draw[->-ii=0.6,thick,color=blue] (5.5,-0.5) to [out=0, in=180,looseness=0.75] (6.5,-0.5);
\draw[->-=0.6,thick,color=mygreen] (5.5,-0.5) to [out=-90, in=-90,looseness=0.75] (6.5,-0.5);
\fill (5.5,-0.5) circle (1.5pt);
\fill (6.5,-0.5) circle (1.5pt);
\draw (5.5,-0.9) node {$u$};
\draw (6.5,-0.9) node {$v$};
\draw (7,-0.5) node {$X$};
\draw (5.5,-0.5) circle (2.5pt);
\end{scope}

\draw (0.4,-1) node {$\G_H$};
\draw (6.5,4.5) node {$\G_K$};
\draw (0.4,2) node {$\G_{H\cap K}$};
\draw (0.3,4) node {$G_{H\cap K}$};

\draw[->] (4.55,2.5) to (5.45,2.5); 
\draw (5,2.75) node {$\Pi_K$};
\draw[->] (2.5,0.45) to (2.5,-0.35); 
\draw (2.85,0.05) node {$\Pi_H$};
\draw[->] (4.55,-1) to (5.45,-1); 
\draw (5,-0.75) node {$p_H$};
\draw[->] (6.5,0.45) to (6.5,-0.35); 
\draw (6.85,0.05) node {$p_K$};
\end{scope}

\begin{scope}[yshift=-1.5cm, xshift=2cm]
\begin{scope}[thick]
\draw[->-i=0.53,color=red] (2,-2) to [out=90, in=90,looseness=0.75] (4,-2);
\draw[->-ii=0.58,color=blue] (2,-2) to [out=0, in=-180,looseness=0.75] (4,-2);
\draw[->-=0.58,color=mygreen] (2,-2) to [out=-45, in=-135,looseness=0.75] (4,-2);
\draw[->-=0.57,color=mygreen] (2,-2) to [out=-90, in=-90,looseness=1] (4,-2);
\end{scope}
\fill (2,-2) circle (2pt);
\fill (4,-2) circle (2pt);
\draw (1,-2) node {$\{1,3,5,7\}$};
\draw (5,-2) node {$\{2,4,6,8\}$};
\draw (-0.75,-2) node {$\T$:};
\draw (2,-2) circle (3pt);
\end{scope}

\begin{scope}[xshift=1.5cm,yshift=-0.3cm]
\draw (-2,-5) node {\textit{Legend:}};
\draw (0,-5) node {$a$--edges:};
\draw[->-i=0.58,color=red,thick] (0.85,-5) to (2,-5) [out=0, in=180];
\draw (3.45,-5) node {$b$--edges:};
\draw[->-ii=0.58,color=blue,thick] (4.3,-5) to (5.45,-5) [out=0, in=180];
\draw (6.75,-5) node {$c$--edges:};
\draw[->-=0.64,color=mygreen,thick] (7.60,-5) to (8.75,-5) [out=0, in=180];
\end{scope}

\end{tikzpicture}
\end{center}
\caption{The core graphs for the subgroups 
$H,K,H\cap K$ of Example~\ref{ex:main} and their topological pushout $\T$. The thin dashed $c$--edges belong to $G_{H\cap K}$ but not to $\G_{H\cap K}$. The set of all $a$--edges of $\G_H\sqcup \G_K$ and the set of all $b$--edges of $\G_H\sqcup \G_K$ each form their own single equivalence class of edges in $\T$, whereas the $c$--edges of $\G_H\sqcup \G_K$ form two $2$--element classes of edges in $\T$: $\{(1,2),(7,8)\}$ and $\{(3,4),(5,6)\}$. 
\label{fig:mainpic}}
\end{figure}
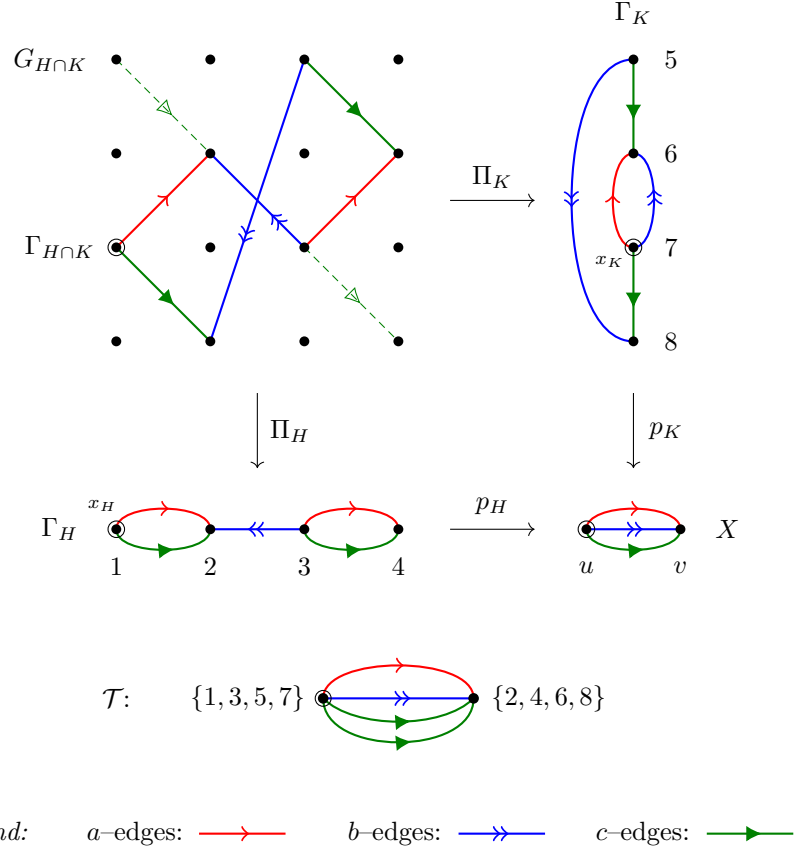

Now we show that without loss of generality we can assume that the core graphs $\G_H$, $\G_K$ and $\G_{H\cap K}$ do not have vertices of valence $1$.

\begin{lem}\label{lem:noextremal}
For any $H,K\le \theta(F)\le F(a,b,c)$ such that $H\cap K\ne 1$, we can find an element $g\in F(a,b,c)$ such that all the core graphs $\G_{H^g}$, $\G_{K^g}$, $\G_{H^g\cap K^g}$ for the conjugated subgroups $H^g$, $K^g$, $H^g\cap K^g$, respectively, do not have vertices of valence~$1$.
\end{lem}
\begin{proof}
The only vertices in $\G_H$, $\G_K$, $\G_{H\cap K}$ that may have valence $1$ are the basepoints $x_H$, $x_K$, $(x_H, x_K)$, respectively. If $(x_H,x_K)$ has valence $2$ or more, then its projections $x_H$ and $x_K$ also have valence $2$ or more, and there is nothing to prove.

Now suppose that the basepoint $p=(x_H,x_K)$ of $\G_{H\cap K}$ has valence $1$. Since $H\cap K\ne 1$, the graph $\G_{H\cap K}$ must have a closed circuit, and the vertex $p$ does not belong to it. Therefore there exists a vertex of valence $3$ in $\G_{H\cap K}$. Let $p{-}q$ be the shortest path in $\G_{H\cap K}$ to a valence $3$ vertex $q$, and let $w\in F(a,b,c)$ be the label on this path. Then projections $\Pi_H(p{-}q)$ and $\Pi_K(p{-}q)$ are immersed paths in $\G_H$, $\G_K$, respectively, with the same label $w$ on them. Since $q$ is a valence $3$ vertex in $\G_{H\cap K}$, the vertices $q_H=\Pi_H(q)$ and $q_K=\Pi_K(q)$ also have valence $3$. However, one may have other vertices of valence $3$ on the paths $\Pi_H(p{-}q)=x_H{-}q_H$ and $\Pi_K(p{-}q)=x_K{-}q_K$. 
Let $q'_H$ be a vertex on the path $x_H{-}q_H$ defined as follows: $q'_H=x_H$, if $x_H$ has valence greater than $1$ in $\G_H$, and $q'_H$ is the vertex of valence $3$ on the path $x_H{-}q_H$ with the least distance along this path from $x_H$, otherwise. Define $q'_K$ similarly in $\G_K$. 
Then conjugating by $g=w^{-1}$ inside $F(a,b,c)$ yields a triple of subgroups $H^g, K^g, H^g\cap K^g$ such that their core graphs $\G_{H^g}$, $\G_{K^g}$, $\G_{H^g\cap K^g}$ differ from $\G_{H}$, $\G_{K}$, $\G_{H\cap K}$ by moving their basepoint to vertices $q_H$, $q_K$, and $q$, respectively, and deleting the hanging trees $p{-}q$ from $\G_{H\cap K}$ and $x_H{-}q'_H$, $x_K{-}q'_K$ from $\G_H$, $\G_K$, respectively. (Vertices $q$, $q'_H$, $q'_K$ themselves are not deleted.) This gives us a triple of the core graphs $\G_{H^g}$, $\G_{K^g}$, $\G_{H^g\cap K^g}$ for the subgroups $H^g$, $K^g$, $H^g\cap K^g$ with no vertices of valence $1$. See Figure~\ref{fig:extremal} for an illustration.
\end{proof}

\begin{rem}\label{rem:novalence1}
Since $\rk H^g=\rk H$, $\rk K^g=\rk K$, $\rk H^g\cap K^g=\rk H\cap K$, for the purposes of this paper we may assume (and will do so from now on) without loss of generality that the groups $H$, $K$, $H\cap K$ have the core graphs which do not have vertices of valence~$1$. It may happen that after the procedure described in Lemma~\ref{lem:noextremal} the basepoints of $\G_H$, $\G_K$ and $\G_{H\cap K}$ all map to the vertex $v\in X$, instead of $u$, but that does not limit generality, since we may have chosen vertex $v\in X$ as the basepoint of $X$ from the very beginning. 
\end{rem}

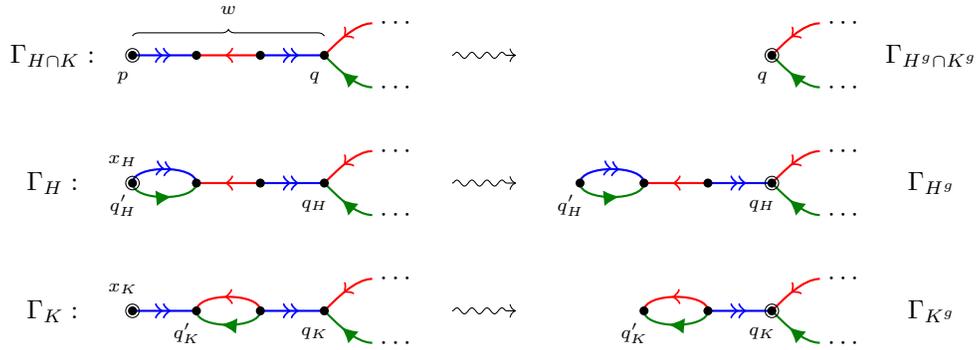
\begin{figure}[ht]
\begin{center}
\begin{tikzpicture}
\begin{scope}[thick,looseness=0.75,scale=0.85]
\draw[thin,decorate,decoration=brace] (1,-0.2) -- (4,-0.2);
\draw (2.5,0.2) node {$\scriptstyle w$};
\draw[->-ii=0.6,color=blue] (1,-0.5) to [out=0, in=180] (2,-0.5);
\draw[->-i=0.55,color=red] (3,-0.5) to [out=180, in=0] (2,-0.5);
\draw[->-ii=0.6,color=blue] (3,-0.5) to [out=0, in=180] (4,-0.5);
\draw[->-i=0.55,color=red] (4.75,0) to [out=-180, in=45] (4,-0.5);
\draw[->-=0.6,color=mygreen] (4.75,-1) to [out=180, in=-45] (4,-0.5);
\draw (5.15,0) node {$\cdots$};
\draw (5.15,-1) node {$\cdots$};
\fill (1,-0.5) circle (2pt);
\fill (2,-0.5) circle (2pt);
\fill (3,-0.5) circle (2pt);
\fill (4,-0.5) circle (2pt);
\draw[thin] (1,-0.5) circle (3pt);
\draw (0.85,-0.85) node {$\scriptstyle p$};
\draw (3.85,-0.85) node {$\scriptstyle q$};
\draw (-0.25,-0.5) node {$\G_{H\cap K}:$};
\draw [thin,->,decorate,decoration={snake,amplitude=.4mm,segment length=2mm,post length=1mm}] (6,-0.5)--(7,-0.5);
\begin{scope}[xshift=7cm]
\draw[->-i=0.55,color=red] (4.75,0) to [out=-180, in=45] (4,-0.5);
\draw[->-=0.6,color=mygreen] (4.75,-1) to [out=180, in=-45] (4,-0.5);
\draw (5.15,0) node {$\cdots$};
\draw (5.15,-1) node {$\cdots$};
\fill (4,-0.5) circle (2pt);
\draw[thin] (4,-0.5) circle (3pt);
\draw (3.85,-0.85) node {$\scriptstyle q$};
\draw (6.5,-0.5) node {$\G_{H^g\cap K^g}$};
\end{scope}

\begin{scope}[yshift=-2cm]
\draw[->-=0.6,color=mygreen] (1,-0.5) to [out=-90, in=-90] (2,-0.5);
\draw[->-ii=0.6,color=blue] (1,-0.5) to [out=90, in=90] (2,-0.5);
\draw[->-i=0.55,color=red] (3,-0.5) to [out=180, in=0] (2,-0.5);
\draw[->-ii=0.6,color=blue] (3,-0.5) to [out=0, in=180] (4,-0.5);
\draw[->-i=0.55,color=red] (4.75,0) to [out=-180, in=45] (4,-0.5);
\draw[->-=0.6,color=mygreen] (4.75,-1) to [out=180, in=-45] (4,-0.5);
\draw (5.15,0) node {$\cdots$};
\draw (5.15,-1) node {$\cdots$};
\fill (1,-0.5) circle (2pt);
\fill (2,-0.5) circle (2pt);
\fill (3,-0.5) circle (2pt);
\fill (4,-0.5) circle (2pt);
\draw[thin] (1,-0.5) circle (3pt);
\draw (0.85,-0.15) node {$\scriptstyle x_H$};
\draw (0.85,-0.85) node {$\scriptstyle q'_H$};
\draw (3.85,-0.85) node {$\scriptstyle q_H$};
\draw (-0.25,-0.5) node {$\G_{H}:$};
\draw [thin,->,decorate,decoration={snake,amplitude=.4mm,segment length=2mm,post length=1mm}] (6,-0.5)--(7,-0.5);
\begin{scope}[xshift=7cm]
\draw[->-=0.6,color=mygreen] (1,-0.5) to [out=-90, in=-90] (2,-0.5);
\draw[->-ii=0.6,color=blue] (1,-0.5) to [out=90, in=90] (2,-0.5);
\draw[->-i=0.55,color=red] (3,-0.5) to [out=180, in=0] (2,-0.5);
\draw[->-ii=0.6,color=blue] (3,-0.5) to [out=0, in=180] (4,-0.5);
\draw[->-i=0.55,color=red] (4.75,0) to [out=-180, in=45] (4,-0.5);
\draw[->-=0.6,color=mygreen] (4.75,-1) to [out=180, in=-45] (4,-0.5);
\draw (5.15,0) node {$\cdots$};
\draw (5.15,-1) node {$\cdots$};
\fill (1,-0.5) circle (2pt);
\fill (2,-0.5) circle (2pt);
\fill (3,-0.5) circle (2pt);
\fill (4,-0.5) circle (2pt);
\draw[thin] (4,-0.5) circle (3pt);
\draw (0.85,-0.85) node {$\scriptstyle q'_H$};
\draw (3.85,-0.85) node {$\scriptstyle q_H$};
\draw (6.5,-0.5) node {$\G_{H^g}$};
\end{scope}
\end{scope}

\begin{scope}[yshift=-4cm]
\draw[->-ii=0.6,color=blue] (1,-0.5) to [out=0, in=180] (2,-0.5);
\draw[->-=0.6,color=mygreen] (3,-0.5) to [out=-90, in=-90] (2,-0.5);
\draw[->-i=0.55,color=red] (3,-0.5) to [out=90, in=90] (2,-0.5);
\draw[->-ii=0.6,color=blue] (3,-0.5) to [out=0, in=180] (4,-0.5);
\draw[->-i=0.55,color=red] (4.75,0) to [out=-180, in=45] (4,-0.5);
\draw[->-=0.6,color=mygreen] (4.75,-1) to [out=180, in=-45] (4,-0.5);
\draw (5.15,0) node {$\cdots$};
\draw (5.15,-1) node {$\cdots$};
\fill (1,-0.5) circle (2pt);
\fill (2,-0.5) circle (2pt);
\fill (3,-0.5) circle (2pt);
\fill (4,-0.5) circle (2pt);
\draw[thin] (1,-0.5) circle (3pt);
\draw (0.85,-0.15) node {$\scriptstyle x_K$};
\draw (1.85,-0.85) node {$\scriptstyle q'_K$};
\draw (3.85,-0.85) node {$\scriptstyle q_K$};
\draw (-0.25,-0.5) node {$\G_{K}:$};
\draw [thin,->,decorate,decoration={snake,amplitude=.4mm,segment length=2mm,post length=1mm}] (6,-0.5)--(7,-0.5);
\begin{scope}[xshift=7cm]
\draw[->-=0.6,color=mygreen] (3,-0.5) to [out=-90, in=-90] (2,-0.5);
\draw[->-i=0.55,color=red] (3,-0.5) to [out=90, in=90] (2,-0.5);
\draw[->-ii=0.6,color=blue] (3,-0.5) to [out=0, in=180] (4,-0.5);
\draw[->-i=0.55,color=red] (4.75,0) to [out=-180, in=45] (4,-0.5);
\draw[->-=0.6,color=mygreen] (4.75,-1) to [out=180, in=-45] (4,-0.5);
\draw (5.15,0) node {$\cdots$};
\draw (5.15,-1) node {$\cdots$};
\fill (2,-0.5) circle (2pt);
\fill (3,-0.5) circle (2pt);
\fill (4,-0.5) circle (2pt);
\draw[thin] (4,-0.5) circle (3pt);
\draw (1.85,-0.85) node {$\scriptstyle q'_K$};
\draw (3.85,-0.85) node {$\scriptstyle q_K$};
\draw (6.5,-0.5) node {$\G_{K^g}$};
\end{scope}
\end{scope}

\end{scope}
\end{tikzpicture}
\end{center}
\caption{Eliminating vertices of valence $1$\label{fig:extremal} (for some generic $H$ and $K$).}
\end{figure}

\subsection{Join of subgroups and the topological pushout} \label{ssec:tp}
As was shown by Stallings~\cite{Sta}, the core graph $\G_{H\vee K}$ for the join of two subgroups is obtained by joining the core graphs for $\G_H$ and $\G_K$ at their respective basepoints and performing a sequence of identifications of edges with the same labels called \emph{foldings}:
\[
\G_H\vee \G_K\xrightarrow{\mathrm{\ foldings\ }}\G_{H\vee K}
\]
In general, the number of foldings required to produce $\G_{H\vee K}$ and the rank of $\G_{H\vee K}$ are hard to estimate directly from the information about $\G_H$, $\G_K$, without actually performing the required sequence of foldings. In~\cite{Ke2}, Kent works with an intermediate object, the \emph{topological pushout} $\T$ of $\G_H$ and $\G_K$, which fits into the diagram:
\[
\G_H\vee\G_K \xrightarrow{\mathrm{\ foldings\ }} \T \xrightarrow{\mathrm{\ foldings\ }} \G_{H\vee K}
\]
and whose rank is much easier to estimate than the rank of $\G_{H\vee K}$. 
Since the folding operation is surjective at the level of fundamental groups~\cite[Cor.~4.4]{Sta}, we also have 
\[
\rk\T\ge \rk \G_{H\vee K}. 
\]

$\T$ is defined as follows.

Let $x\in\G_H$ and $y\in\G_K$ be either two vertices or two edges of $\G_H$ and $\G_K$. The graph $\T$ is the quotient of the disjoint union $\G_H\sqcup\G_K$ by the equivalence relation generated by the following relation: $x\sim y$ if $x\in \Pi_H\big((\Pi_K|_{\G_{H\cap K}})^{-1}(y)\big)$ or $y\in \Pi_K\big((\Pi_H|_{\G_{H\cap K}})^{-1}(x)\big)$. In other words, $x\sim y$ if and only if there is an element $z$ of $\G_{H\cap K}$ such that $x$ and $y$ are the images under $\Pi_H$, $\Pi_K$, respectively, of $z$. Recall that, by construction, the vertices of $\G_{H\cap K}$ can be identified with a certain subset of $V(\G_H)\times V(\G_K)$ and the same is true for edges. Thus two elements (i.e.\ two vertices or two edges) $a,b$ of $\G_H\sqcup\G_K$ map to the same element in $\T$ if and only if there is a sequence of elements $(x_1,y_1),\dots, (x_n,y_n)$ in $\G_{H\cap K}$ with $x_i\in \G_H$, $y_i\in \G_K$ such that $a$ is either $x_1$ or $y_1$, $b$ is either $x_n$ or $y_n$ and for each $i$ either $x_i=x_{i+1}$ or $y_i=y_{i+1}$.

Equivalently, $\T$ can be obtained from the join of $\G_H$ and $\G_K$ over their respective basepoints $x_H$ and $x_K$, followed by a sequence of foldings along the edges of $\G_{H\cap K}$ only. I.e.\ we may choose a circuit (i.e.\ a closed path) $\gamma$ in $\G_{H\cap K}$ that starts at the basepoint $(x_H,x_K)$ and traverses each edge of $\G_{H\cap K}$ at least once, and perform a sequence of foldings, identifying $\Pi_H(z)\in\G_H$ with $\Pi_K(z)\in\G_K$ for $z$ running consecutively through all vertices and edges along $\gamma$. Since $\G_{H\cap K}$ is connected, a simple inductive argument shows that the result of this sequence of foldings is exactly the topological pushout $\T$ of $\G_H$ and $\G_K$.

Figure~\ref{fig:mainpic} 
shows the topological pushout for the groups $H,K$ of Example~\ref{ex:main}. We see that the topological pushout $\T$ may be different from $\G_{H\vee K}$. In Figure~\ref{fig:mainpic}, the graph $\G_{H\vee K}$ is equal to $X$ and it is obtained from $\T$ by identifying (folding) two $c$--edges.

\section{Proof of Theorem~\ref{thm:real}}\label{sec:real}

For the purposes of this section it will be convenient to assume that $F$ is a free group of countable rank so that we have a freedom to add new generators if necessary, without the need of explicitly embedding them into the free group of rank $2$. Also, for that purpose, we fix the wedge of countably many circles as the base CW complex for $F$.

We will call a tuple of values $(h,k;v,c)$ \emph{realizable}, if there exist finitely generated subgroups $H,K$ of $F$ with 
$\rk H=h$, $\rk K=k$, $\rk (H\vee K)=v$, and $\rk (H\cap K)=c$.
If the values $h,k$ (and sometimes also $v$) are clear from the context, we will also call the tuple $(v,c)$ (respectively, the number $c$) \emph{realizable}, and say that it belongs to \emph{page} $(h,k)$.

Excluding trivial cases, we may assume that $\rk H\ge 2$ and $\rk K\ge 2$, so that $\rk (H\vee K)\ge 2$. The upper bound for $\rk (H\vee K)$ is obviously $h+k$. On the other hand, the limits for $\rk (H\cap K)$ are $0$ and $(h-1)(k-1)+1$, as is stipulated by the Friedman--Mineyev theorem. 

It turns out that the set of all known realizable values $(h,k;v,c)$ can be described for any fixed $(h,k)$ by a finite sequence of nonnegative integers $(a_i)$, such that for any given $h,k,v$ all (known) realizable values of $c$ are described as the range $0\le c\le a_i$, where $i=h+k-v$: 

\begin{thm11} 
Let $F$ be a free group and let integers $h,k,c,v$ satisfy $2\le h\le k$, $2\le v\le h+k$, $0\le c\le (h-1)(k-1)+1$. Define a sequence $(a_i)$ as follows:
\begin{align*}
a_0&=0\,;\\
a_i&=\Big\lfloor \frac {i^2}{4}\Big\rfloor +1,\quad \text{for\quad $i=1,\dots, 2(h-1)\,;$}\\
a_i&=(h-1)(i-h+1)+1,\quad \text{for\quad $i=2(h-1),\dots,h+k-2$.}
\end{align*}
If we denote $i=h+k-v$, then for any $c\le a_i$ there exist subgroups $H,K\le F$ such that\/ $\rk(H)=h$, $\rk(K)=k$, $\rk(H\cap K)=c$, and\/ $\rk(H\vee K)=v$.
\end{thm11}

\begin{example} 
The diagram in Figure~\ref{fig:57} shows the realizable values from Theorem~\ref{thm:real} for $h=\rk H=5$, $k=\rk K=7$. They correspond to the sequence 
\[
(a_i)=(0,1,2,3,5,7,10,13,17,21,25).
\]

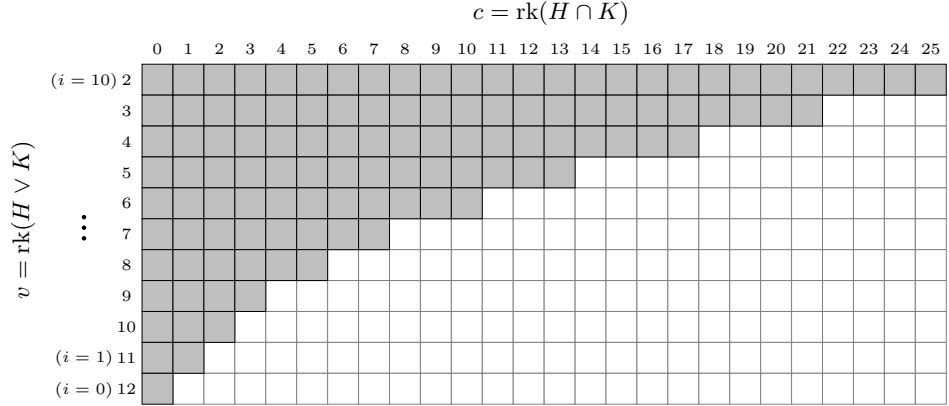
\begin{figure}[ht]
\begin{center}
\begin{tikzpicture}[scale=0.9]
{\small
\draw (6.5cm,0.75cm) node {$c=\rk (H\cap K)$};
\draw (-1.3cm,-2.3cm) node [rotate=90] {$v=\rk (H\vee K)$};
\draw (-0.4cm,-4.82cm) node {{\tiny $(i=0)$}};
\draw (-0.4cm,-4.35cm) node {{\tiny $(i=1)$}};
\draw (-0.4cm,-0.25cm) node {{\tiny $(i=10)$}};
\draw (-0.4cm,-2.30cm) node {{\Huge $\vdots$}};
}
{\tiny
\Ylinecolour{black!50}
\Yfillcolour{white}
\tgyoung(0cm,0cm,%
:::::::::::::::::::::::::::,%
:::::::::::::::::::::::::::,%
:::::::::::::::::::::::;;;;,%
:::::::::::::::::::;;;;;;;;,%
:::::::::::::::;;;;;;;;;;;;,%
::::::::::::;;;;;;;;;;;;;;;,%
:::::::::;;;;;;;;;;;;;;;;;;,%
:::::::;;;;;;;;;;;;;;;;;;;;,%
:::::;;;;;;;;;;;;;;;;;;;;;;,%
::::;;;;;;;;;;;;;;;;;;;;;;;,%
:::;;;;;;;;;;;;;;;;;;;;;;;;,%
::;;;;;;;;;;;;;;;;;;;;;;;;;%
)
\Ylinecolour{black}
\Yfillcolour{black!25}
\tgyoung(0cm,0cm,::0:1:2:3:4:5:6:7:8:9:<10>:<{1}1>:<12>:<13>:<14>:<15>:<16>:<17>:<18>:<19>:<20>:<21>:<{2}2>:<23>:<24>:<25>,%
:2;;;;;;;;;;;;;;;;;;;;;;;;;;,%
:3;;;;;;;;;;;;;;;;;;;;;;,%
:4;;;;;;;;;;;;;;;;;;,%
:5;;;;;;;;;;;;;;,%
:6;;;;;;;;;;;,%
:7;;;;;;;;,%
:8;;;;;;,%
:9;;;;,%
:<10>;;;,%
:<{1}1>;;,%
:<12>;%
)
}
\end{tikzpicture}
\caption{The realizable values for $(h,k)=(5,7)$ from Theorem~\ref{thm:real}.\label{fig:57}}
\end{center}
\end{figure}

\end{example}

We see that the sequence $(a_i)$ from the Theorem~\ref{thm:real} is a union of a discrete quadratic function and a linear function, with the linear part present only when $h<k$. The value $a_0=0$ reflects the fact that if $\rk (H\vee K)=\rk H + \rk K$ then $\rk (H\cap K)=0$. (This is a consequence of the property of finitely generated free groups being Hopfian.) 
On the other hand, the value $v=2$ corresponds to $i=h+k-2$, and $a_{h+k-2}$ equals $(h-1)(k-1)+1$. This reflects the fact that all possible values for $\rk (H\cap K)=0,\dots,(h-1)(k-1)+1$ are realizable when $\rk (H\vee K)=2$, as was shown by Kent in~\cite{Ke1}.

\begin{proof}[Proof of Theorem~\ref{thm:real}]
In what follows, we will call a finite sequence $(a_i)$, ($i=0,\dots,n$), \emph{greater} than a sequence $(b_i)$, ($i=0,\dots,n$), if for each $i$, we have $a_i\ge b_i$. In this case we also say that the sequence $(b_i)$ is \emph{smaller} than the sequence $(a_i)$.

We will obtain the required set of realizable values for page $(h',k')$ inductively from the realizable values for page $(h,k)$ with $h\le h'$, $k\le k'$, by using the following operations:

Ia. Adding a new generator to $H$. This operation copies all realizable values from page $(h,k)$ to page $(h+1,k)$ as follows:
\[
(h,k;v,c)\longmapsto (h+1,k;v+1,c)
\]

Ib. Adding a new generator to $K$. This operation copies all realizable values from page $(h,k)$ to page $(h,k+1)$ as follows:
\[
(h,k;v,c)\longmapsto (h,k+1;v+1,c)
\]

II. Adding the same new generator to both $H$ and $K$. We get:
\[
(h,k;v,c)\longmapsto (h+1,k+1;v+1,c+1)
\]

III. Populating the first row of any page $(h,k)$ with values 
\[
v=2,\quad c=0,\,\dots,\,(h-1)(k-1)+1
\] 
corresponding to the explicit examples produced by Kent~\cite{Ke1}. 

To prove the claimed effect on ranks under operations Ia, Ib and II, we notice that adding a new generator to a subgroup $H$ amounts to attaching the loop corresponding to this generator to the core graph $\G_H$, at its basepoint. (We can do that in view of the assumption in the opening paragraph of the current section.) This makes the effect of operations Ia and Ib obvious, while for operation II we recall the construction of the core graph for $H\cap K$ from Section~\ref{sec:graphs}. It is clear that if we attach a loop labeled with the same new generator to both core graphs $\G_H$, $\G_K$ at their respective basepoints, then the core graph of their intersection $\G_{H\cap K}$ also gets the loop labeled with the same generator attached to its basepoint.

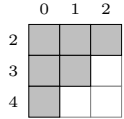
\begin{figure}
\begin{center}
\begin{tikzpicture}[scale=0.9]
{\tiny
\Ylinecolour{black!50}
\Yfillcolour{white}
\tgyoung(0cm,0cm,%
::::,%
::::,%
:::;,%
::;;%
)
\Ylinecolour{black}
\Yfillcolour{black!25}
\tgyoung(0cm,0cm,::0:1:2,%
:2;;;,%
:3;;,%
:4;%
)
}
\end{tikzpicture}
\caption{The realizable values for $(h,k)=(2,2)$.\label{fig:22}}
\end{center}
\end{figure}

\newcommand{\tincdots}{\scriptstyle\mathbf{\cdots}}
\newcommand{\tinvdots}{\scriptstyle\mathbf{\cdots}}
\newcommand{\tinddots}{\scriptstyle\mathbf{\cdots}}
\newcommand{\tindots}{\tincdots}
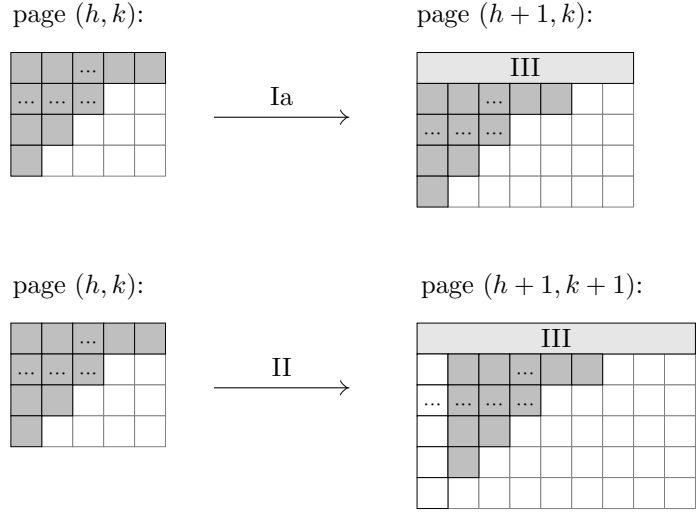
\begin{figure}
\begin{center}
\begin{tikzpicture}[scale=0.9]
\Ystdtext0
\Ylinecolour{black!50}
\Yfillcolour{white}
\tgyoung(0cm,0cm,%
:::::,%
:::;;,%
::;;;,%
:;;;;%
)
\Ylinecolour{black}
\Yfillcolour{black!25}
\tgyoung(0cm,0cm,%
;;;\tincdots;;,%
;\tinvdots;\tinvdots;\tinddots,%
;;,%
;%
)
\draw[mapping=0.99] (3,-0.5)--node[above=1pt] {Ia}(5,-0.5);
\Ylinecolour{black!50}
\Yfillcolour{white}
\tgyoung(6cm,0cm,%
:::::::,%
:::::;;,%
:::;;;;,%
::;;;;;,%
:;;;;;;%
)
\Ylinecolour{black}
\Yfillcolour{black!25}
\tgyoung(6cm,0cm,%
!<\Yfillcolour{black!10}>%
_7<\mathrm{III}>,%
!<\Yfillcolour{black!25}>%
;;;\tincdots;;,%
;\tinvdots;\tinvdots;\tinddots,%
;;,%
;%
)
\draw (1,1) node {page $(h,k)$:};
\draw (7.3,1) node {page $(h+1,k)$:};

\begin{scope}[yshift=-4cm]
\Ystdtext0
\Ylinecolour{black!50}
\Yfillcolour{white}
\tgyoung(0cm,0cm,%
:::::,%
:::;;,%
::;;;,%
:;;;;%
)
\Ylinecolour{black}
\Yfillcolour{black!25}
\tgyoung(0cm,0cm,%
;;;\tincdots;;,%
;\tinvdots;\tinvdots;\tinddots,%
;;,%
;%
)
\draw[mapping=0.99] (3,-0.5)--node[above=1pt] {II}(5,-0.5);
\Ylinecolour{black!50}
\Yfillcolour{white}
\tgyoung(6cm,0cm,%
:::::::::,%
::::::;;;,%
::::;;;;;,%
:::;;;;;;,%
::;;;;;;;,%
:;;;;;;;;%
)
\Ylinecolour{black}
\Yfillcolour{black!25}
\tgyoung(6cm,0cm,%
!<\Yfillcolour{black!10}>%
_9<\mathrm{III}>,%
!<\Yfillcolour{white}>;%
!<\Yfillcolour{black!25}>%
;;;\tincdots;;,%
!<\Yfillcolour{white}>;\tinvdots%
!<\Yfillcolour{black!25}>%
;\tinvdots;\tinvdots;\tinddots,%
!<\Yfillcolour{white}>;%
!<\Yfillcolour{black!25}>%
;;,%
!<\Yfillcolour{white}>;%
!<\Yfillcolour{black!25}>%
;,%
!<\Yfillcolour{white}>;%
)
\draw (1,1) node {page $(h,k)$:};
\draw (7.7,1) node {page $(h+1,k+1)$:};
\end{scope}
\end{tikzpicture}
\caption{The effect of operations Ia and II followed by III. The horizontal axes of the tables correspond to the values of $c=\rk (H\cap K)$, the vertical axes to the values of $v=\rk (H\vee K)$. Operations Ia and II are shown in dark gray, operation III in light gray. Operation Ib has the same effect as Ia, but for the target page $(h,k+1)$.\label{fig:23}}
\end{center}
\end{figure}

We start with page $(h,k)=(2,2)$, which corresponds to the sequence $(a_n)=(0,1,2)$, (as was shown in~\cite[p.~307]{Ke2}, see Figure~\ref{fig:22}), and determine a sequence of operations Ia, Ib, II, III that leads to the greatest sequence $(a_n)$, for the required values $(h,k)$. 

We first show that if $h<k$ then applying (Ia + III) followed by (Ib + III) produces a greater sequence $(a_n)$ of realizable $c$--values for page $(h+1,k+1)$ than applying first (Ib + III) and then (Ia + III). Indeed, let $(a_n)$ be the sequence for $(h,k)$ with $n$ ranging from $0$ to $h+k-2$. Denote $(a_n')$ the result of applying (Ia + III) to $(a_n)$ and $(a_n'')$ the result of applying (Ib + III) to $(a_n')$. Similarly, denote $(b_n')$ the result of applying (Ib + III) to $(a_n)$ and $(b_n'')$ the result of applying (Ia + III) to $(b_n')$:
\begin{gather*}
(a_n)\xrightarrow{\text{Ia+III}}(a_n')\xrightarrow{\text{Ib+III}}(a_n''),\\
(a_n)\xrightarrow{\text{Ib+III}}(b_n')\xrightarrow{\text{Ia+III}}(b_n'').
\end{gather*}
We want to show that $a_n''\ge b_n''$ for all $n$. Since operation Ia copies all values of $(a_i)$ to the new page, we have: $a_i'=a_i$ if $i<(h+1)+k-2$, and since operation III populates the first row (corresponding to the last value of $a_i$) with a sequence of $h(k-1)+1$ values, we get $a'_{(h+1)+k-2}=h(k-1)+1$. Similarly, $a_i''=a_i$ for $i<(h+1)+(k+1)-3$, $a''_{(h+1)+(k+1)-3}=h(k-1)+1$, and $a''_{(h+1)+(k+1)-2}=hk+1$. Performing these operations in the opposite order, i.e.\ applying (Ib + III) to $(a_n)$ first followed by (Ia + III), we get in a similar fashion: $b''_i=a_i$ for $i<(h+1)+(k+1)-3$, $b''_{(h+1)+(k+1)-3}=hk-k+1$, and $b''_{(h+1)+(k+1)-2}=hk+1$. The sequences $(a_n'')$ and $(b_n'')$ agree for all values of $n$ except the penultimate one, $n=(h+1)+(k+1)-3$. Comparing them and taking into account that $h<k$, we see that $hk-h+1>hk-k+1$, i.e.\ $a_n''\ge b_n''$ for all $n$. This means that if $h<k$, applying (Ia + III) followed by (Ib + III) produces a greater sequence $(a_n)$ than if applying these operations in the opposite order. 

The above formulas also show that in the case when $h=k$ the result of applying operations (Ia + III) and (Ib + III) does not depend on their order.

Now let's examine operation (II + III). It copies realizable values from page $(h,k)$ to page $(h+1,k+1)$ as shown in the lower part of Figure~\ref{fig:23}. Let $(a_n)$ and $(a_n''')$ be the corresponding sequences of realizable values for pages $(h,k)$ and $(h+1,k+1)$, respectively, and let $(a''_n)$ be the sequence obtained from $(a_n)$ by the composition of operations (Ia + III) and (Ib + III), as before. We claim that $(a_n''')$ is smaller than $(a_n'')$. Indeed, we saw in the previous paragraph that $a_i''=a_i$ for $i<h+k-1$, $a''_{h+k-1}=hk-h+1$ and $a''_{h+k}=hk+1$. From Figure~\ref{fig:23} we see that $a'''_{i}=a_{i-1}+1$ for $i=1,\dots,h+k-1$ and $a'''_{h+k}=hk+1$. Thus to prove that $a''_i\ge a'''_i$ for all $i$, we 
claim the following:
\begin{enumerate}
\item [(i)] $a_{i}\ge a_{i-1}+1$ for $i=1,\dots,h+k-2$, and
\item [(ii)] $hk-h+1\ge a_{h+k-2}+1$.
\end{enumerate}
We will prove these using the following fact: \emph{for any page $(h,k)$ the last value of the sequence $(a_n)$, i.e.\ the term $a_{h+k-2}$, equals $(h-1)(k-1)+1$}, which is the content of operation III. (Recall that operation III is applied every time we apply Ia, Ib or II.) Thus, inequality (ii) is established: $a_{h+k-2}=(h-1)(k-1)+1\le hk-h$ since $k\ge2$. To prove (i) we observe that all operations Ia, Ib, II preserve the difference between consecutive elements $a_{j}-a_{j-1}$, so all that needs to be proved is that the top value $a_{h+k-2}=(h-1)(k-1)+1$ is always at least $1$ bigger than the previous value of $a_{h+k-3}$. Assuming that (i) holds true for page $(h,k)$, we examine how it changes under the application of operations (Ia + III), (Ib + III) and (II + III). In the first case, operation Ia makes the next-to-last value of the sequence $(a_n')$ to be $(h-1)(k-1)+1$, while the last one is $h(k-1)+1$, with the difference $k-1$ between the two. In the second case, arguing in a similar fashion, we get that the difference between the last two values of $(b_n')$ equals $h-1$. And in the case of operation (II + III), we get that the difference equals $(hk+1) - \big((h-1)(k-1)+2\big)=h+k-2$. We see that in all three cases this difference is at least $1$, which proves claim (i). Thus operation (II + III) creates a sequence $(a_n''')$ which is smaller than the sequence $(a_n'')$ created by (Ia + III) followed by (Ib + III).

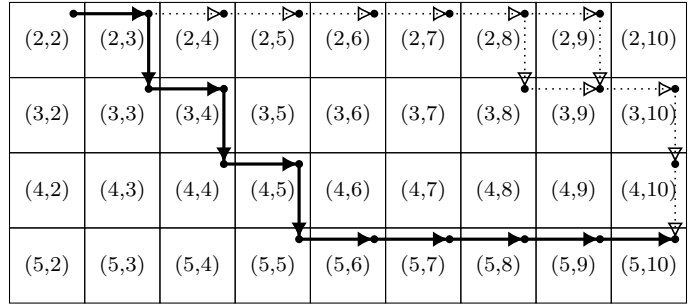
\begin{figure}
\begin{tikzpicture}
\foreach \x in {2,...,10}
\foreach \y in {2,...,5}
{
\draw (\x,\y) +(-.5,-.5) rectangle ++(.5,.5);
\draw (\x,7-\y) node{\footnotesize(\y,\x)};
}

\fill (2,5)+(0.35,0.35) circle (1.5pt);
\fill (3,5)+(0.35,0.35) circle (1.5pt);
\fill (3,4)+(0.35,0.35) circle (1.5pt);
\fill (4,4)+(0.35,0.35) circle (1.5pt);
\fill (4,3)+(0.35,0.35) circle (1.5pt);
\fill (5,3)+(0.35,0.35) circle (1.5pt);
\fill (5,2)+(0.35,0.35) circle (1.5pt);
\fill (6,2)+(0.35,0.35) circle (1.5pt);
\fill (7,2)+(0.35,0.35) circle (1.5pt);
\fill (8,2)+(0.35,0.35) circle (1.5pt);
\fill (9,2)+(0.35,0.35) circle (1.5pt);
\fill (10,2)+(0.35,0.35) circle (1.5pt);
\begin{scope}[very thick]
\draw [->-=0.99] (2.35,5.35)--(3.35,5.35);
\draw [->-=0.99] (3.35,5.35)--(3.35,4.35);
\draw [->-=0.99] (3.35,4.35)--(4.35,4.35);
\draw [->-=0.99] (4.35,4.35)--(4.35,3.35);
\draw [->-=0.99] (4.35,3.35)--(5.35,3.35);
\draw [->-=0.99] (5.35,3.35)--(5.35,2.35);
\draw [->-=0.99] (5.35,2.35)--(6.35,2.35);
\draw [->-=0.99] (6.35,2.35)--(7.35,2.35);
\draw [->-=0.99] (7.35,2.35)--(8.35,2.35);
\draw [->-=0.99] (8.35,2.35)--(9.35,2.35);
\draw [->-=0.99] (9.35,2.35)--(10.35,2.35);
\end{scope}

\begin{scope}[dotted, semithick]
\draw [->-o=0.99] (3.35,5.35)--(4.35,5.35);
\draw [->-o=0.99] (4.35,5.35)--(5.35,5.35);
\draw [->-o=0.99] (5.35,5.35)--(6.35,5.35);
\draw [->-o=0.99] (6.35,5.35)--(7.35,5.35);
\draw [->-o=0.99] (7.35,5.35)--(8.35,5.35);
\draw [->-o=0.99] (8.35,5.35)--(8.35,4.35);
\draw [->-o=0.99] (8.35,5.35)--(9.35,5.35);
\draw [->-o=0.99] (9.35,5.35)--(9.35,4.35);
\draw [->-o=0.99] (8.35,4.35)--(9.35,4.35);
\draw [->-o=0.99] (9.35,4.35)--(10.35,4.35);
\draw [->-o=0.99] (10.35,4.35)--(10.35,3.35);
\draw [->-o=0.99] (10.35,3.35)--(10.35,2.35);
\end{scope}

\fill[black] (4,5)+(0.35,0.35) circle (1.5pt);
\fill[black] (5,5)+(0.35,0.35) circle (1.5pt);
\fill[black] (6,5)+(0.35,0.35) circle (1.5pt);
\fill[black] (7,5)+(0.35,0.35) circle (1.5pt);
\fill[black] (8,5)+(0.35,0.35) circle (1.5pt);
\fill[black] (8,4)+(0.35,0.35) circle (1.5pt);
\fill[black] (9,5)+(0.35,0.35) circle (1.5pt);
\fill[black] (9,4)+(0.35,0.35) circle (1.5pt);
\fill[black] (10,4)+(0.35,0.35) circle (1.5pt);
\fill[black] (10,3)+(0.35,0.35) circle (1.5pt);
\end{tikzpicture}

\caption{The ways to reach page $(h,k)=(5,10)$ from page $(h_0,k_0)=(2,2)$ are given by broken lines composed of horizontal and vertical arrows, which correspond to operations (Ib+III) and (Ia+III), respectively. The thick line corresponds to an optimal sequence of operations, giving the maximal sequence $(a_n)$.
\label{fig:pages}}
\end{figure}

The above analysis shows that to obtain the greatest sequence $(a_n)$ of realizable values for page $(h,k)$, $h\le k$, one can discard operation (II + III) completely and apply only operations (Ia + III) and (Ib + III), starting with page $(h_0,k_0)=(2,2)$. All the ways to get from page $(2,2)$ to page $(h,k)$ by applying the said operations can be encoded by broken lines running in a rectangular table from entry $(2,2)$ to entry $(h,k)$ with horizontal and vertical segments corresponding to operations (Ib + III) and (Ia + III), respectively, see Figure~\ref{fig:pages}. Every time when operation (Ib + III) followed by (Ia + III) is applied to a page $(h',k')$ with $h'<k'$, we can interchange the order of these operations thus producing a bigger sequence $(a_n)$. By repeatedly doing this interchange, we obtain an optimal sequence which can be described as follows: 
\begin{quote}
\emph{Start with $(h_0,k_0)=(2,2)$. Alternate operations (Ib + III) and (Ia + III) to reach page $(h,h)$. If $h<k$, keep applying (Ib + III) to reach page $(h,k)$}:
\end{quote}
\begin{multline*}
(2,2)\xrightarrow{\text{Ib+III}}(2,3)\xrightarrow{\text{Ia+III}}(3,3)\xrightarrow{\text{Ib+III}}\dots
\xrightarrow{\text{Ia+III}}(h,h)\xrightarrow{\text{Ib+III}}\\
\xrightarrow{\text{Ib+III}}(h,h+1)\xrightarrow{\text{Ib+III}}(h,h+2)\xrightarrow{\text{Ib+III}}\dots
\xrightarrow{\text{Ib+III}}(h,k).
\end{multline*}

Note that $a_0=0,a_1=1,a_2=2$ for $(h_0,k_0)=(2,2)$, and each of operations (Ia + III), (Ib + III) augments the existing sequence $a_0,\dots,a_n$ with a new value $a_{n+1}$, which is computed according to operation III as follows: 
\begin{align*}
(\ell,\ell)&\xrightarrow{\text{Ib+III}}(\ell,\ell+1): &a_{n+1}=\ell(\ell-1)+1,\quad &\text{with }n=2(\ell-1);\\
(\ell,\ell+1)&\xrightarrow{\text{Ia+III}}(\ell+1,\ell+1): &a_{n+1}=\ell^2+1,\quad &\text{with }n=2\ell-1,
\end{align*}
for $\ell=2,\dots,h-1$, and 
\[
(h,h+j)\xrightarrow{\text{Ib+III}}(h,h+j+1):\quad a_{n+1}=(h-1)(h+j)+1,\,\text{with }n=2(h-1)+j,
\]
for $j=0,\dots,k-h-1$, if $h<k$.

Now we observe that the values $a_{n+1}$ for $0\le n\le 2h-3$ can be written concisely as $a_{n+1}=\big\lfloor\big(\frac{n+1}2\big)^2\big\rfloor+1$. Indeed, if $n=2(\ell-1)$, then $\big\lfloor\big(\frac{n+1}2\big)^2\big\rfloor=\big\lfloor(\ell-\frac12)^2\big\rfloor=\big\lfloor\ell^2-\ell+\frac14\big\rfloor=\ell(\ell-1)$, and if $n=2\ell-1$, then $\big\lfloor\big(\frac{n+1}2\big)^2\big\rfloor=\ell^2$. This proves that the above sequence of operations produces the sequence $a_i$ described in the Theorem, which finishes the proof.
\end{proof}

Now we have counterexamples to GTC for all $m\ge 5$.
\begin{cor}
Guzman's ``Group-Theoretic Conjecture'' does not hold for any $m\ge 5$.
\end{cor}
\begin{proof}
Theorem~\ref{thm:real} guarantees the existence of subgroups $H,K\le F$ such that $h=\rk (H)=m$, $k=\rk (K)=m$, $c=\rk(H\cap K)=m$ and $v=\rk(H\vee K)=m+1$, if $m\ge 5$. Indeed, $i=h+k-v=m-1$ in this case, which is less than $2(h-1)=2(m-1)$. So the value of $a_i$ in Theorem~\ref{thm:real} is equal to $\big\lfloor \frac{i^2}4\big\rfloor+1 = \big\lfloor\frac{(m-1)^2}4\big\rfloor+1$, which is bigger than or equal to $c=m$ for all $m\ge 5$. See Figure~\ref{fig:intro} for an illustration of the case $m=6$, which shows the existence of rank $6$ subgroups $H,K$ with $\rk(H\vee K)=7$ and $\rk(H\cap K)=6,7$.
\end{proof}

\section{The structure of the topological pushout}\label{sec:dicks}

In this section we study the combinatorial structure of the topological pushout using graphs introduced by Dicks in~\cite{Di}. This will allow us to obtain an upper bound on the rank of the topological pushout of two core graphs and hence on the rank of the join of the corresponding subgroups.

\subsection{The Dicks graphs}
Let $\G_H$, $\G_K$ be the core graphs for subgroups $H,K\le F$ and the core graph $\G_{H\cap K}$ is constructed as the pullback of graph immersions, as in Section~\ref{sec:graphs}: 
\[
\begin{CD}
\G_{H\cap K} @>\Pi_K>> \G_K\\
@VV\Pi_H V               @VVp_K V\\
\G_H         @>p_H>>       X
\end{CD}
\]
Each element $z$ (a vertex or an edge) of $\G_H$ and $\G_K$ inherits its \emph{type} from the mapping to $X$, that is, an element of $V(X)\sqcup E(X)=\{u,v,a,b,c\}$ to which $z$ maps. 

We are now going to define five bipartite undirected graphs $\Omega_u, \Omega_v, \Omega_a, \Omega_b, \Omega_c$, adapting the construction of~\cite{Di} to our needs.

First, we define $\Omega_u$ as follows.
\begin{align*}
V(\Omega_u)&=\{z\in V(\G_H) \mid p_H(z)=u\} \sqcup \{z'\in V(\G_K) \mid p_K(z')=u\},\\
E(\Omega_u)&=\{(z,z')\in V(\G_{H\cap K})\mid p_H(z)=u \text{ and } p_K(z')=u\}.
\end{align*}
In other words, two vertices $z\in V(\G_H)$, $z'\in V(\G_K)$ from $V(\Omega_u)$ are connected with a single undirected edge if the vertex $(z,z')$ of $G_{H\cap K}$ actually belongs to $\G_{H\cap K}$.

The graph $\Omega_v$ is defined analogously to $\Omega_u$ with the obvious modification ($u\rightsquigarrow v$). Denote also $\Omega=\Omega_u\sqcup\Omega_v$.

We define $\Omega_a$ similarly, by dealing with edges instead of vertices:
\begin{align*}
V(\Omega_a)&=\{e\in E(\G_H) \mid p_H(e)=a\} \sqcup \{e'\in E(\G_K) \mid p_K(e')=a\},\\
E(\Omega_a)&=\{(e,e')\in E(\G_{H\cap K})\mid p_H(e)=a \text{ and } p_K(e')=a\}
\end{align*}

The graphs $\Omega_b$, $\Omega_c$ are defined analogously with the obvious modifications. 

The bipartite structure on the defined graphs is given by grouping all vertices/edges of graph $\G_H$ in one part, and those of $\G_K$ in the other.

Figure~\ref{fig:dicks-uvabc} shows the graphs $\Omega_u, \Omega_v, \Omega_a, \Omega_b, \Omega_c$ for the core graphs of the subgroups $H,K$ from Example~\ref{ex:main}. 

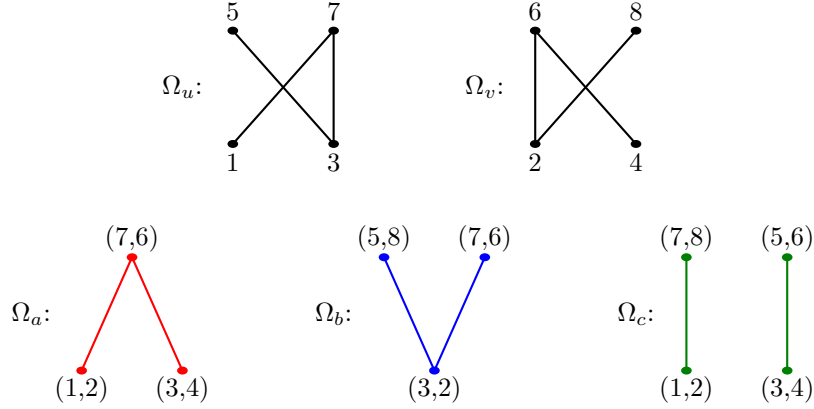
\begin{figure}[ht]
\begin{center}
\begin{tikzpicture}[yscale=0.75]
\begin{scope}[thick,scale=0.67]
\fill (0,0) circle (3pt);
\fill (2,0) circle (3pt);
\fill (0,3) circle (3pt);
\fill (2,3) circle (3pt);
\draw (0,0)--(2,3);
\draw (2,0)--(0,3);
\draw (2,0)--(2,3);
\draw (0,-0.5) node {$1$};
\draw (2,-0.5) node {$3$};
\draw (0,3.5) node {$5$};
\draw (2,3.5) node {$7$};
\draw (-1,1.5) node {$\Omega_u$:};

\begin{scope}[xshift=6cm]
\fill (0,0) circle (3pt);
\fill (2,0) circle (3pt);
\fill (0,3) circle (3pt);
\fill (2,3) circle (3pt);
\draw (0,0)--(2,3);
\draw (2,0)--(0,3);
\draw (0,0)--(0,3);
\draw (0,-0.5) node {$2$};
\draw (2,-0.5) node {$4$};
\draw (0,3.5) node {$6$};
\draw (2,3.5) node {$8$};
\draw (-1,1.5) node {$\Omega_v$:};
\end{scope}

\begin{scope}[xshift=-3cm,yshift=-6cm]
\begin{scope}[red]
\fill (0,0) circle (3pt);
\fill (2,0) circle (3pt);
\fill (1,3) circle (3pt);
\draw (0,0)--(1,3);
\draw (2,0)--(1,3);
\end{scope}
\draw (0,-0.5) node {(1,2)}; 
\draw (2,-0.5) node {(3,4)}; 
\draw (1,3.5) node {(7,6)}; 
\draw (-1,1.5) node {$\Omega_a$:};
\end{scope}

\begin{scope}[xshift=3cm,yshift=-6cm]
\begin{scope}[blue]
\fill (1,0) circle (3pt);
\fill (0,3) circle (3pt);
\fill (2,3) circle (3pt);
\draw (1,0)--(0,3);
\draw (2,3)--(1,0);
\end{scope}
\draw (1,-0.5) node {(3,2)}; 
\draw (0,3.5) node {(5,8)}; 
\draw (2,3.5) node {(7,6)}; 
\draw (-1,1.5) node {$\Omega_b$:};
\end{scope}

\begin{scope}[xshift=9cm,yshift=-6cm]
\begin{scope}[mygreen]
\fill (0,0) circle (3pt);
\fill (2,0) circle (3pt);
\fill (0,3) circle (3pt);
\fill (2,3) circle (3pt);
\draw (0,0)--(0,3);
\draw (2,0)--(2,3);
\end{scope}
\draw (0,-0.5) node {(1,2)}; 
\draw (0,3.5) node {(7,8)}; 
\draw (2,-0.5) node {(3,4)}; 
\draw (2,3.5) node {(5,6)}; 
\draw (-1,1.5) node {$\Omega_c$:};
\end{scope}
\end{scope}
\end{tikzpicture}
\end{center}
\caption{The Dicks graphs $\Omega_u, \Omega_v, \Omega_a, \Omega_b, \Omega_c$ for the subgroups in Example~\ref{ex:main}. A pair $(x,y)$ denotes the unique $a$--, $b$--, or $c$--edge of $E(\G_H)\sqcup E(\G_K)$ originating at a vertex $x$ and terminating at a vertex $y$.
\label{fig:dicks-uvabc}}
\end{figure}

Notice that the operations of taking the origin and the terminus of an edge induce embeddings $\tilde o$, $\tilde t$ of the graph $\Omega_a\sqcup \Omega_b\sqcup\Omega_c$ into $\Omega_u, \Omega_v$, respectively. Let's show that $\tilde o|_{\Omega_a}\colon\Omega_a\to \Omega_u$ is an embedding. If $e,e'\in V(\Omega_a)$, with $e\ne e'$, then they correspond to $a$--edges of $\G_H\sqcup\G_K$. If $e\in E(\G_H)$, $e'\in E(\G_K)$, their origins are different. If they both belong to the same graph then their origins are also different, since $\G_H$ and $\G_K$ are folded, i.e.\ at every vertex of $\G_H$, $\G_K$ there is at most one $a$--edge having this vertex as the origin. This proves that the map $\tilde o|_{\Omega_a}\colon V(\Omega_a)\to V(\Omega_u)$ is injective. Now, if $e,e'\in V(\Omega_a)$ are connected with an edge, this means that there is an $a$--edge $(e,e')$ in $E(\G_{H\cap K})$ whose projections under $\Pi_H, \Pi_K$ are $e,e'$. In particular, the origin of $(e,e')$ projects to the origins of $e,e'$, and we see that the images of $e,e'$ in $V(\Omega_u)$ are connected with an edge as well. This proves that $\tilde o|_{\Omega_a}\colon\Omega_a\to\Omega_u$ is an injective graph homomorphism, i.e.\ is an isomorphism onto its image. Denote
\begin{align}\label{eq:dicks2}
\Omega_{u,a}&=\im(\Omega_a\xhookrightarrow{\quad\tilde o\quad}\Omega_u), & \Omega_{v,a}&=\im(\Omega_a\xhookrightarrow{\quad\tilde t\quad}\Omega_v), \nonumber\\
\Omega_{u,b}&=\im(\Omega_b\xhookrightarrow{\quad\tilde o\quad}\Omega_u), & \Omega_{v,b}&=\im(\Omega_b\xhookrightarrow{\quad\tilde t\quad}\Omega_v), \\
\Omega_{u,c}&=\im(\Omega_c\xhookrightarrow{\quad\tilde o\quad}\Omega_u), & \Omega_{v,c}&=\im(\Omega_c\xhookrightarrow{\quad\tilde t\quad}\Omega_v) \nonumber
\end{align}
the images of these embeddings, and set
\begin{align*}
\Omega_{u,ab}&=\Omega_{u,a}\cap\Omega_{u,b}, & \Omega_{v,ab}&=\Omega_{v,a}\cap\Omega_{v,b}, & \Omega_{ab}&=\Omega_{u,ab}\sqcup\Omega_{v,ab},\\
\Omega_{u,bc}&=\Omega_{u,b}\cap\Omega_{u,c}, & \Omega_{v,bc}&=\Omega_{v,b}\cap\Omega_{v,c}, & \Omega_{bc}&=\Omega_{u,bc}\sqcup\Omega_{v,bc},\\
\Omega_{u,ac}&=\Omega_{u,a}\cap\Omega_{u,c}, & \Omega_{v,ac}&=\Omega_{v,a}\cap\Omega_{v,c}, & \Omega_{ac}&=\Omega_{u,ac}\sqcup\Omega_{v,ac},\\
\Omega_{u,abc}&=\Omega_{u,a}\cap\Omega_{u,b}\cap\Omega_{u,c}, & \Omega_{v,abc}&=\Omega_{v,a}\cap\Omega_{v,b}\cap\Omega_{v,c}, & \Omega_{abc}&=\Omega_{u,abc}\sqcup\Omega_{v,abc}.
\end{align*}

Since by Remark~\ref{rem:novalence1} we assume that neither of graphs $\G_H$, $\G_K$, $\G_{H\cap K}$ can have a vertex of valence $1$, we observe that each vertex in $\Omega_{u,a}$ is also the origin of either another $b$--edge, or $c$--edge, or both. Thus,
\[
\Omega_{u,a}=\Omega_{u,ab}\bigvee_{\Omega_{u,abc}}\Omega_{u,ac}, 
\]
where $K=L\bigvee_NM$ means that $K=L\cup M$ and $L\cap M=N$. Similarly, each vertex in $\Omega_{v,a}$ is also the terminus of either another $b$--edge, or $c$--edge, or both. Thus,
\[
\Omega_{v,a}=\Omega_{v,ab}\bigvee_{\Omega_{v,abc}}\Omega_{v,ac}, 
\]
Of course, a completely similar statement holds for all other graphs~\eqref{eq:dicks2} in place of $\Omega_{u,a}$, $\Omega_{v,a}$. 

Clearly, $\Omega_{ab}\cap\Omega_{bc}=\Omega_{bc}\cap\Omega_{ab}=\Omega_{ac}\cap\Omega_{ab}=\Omega_{abc}$. Thus we can denote
\[
A=\Omega_{ab}\bigvee_{\Omega_{abc}}\Omega_{ac},\qquad B=\Omega_{ab}\bigvee_{\Omega_{abc}}\Omega_{bc},\qquad
C=\Omega_{ac}\bigvee_{\Omega_{abc}}\Omega_{bc}. 
\]

Notice that, by construction, all graphs $A$, $B$, $C$ are subgraphs of $\Omega=\Omega_u\sqcup\Omega_v$.

In what follows we will depict connected components of $\Omega$ against the following `trefoil' Venn diagram, illustrating the relation
\[
\Omega_{ab}\cap\Omega_{bc}=\Omega_{ab}\cap\Omega_{ac}=\Omega_{bc}\cap\Omega_{ac}=\Omega_{abc},
\]
see~Figure~\ref{fig:venn}.

\begin{figure}[ht]
\begin{center}
\begin{tikzpicture}[scale=0.75]
\draw [rounded corners=5pt] (0,3.464) arc (60:-60:2) arc (240:60:2) arc (360:180:2) arc (120:60:2);
\draw (0,2.35) node {$\Omega_{abc}$};
\draw (0,0.85) node {$\Omega_{bc}$};
\draw (-1.25,3) node {$\Omega_{ab}$};
\draw (1.25,3) node {$\Omega_{ac}$};
\begin{scope}[yshift=0.5cm]
\draw [snake=brace,segment amplitude=2mm] (-2,4) -- (2,4); \draw (0,4.65) node {$A$};
\end{scope}
\begin{scope}[xshift=-0.25cm,yshift=-0.25cm]
\draw [snake=brace,segment amplitude=2mm] (-0.6,-0.25) -- (-2.5,3); \draw (-2.15,1) node {$B$};
\end{scope}
\begin{scope}[xshift=0.25cm,yshift=-0.25cm]
\draw [snake=brace,segment amplitude=2mm] (2.5,3) -- (0.6,-0.25); \draw (2.15,1) node {$C$};
\end{scope}
\end{tikzpicture}
\end{center}
\caption{Venn diagram for the Dicks graphs. 
\label{fig:venn}}
\end{figure}
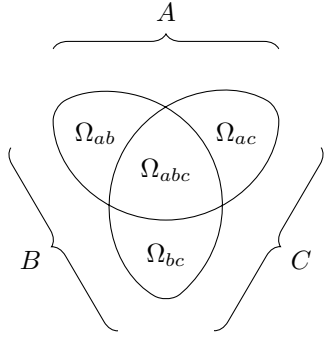

Now, for $A$, we have:
\begin{equation}\label{eq:A}
\begin{multlined}
A=\Omega_{ab}\cup\Omega_{ac}=(\Omega_{u,ab}\sqcup\Omega_{v,ab})\cup(\Omega_{u,ac}\sqcup\Omega_{v,ac})=\\[1ex]
(\Omega_{u,ab}\cup\Omega_{u,ac})\sqcup(\Omega_{v,ab}\cup\Omega_{v,ac})=\Omega_{u,a}\sqcup\Omega_{v,a}\cong \Omega_a\sqcup\Omega_a,
\end{multlined}
\end{equation}
and similarly for $B$, $C$.

Thus we established the following duality, discovered by Dicks~\cite{Di}:
\begin{prop}[Dicks' duality]\label{prop:dicksdual}
Each of the graphs $A,B,C$ defined above consists of an even number of connected components, which are isomorphic in pairs. 
If\/ $\{Z,Z'\}$ is such a pair of components of $A$, then one of $Z,Z'$ is a component of\/ $\Omega_{u,a}$ and the other is a component of\/ $\Omega_{v,a}$, and the isomorphism between them preserves the bipartite structure. Similar statements are true for $B$ and $C$ in place of $A$. \qed
\end{prop}

\subsection{Modeling the topological pushout on the Dicks graphs}\label{ssec:modeling} 
Notice that the Dicks graphs $\Omega_u$, $\Omega_v$, $\Omega_a$, $\Omega_b$, $\Omega_c$ express the equivalence relation $\sim$ used to define the topological pushout in subsection~\ref{ssec:tp}: two vertices $z,z'$ are connected by an edge in the graph $\Omega_x$ (where $x$ stands for any of $u,v,a,b,c$) if and only if they are the two projections under $\Pi_H$, $\Pi_K$ of the same element of $\G_{H\cap K}$. Thus, the connected components of $\Omega_u\sqcup\Omega_v$ are exactly the vertices of the topological pushout, and the connected components of $\Omega_a\sqcup\Omega_b\sqcup\Omega_c$ are the edges of the topological pushout, with inclusions $\tilde o$, $\tilde t$ defined above being the origin and the terminus maps. 
In particular, we see from Figure~\ref{fig:dicks-uvabc} that the topological pushout $\T$ from Example~\ref{ex:main} (see Figure~\ref{fig:mainpic}) has exactly two vertices corresponding to the connected graphs $\Omega_u$ and $\Omega_v$ and four directed edges: one $a$--edge for the connected graph $\Omega_a$, one $b$--edge for the connected graph $\Omega_b$ and two $c$--edges for the two connected components of $\Omega_c$.

It will be useful for us to recast the topological pushout in terms of graphs $A,B,C$ defined above:  
\begin{prop}\label{prop:toppush}
The topological pushout $\T$ admits the following description:

\textbf{Vertices:} connected components of\/ $\Omega=\Omega_u\sqcup\Omega_v=\Omega_{ab}\bigvee_{\Omega_{abc}}\Omega_{bc}\bigvee_{\Omega_{abc}}\Omega_{ac} = A\cup B\cup C$.

\textbf{Edges:} pairs of connected components of $A$, $B$, $C$ from the pairing in Proposition~\ref{prop:dicksdual}.

If $e=\{Z,Z'\}$ is such a pair, viewed as a directed edge, it inherits its type from the corresponding graph it belongs to: if $e\subset A$, then $e$ is an $a$--edge, if $e\subset B$, then $e$ is a $b$--edge, and if $e\subset C$ then $e$ is a $c$--edge. The origin and terminus maps are defined as follows. If $e\subset A$, then one of $Z,Z'$ is a connected component of\/ $\Omega_{u,a}$, and the other, of\/ $\Omega_{v,a}$. Let $Z\subset\Omega_{u,a}$, $Z'\subset\Omega_{v,a}$, say. Then the origin of $e$ is the connected component of $(A\cup B\cup C)\cap \Omega_{u}$  in which $Z$ lies, 
and the terminus of $e$ is the connected component of $(A\cup B\cup C)\cap \Omega_{v}$ in which $Z'$ lies. 
Similar definitions apply to $B$ and $C$ in place of $A$.
\end{prop}
\begin{proof}
By the definition of $\T$, its vertices are the connected components of $\Omega=\Omega_u\sqcup\Omega_v$. Since every vertex of $\G_H$, $\G_K$ is incident to either an $a$--edge, a $b$--edge, or a $c$--edge, the same is true for $\T$. Hence $\Omega_u=\Omega_{u,a}\cup\Omega_{u,b}\cup\Omega_{u,c}$ and $\Omega_v=\Omega_{v,a}\cup\Omega_{v,b}\cup\Omega_{v,c}$. From~\eqref{eq:A} we see that $\Omega_{u,a}\sqcup\Omega_{v,a}=A$, $\Omega_{u,b}\sqcup\Omega_{v,b}=B$, and $\Omega_{u,c}\sqcup\Omega_{v,c}=C$. It follows that $\Omega_u\sqcup\Omega_v=A\cup B\cup C$.

The edges of $\T$, by definition, are the connected components of $\Omega_a\sqcup\Omega_b\sqcup\Omega_c$, with the attaching maps $\tilde o$, $\tilde t$, defined above. 
If $Z_0$ is a connected component of $\Omega_a$, say, it defines two isomorphic connected subgraphs $Z=\tilde o(Z_0)\subset\Omega_{u,a}$ and $Z'=\tilde t(Z_0)\subset\Omega_{v,a}$. Since $\Omega_{u,a}$ and $\Omega_{v,a}$ are, by definition, isomorphic copies of $\Omega_a$, the subgraphs $Z,Z'$ are the whole connected components of $\Omega_{u,a}$, $\Omega_{v,a}$, respectively. By~\eqref{eq:A}, $A=\Omega_{u,a}\sqcup\Omega_{v,a}$, therefore the pair $\{Z,Z'\}$ is a pair of connected components of $A$ determined by $Z_0$. The attaching maps for $\{Z,Z'\}$ described in the Proposition are induced by $\tilde o$ and $\tilde t$ applied to $Z_0$. 
\end{proof}

Figure~\ref{fig:toppush} shows the topological pushout for the subgroups of Example~\ref{ex:main},  
modeled by subsets $\Omega_{ab}$, $\Omega_{bc}$, $\Omega_{ac}$, $\Omega_{abc}$, in accordance with Proposition~\ref{prop:toppush}. Notice that the connected components of $\Omega_a,\Omega_b,\Omega_c$ establish bijections between parts $A,B,C$ of $\Omega_u$ (the left `trefoil') and the corresponding parts of $\Omega_v$ (the right `trefoil'). In particular, the connected graph $\Omega_a$ from Figure~\ref{fig:dicks-uvabc} acts as the $a$--edge of $\T$ and establishes a bijection between $\tilde o(\Omega_a)=\{1{-}7{-}3\}$ with $\tilde t(\Omega_a)=\{2{-}6{-}4\}$. Also, the connected graph $\Omega_b$ from Figure~\ref{fig:dicks-uvabc} is the $b$--edge of $\T$ and it establishes a bijection between $\tilde o(\Omega_b)=\{5{-}3{-}7\}$ and $\tilde t(\Omega_b)=\{8{-}2{-}6\}$. Finally, each connected component of $\Omega_c$ from Figure~\ref{fig:dicks-uvabc} serves as a $c$--edge of $\T$ and they establish bijections of the two connected components $1{-}7$, $3{-}5$ of $\Omega_{u,c}$ with $2{-}8$, $4{-}6$ of $\Omega_{v,c}$. Notice also that this bijection does not preserve subsets $\Omega_{ac}$ and $\Omega_{bc}$ individually but leaves invariant their union $C=\Omega_{ac}\bigvee_{\Omega_{abc}}\Omega_{bc}$: the subgraph $1{-}7$ of $\Omega_{u,ac}$ gets paired with the subgraph $2{-}8$ of $\Omega_{v,bc}$, and also $3{-}5$ of $\Omega_{u,bc}$ gets paired with $4{-}6$ of $\Omega_{v,ac}$. The curved edges $3{-}7$ and $2{-}6$ reflect the fact that vertices $3$ and $7$ (and also $2$ and $6$) are adjacent in $\Omega_{ab}$, but not in $\Omega_{abc}$.
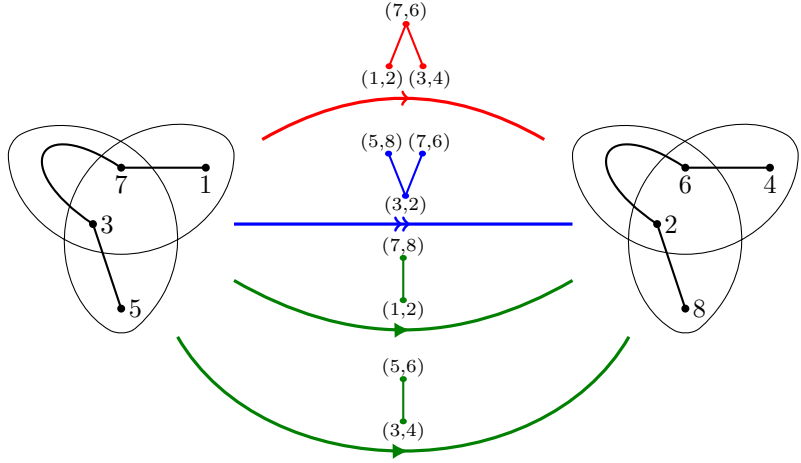
\begin{figure}[ht]
\begin{center}
\begin{tikzpicture}[scale=0.75]
\draw [rounded corners=5pt] (0,3.464) arc (60:-60:2) arc (240:60:2) arc (360:180:2) arc (120:60:2);
\fill (1.5,3) circle (2pt);
\fill (0,3) circle (2pt);
\fill (-0.5,2) circle (2pt);
\fill (0,0.5) circle (2pt);
\draw [thick] (1.5,3)--(0,3);
\draw [thick] (-0.5,2)--(0,0.5);
\draw [thick] (0,3) .. controls (-1.5,4) and (-2,3) .. (-0.5,2);
\draw (1.5,2.7) node {$1$};
\draw (0,2.7) node {$7$};
\draw (-0.25,2) node {$3$};
\draw (0.25,0.5) node {$5$};

\draw[->-i=0.52,color=red,very thick] (2.5,3.5) to [out=30, in=150] (7.5,3.5);
\draw[->-ii=0.52,color=blue,very thick] (2,2) to [out=0, in=180] (8,2);
\draw[->-=0.52,color=mygreen,very thick] (2,1) to [out=-30, in=210] (8,1);
\draw[->-=0.51,color=mygreen,very thick] (1,0) to [out=-60, in=240] (9,0);

\begin{scope}[thick, xshift=4.75cm, yshift=4.8cm, xscale=0.3, yscale=0.25, red]
\fill (0,0) circle (6pt);
\fill (2,0) circle (6pt);
\fill (1,3) circle (6pt);
\draw (0,0)--(1,3);
\draw (2,0)--(1,3);
\draw [black] (-0.4,-0.9) node {\scriptsize(1,2)}; 
\draw [black] (2.4,-0.9) node {\scriptsize(3,4)}; 
\draw [black] (1,3.9) node {\scriptsize(7,6)}; 
\end{scope}

\begin{scope}[thick, xshift=4.74cm, yshift=2.5cm, xscale=0.3, yscale=0.25, blue]
\fill (1,0) circle (6pt);
\fill (0,3) circle (6pt);
\fill (2,3) circle (6pt);
\draw (1,0)--(0,3);
\draw (2,3)--(1,0);
\draw [black] (1,-0.7) node {\scriptsize(3,2)}; 
\draw [black] (-0.4,3.8) node {\scriptsize(5,8)}; 
\draw [black] (2.4,3.8) node {\scriptsize(7,6)}; 
\end{scope}

\begin{scope}[thick, xshift=5.0cm, yshift=0.65cm, xscale=0.3, yscale=0.25, mygreen]
\fill (0,0) circle (6pt);
\fill (0,3) circle (6pt);
\draw (0,0)--(0,3);
\draw [black] (0,-0.7) node {\scriptsize(1,2)}; 
\draw [black] (0,3.8) node {\scriptsize(7,8)}; 
\end{scope}

\begin{scope}[thick, xshift=4.4cm, yshift=-1.5cm, xscale=0.3, yscale=0.25, mygreen]
\fill (2,0) circle (6pt);
\fill (2,3) circle (6pt);
\draw (2,0)--(2,3);
\draw [black] (2,-0.7) node {\scriptsize(3,4)}; 
\draw [black] (2,3.8) node {\scriptsize(5,6)}; 
\end{scope}

\begin{scope}[xshift=10cm]
\draw [rounded corners=5pt] (0,3.464) arc (60:-60:2) arc (240:60:2) arc (360:180:2) arc (120:60:2);
\fill (1.5,3) circle (2pt);
\fill (0,3) circle (2pt);
\fill (-0.5,2) circle (2pt);
\fill (0,0.5) circle (2pt);
\draw [thick] (1.5,3)--(0,3);
\draw [thick] (-0.5,2)--(0,0.5);
\draw [thick] (0,3) .. controls (-1.5,4) and (-2,3) .. (-0.5,2);
\draw (1.5,2.7) node {$4$};
\draw (0,2.7) node {$6$};
\draw (-0.25,2) node {$2$};
\draw (0.25,0.5) node {$8$};

\end{scope}

\end{tikzpicture}
\end{center}
\caption{The topological pushout $\T$ from Figure~\ref{fig:mainpic} modeled on the Dicks graphs. 
Vertices $3$ and $7$ (and also $2$ and $6$) are adjacent in $\Omega_{ab}$, but not in $\Omega_{abc}$, which is depicted by a curved edge lying in the `petal' for $\Omega_{ab}$. 
\label{fig:toppush}}
\end{figure}

Now we are ready to relate ranks of $H$, $K$, $H\cap K$ and $\T$ with the structure of the Dicks graph $\Omega=\Omega_{ab}\bigvee_{\Omega_{abc}}\Omega_{bc}\bigvee_{\Omega_{abc}}\Omega_{ac}$. We are going to prove the following Theorem~\ref{thm:ranks} assuming the validity of Proposition~\ref{prop:cn}, which will be proved in Section~\ref{sec:cn}. Recall that a \emph{cycle} in an undirected graph is a path $x_0e_0x_1e_1\dots x_ke_kx_0$ such that all vertices $x_i$ are pairwise different. We denote by $\#W$ the cardinality of a finite set $W$ and, as before, by $\rr(H)$ the reduced rank of group $H$, i.e.\ the quantity $\max(0,\rk(H)-1)$.

\begin{thm}\label{thm:ranks}
Let $F$ be a free group and $H,K\le F$ be finitely generated subgroups. Let\/ $\G_H$, $\G_K$ and\/ $\G_{H\cap K}$ be the core graphs of the corresponding subgroups, and let $\T$ be their topological pushout. Let also\/ $\Omega$ and\/ $\Omega_{abc}$ be the Dicks graphs defined above. Then
\begin{enumerate}
\item $\Omega_{abc}$ is a bipartite graph with\/ $2\rr(H)$ vertices in one part and\/ $2\rr(K)$ vertices in the other;
\item $\Omega_{abc}$ has\/ $2\rr(H\cap K)$ edges;
\item (\# connected components of\/ $\Omega_{abc}$) $\ge 2\rr(\T)$, with the equality taking place if and only if every cycle of $\Omega$ lies entirely in one of the subgraphs\/ $\Omega_{ab}$, $\Omega_{bc}$, $\Omega_{ac}$ (with different cycles possibly lying in different subgraphs).
\end{enumerate}
\end{thm}
\begin{proof}
The graph $\Omega_{abc}$ has as its vertices all vertices of valence $3$ of $\G_H$ and $\G_K$. Since graphs $\G_H$, $\G_K$, $\G_{H\cap K}$ are normalized as in Remark~\ref{rem:novalence1}, all their vertices have valence either $2$ or $3$. Computing Euler characteristic of $\G$ (where $\G$ stands for any of $\G_H$, $\G_K$, $\G_{H\cap K}$) gives:
\[
1-\rk(\G) = \#V(\G)-\#\big(E^+(\G)\big)= \sum_{v\in V(\G)}\left(1-\frac{\val(v)}2\right),
\]
which is equivalent to
\[
2\rr(\G)=\sum_{v\in V(\G)}\big(\!\val(v)-2\big)=\text{\ \big(\# vertices of valence $3$ in $\G$\big)},
\]
the last equality being true since vertices of valence $2$ contribute $0$ to the sum. This proves part (1).

Edges of $\Omega_{abc}$ are exactly the vertices of valence $3$ in $\G_{H\cap K}$, hence the computation above establishes part (2) as well.

For part (3), exactly as above, we have 
\begin{equation}\label{eq:val}
2\rr(\T)=\sum_{v \in V(\T)} \big(\!\val(v)-2\big).
\end{equation}
According to Proposition~\ref{prop:toppush}, the valence of a vertex $v$ of $\T$ (represented by some connected component $D$ of $A\cup B\cup C$) is the sum of numbers of connected components of $A\cap D$, $B\cap D$ and $C\cap D$. Our goal is to understand the relationship between components of $\Omega_{abc}$, components of $A\cup B\cup C$, and components of $A$, $B$ and $C$, taken separately. The difficulty lies in an observation that two components $P$, $Q$ of $\Omega_{abc}$ may be connected by a path outside $\Omega_{abc}$, i.e.\ by a path all edges of which lie in one of the graphs $\Omega_{ab}\setminus\Omega_{abc}$, $\Omega_{bc}\setminus\Omega_{abc}$, $\Omega_{ac}\setminus\Omega_{abc}$. If this is the case, then $P$ and $Q$ actually correspond to the same connected component $D$ of $A\cup B\cup C$, i.e.\ to the same vertex of $\T$, and their contribution to the valence of $\T$ may be different from the value $2\cdot 3$ expected otherwise. A careful treatment of this situation is given in Section~\ref{sec:cn}, where we take an abstract approach and study a certain class $\Cn$ of graphs $\G$ with a function $\Sigma$ associated to them, which encode the connectedness of components of $\Omega_{abc}$ to each other through the three graphs $\Omega_{ab}$, $\Omega_{bc}$, $\Omega_{ac}$, and their joint contribution to the right-hand side of~\eqref{eq:val}. To get the input for the main result of Section~\ref{sec:cn}, Proposition~\ref{prop:cn}, we form the following undirected graph $\G$, which we will call the \emph{component connectivity graph} (CCG) of $\Omega_{abc}$.

Vertices of $\G$ are connected components of $\Omega_{abc}$.

If $p$, $q\in V(\G)$, they are some components $P$, $Q$ of $\Omega_{abc}$. If there exists a path between $P$ and $Q$ which does not contain any vertices of $\Omega_{abc}$ (except the first vertex of the path and the last one), and all edges of which lie in $\Omega_{ab}\setminus\Omega_{abc}$, we connect vertices $p,q$ in $E(\G)$ with an undirected edge and assign the color \emph{magenta} to it. Similarly, if a $\Omega_{abc}$-avoidant path between $P$ and $Q$ lies in $\Omega_{ac}\setminus\Omega_{abc}$, we add an edge to $E(\G)$ connecting $p$ and $q$ and assign the color \emph{yellow} to it. Lastly, if such path lies in $\Omega_{bc}\setminus\Omega_{abc}$, we add an edge to $E(\G)$ connecting $p$ and $q$ and assign the color \emph{cyan} to it. Thus every two vertices of $\G$ may be connected by up to three undirected edges, each having a different color. (The choice of names for the colors is suggested by mixing the basic colors red, blue and green, which we used to depict $a$--edges, $b$--edges and $c$--edges, respectively. Hence, edges from $\Omega_{ab}\setminus\Omega_{abc}$ get color red-blue, i.e.\ magenta, edges from $\Omega_{ac}\setminus\Omega_{abc}$ get color red-green, i.e.\ yellow, and edges from $\Omega_{bc}\setminus\Omega_{abc}$ get color blue-green, i.e.\ cyan.) 

Thus, the graph $\G$ encodes the connectedness information (within $A\cup B\cup C$) between different connected components of $\Omega_{abc}$. The contribution of vertices of $\T$ to the sum in~\eqref{eq:val} (i.e.\ the right-hand side of \eqref{eq:val}) is equal to the function $\Sigma(\G)$, defined in equation~\eqref{eq:sigma} of Section~\ref{sec:cn}.

Having formed the input for Proposition~\ref{prop:cn}, we can use its conclusion, which reads: $\Sigma(\G)\le n$. Here $n$ is the number of vertices of $\G$, i.e.\ the number of connected components of $\Omega_{abc}$, and $\Sigma(\G)$ is the right-hand side of~\eqref{eq:val}. This proves the inequality in part (3) of the Theorem. Proposition~\ref{prop:cn} also specifies when we have the equality in $\Sigma(\G)\le n$: this happens if and only if all cycles of $\G$ are \emph{monochromatic} in the terminology of Section~\ref{sec:cn}. In terms of the Dicks graphs, this means that every cycle of $\Omega_{ab}\bigvee_{\Omega_{abc}}\Omega_{bc}\bigvee_{\Omega_{abc}}\Omega_{ac}$ lies entirely in either $\Omega_{ab}$, or $\Omega_{bc}$, or $\Omega_{ac}$. This finishes the proof of part (3) of the Theorem.
\end{proof}

Situations when the inequality in part (3) of the Theorem~\ref{thm:ranks} is strict appear quite frequently, see for example Figure~\ref{fig:k23} in Section~\ref{sec:guzman4}.

\section{Class \texorpdfstring{$\Cn$}{Cn}}\label{sec:cn}

The goal of this technical section is to prove Proposition~\ref{prop:cn} needed for the proof of part (3) of Theorem~\ref{thm:ranks}.

Consider a class $\Cn$ of pairs $(\G,c)$ where $\G$ is an undirected graph with multiple edges allowed (but not loops), and $c\colon E(\G)\to \{\text{{\it magenta, yellow, cyan}}\}$ is an edge-coloring map, with the following properties:
\begin{enumerate}
\item Each graph $\G$ from $\Cn$  has exactly $n$ vertices.
\item Let $E(\G,p,q)=E(\G,q,p)$ denote the set of all undirected edges between two different vertices $p,q\in V(\G)$. Then the edge-coloring map $c$ is injective on each set $E(\G,p,q)$.
\end{enumerate}
In other words, any two different vertices $p,q$ of $\G$ may be joined by up to three undirected edges of different colors from the set {\it\{magenta, yellow, cyan\}}.

For any $\G\in\Cn$ we define three subgraphs $\G_{my}$, $\G_{yc}$, $\G_{mc}$ of $\G$ as follows:
\begin{itemize}
\item $V(\G_{my})=V(\G_{yc})=V(\G_{mc})=V(\G)$;
\item $E(\G_{my})=\{\text{all {\it magenta} and {\it yellow} edges of $\G$}\}$;
\item $E(\G_{yc})=\{\text{all {\it yellow} and {\it cyan} edges of $\G$}\}$;
\item $E(\G_{mc})=\{\text{all {\it magenta} and {\it cyan} edges of $\G$}\}$;
\end{itemize}

\begin{figure}[ht]
\begin{center}
\begin{tikzpicture}[scale=0.75]
\draw [rounded corners=5pt] (0,3.464) arc (60:-60:2) arc (240:60:2) arc (360:180:2) arc (120:60:2);
\begin{scope}[yshift=0.5cm]
\draw [snake=brace,segment amplitude=2mm] (-2,4) -- (2,4); \draw (0,4.5) node {$\G_{my}$};
\end{scope}
\begin{scope}[xshift=-0.25cm,yshift=-0.25cm]
\draw [snake=brace,segment amplitude=2mm] (-0.6,-0.25) -- (-2.5,3); \draw (-2,1) node {$\G_{mc}$};
\end{scope}
\begin{scope}[xshift=0.25cm,yshift=-0.25cm]
\draw [snake=brace,segment amplitude=2mm] (2.5,3) -- (0.6,-0.25); \draw (2.25,1) node {$\G_{yc}$};
\end{scope}
\draw[node font=\itshape] (-1.25,3) node [align=center,magenta,font=\scriptsize] {magenta\\ edges};
\draw[node font=\itshape] (1.25,3) node [align=center,orange,font=\scriptsize] {yellow\\ edges};
\draw[node font=\itshape] (0,0.75) node [align=center,cyan,font=\scriptsize] {cyan\\ edges};
\fill (0,3) circle (2pt);
\fill (-0.5,2) circle (2pt);
\fill (0.5,2) circle (2pt);
\draw (0,2.35) node {\footnotesize $V(\G)$};
\end{tikzpicture}
\end{center}
\caption{Venn diagram for graphs $\G_{my}$, $\G_{yc}$, $\G_{mc}$.}
\end{figure}
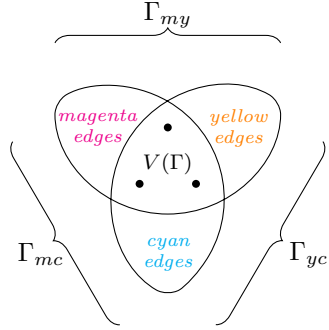

We define a function $\Sigma\colon\Cn\to\Z_{\ge0}$ as follows:
\begin{equation}\label{eq:sigma}
\Sigma(\G)=\sum_{C\in CC(\G)} \big(\valmy(C)+\valyc(C)+\valmc(C)-2\big),
\end{equation}
where $CC(\G)$ denotes the set of all connected components of $\G$, and
\[
\val_{colors}(C)=\text{ \# connected components of $(C\cap \G_{colors})$}
\]
for $\mathit{colors}\in\{my,yc,mc\}$.

Recall that a {\it cycle} in an undirected graph $\G$ is a sequence $x_0e_0x_1e_1\dots x_ke_kx_0$ of pairwise different vertices $x_i$ and edges $e_i\in E(\G,x_i,x_{i+1})$, $e_k\in E(\G,x_{k},x_0)$. A cycle is called {\it monochromatic} if all its edges are of the same color (from the set \{{\it magenta, yellow, cyan}\}).

\begin{prop}\label{prop:cn}
For any\/ $\G\in\Cn$, we have: $\Sigma(\G)\le n$, with the equality taking place if and only if all cycles of\/ $\G$ are monochromatic (with different cycles possibly having different colors).
\end{prop}
\begin{proof}
Our first observation is:

\underline{\it If\/ $\G\in \Cn$ is edgeless then\/ $\Sigma(\G)= n$.}

Indeed, in that case $\G$ has $n$ connected components which are singleton vertices and $\G_{my}=\G_{yc}=\G_{mc}=\G$, so $\Sigma(\G)=n\cdot(1+1+1-2)=n$.

We are going to prove that adding an edge to an arbitrary graph $\G\in\Cn$ may only decrease $\Sigma$, and we identify all cases when the decrease does not happen.

Choose two vertices $p,q\in V(\G)$ such that $|E(\G,p,q)|<3$ and consider 
\[
\G'=\G\cup e,
\]
where $e$ is a new edge between $p$ and $q$ of a color that is not present in $E(\G,p,q)$. Without loss of generality, we may assume that the color of edge $e$ is {\it magenta}.

Let $[p],[q]$ denote the connected components of $\G$ containing $p,q$, respectively, and let $[p]_{colors}$, $[q]_{colors}$ be the connected components of $\G_{colors}$ containing $p,q$, respectively, for $colors\in\{my,yc,mc\}$.

Also, let $[p]',[q]'$ denote the connected components of $\G'$ containing $p,q$, respectively, and let $[p]'_{colors}$, $[q]'_{colors}$ be the connected components of $\G'_{colors}$ containing $p,q$, respectively, for $colors\in\{my,yc,mc\}$.

Let also
\[
\val'_{colors}(C')=\text{ \# connected components of $(C'\cap \G'_{colors})$}
\]
for $C'$ a connected component of $\G'$ and $\mathit{colors}\in\{my,yc,mc\}$.

We look at several cases.

\underline{Case I: $[p]\ne[q]$}. Then $[p]_{colors}\ne[q]_{colors}$ for any $colors\in \{my,yc,mc\}$, and adding {\it magenta} edge $e$ to $p,q$ makes $[p]'=[q]'$, $[p]'_{my}=[q]'_{my}$ and $[p]'_{mc}=[q]'_{mc}$ while $[p]'_{yc}\ne[q]'_{yc}$ (since $\G_{yc}$ by definition has only yellow and cyan edges, so if $[p]'_{yc}=[q]'_{yc}$ then $[p]_{yc}=[q]_{yc}$). Thus 
\[
\Sigma(\G)=\big(\valmy[p]+\valyc[p]+\valmc[p] - 2\big) + 
\big(\valmy[q]+\valyc[q]+\valmc[q] - 2\big) + \sum_{\substack{C\ne[p],[q]\\ C\in CC(\G)}}(\dots),
\]
and
\[
\Sigma(\G')=\big(\valmy'[p]'+\valyc'[p]'+\valmc'[p]' - 2\big) + \sum_{\substack{C'\ne[p]'\\ C'\in CC(\G')}}(\dots),
\]

Now,
\begin{multline*}
\valmy'[p]'=\big(\text{\# components of $[p]'\cap \G'_{my}$}\big) = \\
\big(\text{\# components of $[p]\cap \G_{my}$}\big) + \big(\text{\# components of $[q]\cap \G_{my}$}\big) - 1 = \\
 \valmy[p]+\valmy[q]-1.
\end{multline*}

\begin{figure}[ht]
\begin{center}
\begin{tikzpicture}[scale=1.25]
\draw[rounded corners=8pt] (0,3.464) arc (60:-60:2) arc (240:60:2) arc (360:180:2) arc (120:60:2);
\begin{scope}[xshift=0.1cm]
\draw (-0.8,1.8) .. controls (-0.7,0.25) and (0.0,0.25) .. (-0.1,1.8);
\draw (-0.6,1.8) .. controls (-0.5,0.50) and (-0.1,0.50) .. (-0.1,1.8);
\draw (-0.4,1.8) .. controls (-0.3,0.75) and (-0.2,0.75) .. (-0.1,1.8);
\draw (-0.1,0.5) node {\footnotesize$[p]$};
\draw (-1,1.8) node {\footnotesize$p$};
\draw [snake=brace] (-0.85,1.9)--(-0.05,1.9); \draw (-0.5,2.15) node {\footnotesize$\valmy[p]$};
\fill (-0.8,1.8) circle (1.2pt);
\fill (-0.6,1.8) circle (1.2pt);
\fill (-0.4,1.8) circle (1.2pt);
\fill (-0.1,1.8) circle (1.2pt);
\end{scope}

\draw (0.25,2.75) .. controls (0.25,0.25) and (0.85,0.25) .. (0.75,2.25);
\draw (0.5,2.5) .. controls (0.5,0.5) and (0.7,0.5) .. (0.75,2.25);
\draw (0.4,0.65) node {\footnotesize$[q]$};
\draw (0.125,2.625) node {\footnotesize$q$};
\draw [snake=brace] (0.25,2.85)--(0.85,2.25); \draw (1.25,2.75) node {\footnotesize$\valmy[q]$};

\draw [very thick, dashed, magenta] (-0.7,1.8) .. controls (-2.2,2.8) and (-1.3,4.3) .. (0.25,2.75);
\draw [magenta] (-1.2,3) node {$e$};
\fill (-0.7,1.8) circle (1.2pt);

\fill (0.25,2.75) circle (1.2pt);
\fill (0.5,2.5) circle (1.2pt);
\fill (0.75,2.25) circle (1.2pt);

\end{tikzpicture}
\end{center}
\caption{Case I.}
\end{figure}

Similarly,
\begin{multline*}
\valmc'[p]'=\big(\text{\# components of $[p]'\cap \G'_{mc}$}\big) = \\
\big(\text{\# components of $[p]\cap \G_{mc}$}\big) + \big(\text{\# components of $[q]\cap \G_{mc}$}\big) - 1 = \\
 \valmc[p]+\valmc[q]-1.
\end{multline*}

And
\begin{multline*}
\valyc'[p]'=\big(\text{\# components of $[p]'\cap \G'_{yc}$}\big) = \big(\text{\# components of $[p]'\cap \G_{yc}$}\big)=\\
\big(\text{\# components of $([p]\cup[q])\cap \G_{yc}$}\big) = \big(\text{\# components of $([p]\cap \G_{yc})\cup([q]\cap \G_{yc})$}\big) = \\
 \valyc[p]+\valyc[q].
\end{multline*}

Comparing the contributions of the left- and right-hand sides to the function $\Sigma$, and noticing that the components $C\in CC(\G)\setminus\{[p],[q]\}$ and $C'\in CC(\G')\setminus\{[p]'\}$ pairwise coincide, we conclude that the two sums are equal: $\Sigma(\G')=\Sigma(\G)$ in Case I. This proves, in particular, that

\underline{\it If $\G\in \Cn$ is a forest then $\Sigma(\G)= n$.}

Indeed, every forest can be obtained from an edgeless graph by adding edges which connect disjoint components.

\underline{Case II: $[p]=[q]$}. Thus $[p]'=[q]'$ and the four subcases are possible:
\begin{itemize}
\item[(1)] $[p]_{my}\ne[q]_{my}$, $[p]_{mc}\ne[q]_{mc}$;
\item[(2)] $[p]_{my}=[q]_{my}$, $[p]_{mc}\ne[q]_{mc}$;
\item[($2'$)] $[p]_{my}\ne[q]_{my}$, $[p]_{mc}=[q]_{mc}$;
\item[(3)] $[p]_{my}=[q]_{my}$, $[p]_{mc}=[q]_{mc}$.
\end{itemize}
 
\underline{Subcase (1):} $[p]=[q]$, $[p]_{my}\ne[q]_{my}$, $[p]_{mc}\ne[q]_{mc}$. 

These conditions mean that there exists a path $p{-}q$ in $\G$, but not in $\G_{my}$ or $\G_{mc}$. This means that some edge of the path $p{-}q$ must be {\it cyan}, and some other {\it yellow}. Thus adding a {\it magenta} edge $e$ between $p$ and $q$ does create a non-monochromatic cycle.

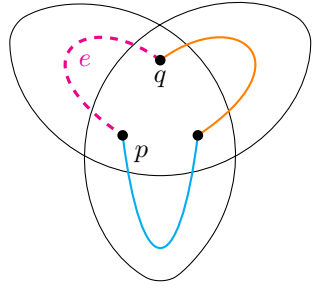
\begin{figure}[ht]
\begin{center}
\begin{tikzpicture}[scale=1.0]
\draw[rounded corners=6.33pt] (0,3.464) arc (60:-60:2) arc (240:60:2) arc (360:180:2) arc (120:60:2);
\draw (-0.25,1.75) node {$p$};
\draw (0,2.75) node {$q$};
\draw [very thick, dashed, magenta] (0,3) .. controls (-1,3.75) and (-2,3) .. (-0.5,2);
\draw [magenta] (-1,3) node {$e$};

\draw [thick, orange] (0,3) .. controls (1,3.75) and (2,3) .. (0.5,2);
\draw [thick, cyan] (-0.5,2) .. controls (-0.25,0) and (0.25,0) .. (0.5,2);

\fill (0,3) circle (2pt);
\fill (-0.5,2) circle (2pt);
\fill (0.5,2) circle (2pt);
\end{tikzpicture}
\end{center}
\caption{Case II, subcase (1).}
\end{figure}

Computing the contribution of $[p]=[q]$ to $\Sigma(\G)$ and of $[p]'=[q]'$ to $\Sigma(\G')$ we see that:
\begin{align*}
\valmy'[p]'=&\valmy[p]-1,\\
\valmc'[p]'=&\valmc[p]-1,\\
\valyc'[p]'=&\valyc[p],
\end{align*}
and
\begin{align*}
\Sigma(\G)=&\sum_{[p]\in CC(\G)}\big(\valmy[p]+\valyc[p]+\valmc[p]-2\big)\\
\Sigma(\G')=&\sum_{[p]'\in CC(\G')}\big(\valmy'[p]'+\valyc'[p]'+\valmc'[p]'-2\big).
\end{align*}
Hence, $\Sigma(\G')=\Sigma(\G)-2$.

Subcases (2) and ($2'$) are symmetric, so we consider only subcase (2).

\underline{Subcase (2):} $[p]=[q]$, $[p]_{my}=[q]_{my}$, $[p]_{mc}\ne[q]_{mc}$. 

This means that $p$ and $q$ are connected by a path $p{-}q$ in $\G_{my}$, but not in $\G_{mc}$. In particular, every such path $p{-}q$ must contain a {\it yellow} edge (otherwise all edges of $p{-}q$ would be {\it magenta} and and $p{-}q$ would lie in $\G_{mc}$).
Thus adding a new {\it magenta} edge $e$ between $p$ and $q$ does create a non-monochromatic cycle.

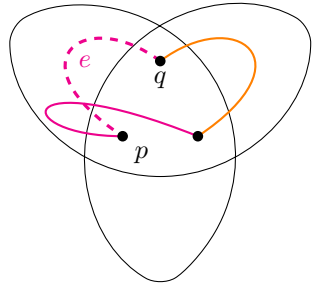
\begin{figure}[ht]
\begin{center}
\begin{tikzpicture}[scale=1.0]
\draw[rounded corners=6.33pt] (0,3.464) arc (60:-60:2) arc (240:60:2) arc (360:180:2) arc (120:60:2);
\draw (-0.25,1.75) node {$p$};
\draw (0,2.75) node {$q$};
\draw [very thick, dashed, magenta] (0,3) .. controls (-1,3.75) and (-2,3) .. (-0.5,2);
\draw [magenta] (-1,3) node {$e$};

\draw [thick, orange] (0,3) .. controls (1,3.75) and (2,3) .. (0.5,2);
\draw [thick, magenta] (-0.5,2) .. controls (-2,2) and (-2,3) .. (0.5,2);

\fill (0,3) circle (2pt);
\fill (-0.5,2) circle (2pt);
\fill (0.5,2) circle (2pt);
\end{tikzpicture}
\end{center}
\caption{Case II, subcase (2).}
\end{figure}

Computing the contribution of $[p]=[q]$ to $\Sigma(\G)$ and of $[p]'=[q]'$ to $\Sigma(\G')$ in this case, we observe that:
\begin{align*}
\valmy'[p]'=&\valmy[p],\\
\valmc'[p]'=&\valmc[p]-1,\\
\valyc'[p]'=&\valyc[p],
\end{align*}
hence $\Sigma(\G')=\Sigma(\G)-1$.

\underline{Subcase (3):} $[p]=[q]$, $[p]_{my}=[q]_{my}$, $[p]_{mc}=[q]_{mc}$. 

\underline{Situation (3a):} Every path $p{-}q$ consists entirely of {\it magenta} edges. In this case, adding a new {\it magenta} edge does not create a non-monochromatic cycle. And in this case,
\begin{align*}
\valmy'[p]'=&\valmy[p],\\
\valmc'[p]'=&\valmc[p],\\
\valyc'[p]'=&\valyc[p],
\end{align*}
so that $\Sigma(\G')=\Sigma(\G)$.

\underline{Situation (3b):} There exists a path $p{-}q$ having a {\it non-magenta} edge, let it be {\it yellow}. Then there is a path $pPq$ in $\G_{my}$ with at least one edge {\it yellow} and a path $pQq$ in $\G_{mc}$. Let $p{-}p'$ be the maximal by inclusion common initial subpath of $pPq$ and $pQq$ consisting entirely of {\it magenta} edges. Similarly, let $q'{-}q$ be the maximal by inclusion common terminal subpath of $pPq$ and $pQq$ consisting entirely of {\it magenta} edges. Thus, we can denote $pPq=p{-}p'P'q'{-}q$ and $pQq=p{-}p'Q'q'{-}q$ for some subpaths $P',Q'$. We observe at once that the union of the paths $p'P'q'$ and $p'Q'q'$ is a non-monochromatic cycle already existing in $\G$.
\begin{figure}[ht]
\begin{center}
\begin{tikzpicture}[thick]
\draw [magenta] (-3,0)--(-1.5,0);
\draw [magenta] (1.5,0)--(3,0);
\draw (-1.5,0) to [out=45, in=180] (0,1);
\draw [orange] (1.5,0) to [out=135, in=0] (0,1);
\draw [cyan] (-1.5,0) to [out=-45, in=180] (0,-1);
\draw (0,-1) to [out=0, in=-135] (1.5,0);

\fill (-3,0) circle (1.5pt); \draw (-3,-0.35) node {$p$};
\fill (-1.5,0) circle (1.5pt); \draw (-1.6,-0.35) node {$p'$};
\fill (1.5,0) circle (1.5pt); \draw (1.6,-0.35) node {$q'$};
\fill (3,0) circle (1.5pt); \draw (3,-0.35) node {$q$};
\fill (0,1) circle (1.5pt); \draw (0.5,1.25) node {$P'$};
\fill (0,-1) circle (1.5pt); \draw (-0.5,-1.25) node {$Q'$};
\end{tikzpicture}
\end{center}
\caption{Case II, subcase (3), situation (3b).}
\end{figure}
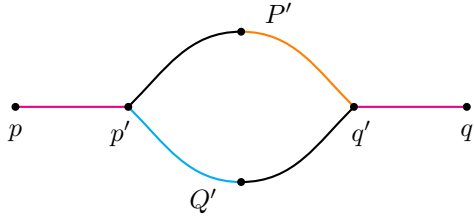

We notice that, in this situation, adding a {\it magenta} edge $e$ to $p,q$ does not change $\Sigma(\G)$ since, as before,
\begin{align*}
\valmy'[p]'=&\valmy[p],\\
\valmc'[p]'=&\valmc[p],\\
\valyc'[p]'=&\valyc[p],
\end{align*}
so that $\Sigma(\G')=\Sigma(\G)$.

However, we can show by an inductive reasoning on the number of edges that in the situation (3b) we have $\Sigma(\G)$ already less than $n$. 

Indeed, we can construct $\G$ from the edgeless graph on $n$ vertices $\G_0$ by adding edges one at a time. We get a sequence of graphs:
\[
\G_0,\G_1,\G_2,\dots, \G_i,\dots, \G_m=\G,
\]
all belonging to $\Cn$ and such that for each $i$, $\G_{i+1}=\G_{i}\cup e_{i+1}$ for a new edge $e_{i+1}$.

If $\G_i$ is a forest, we showed above that $\Sigma(\G_i)=\Sigma(\G_0)=n$.

Suppose that we have already proved by induction on the number $k$ of edges that if $\G_i$ has all cycles monochromatic and $|E(\G_i)|\le k$, then $\Sigma(\G_i)=n$. Consider a new edge $e_{i+1}$ such that $\G_{i+1}=\G_i\cup e_{i+1}$. 

If $e_{i+1}$ does not create a cycle, then $e_{i+1}$ joins two components of $\G_i$, and $\G_{i+1}$ is a forest. Hence, $\Sigma(\G_{i+1})=\Sigma(\G_i)$.

If $e_{i+1}$ creates a non-monochromatic cycle then we are in the subcase (1), (2) or ($2'$) above, and we see that in this case $\Sigma(\G_{i+1})<\Sigma(\G_i)$. 

If $e_{i+1}$ creates a monochromatic cycle, i.e.\ $e_{i+1}$ joins vertices $p,q\in V(\G)$, and there exists a path $p{-}q$ in $\G_i$ of the same color as $e_{i+1}$, then we are in the situation (3a), and $\Sigma(\G_{i+1})=\Sigma(\G_i)$.

The last two cases are mutually exclusive, since, by the inductive hypothesis, $\G_i$ does not have a non-monochromatic cycle.

This inductive reasoning, together with the consideration of the cases and all the subcases above, shows the following:
\begin{itemize}
\item adding an edge to a graph can only decrease the value of $\Sigma$;
\item graphs $\G$ with all cycles monochromatic have $\Sigma(\G)=n$;
\item creating a non-monochromatic cycle (when there were none) decreases the value of $\Sigma$ by 1 or 2.
\end{itemize}

This proves the Proposition.
\end{proof}

\section{Proof of Theorems~\ref{thm:h2} and \ref{thm:guzman4}.\label{sec:guzman4}}

We will prove Theorem~\ref{thm:guzman4} first and Theorem~\ref{thm:h2} at the end of the section. We start with translating Theorem~\ref{thm:guzman4} into a statement about the Dicks graphs. 
\begin{prop}\label{prop:recast}
Let $F$ be a free group, and suppose there exist subgroups $H,K\le F$ such that\/ $\rk(H)$, $\rk(K)\ge 2$, $\rk(H\vee K)=\rk(H)+\rk(K)-i$, and\/ $\rk(H\cap K)=\frac{i(i-1)}{2}+1$, for some $i\ge 3$. Then the corresponding Dicks graphs have the following properties:
\begin{enumerate} 
\item One component of\/ $\Omega_{abc}$ is isomorphic to the complete bipartite graph $K_{i,i-1}$, while all others are singleton vertices.
\item Every cycle of\/ $\Omega=\Omega_{ab}\bigvee_{\Omega_{abc}}\Omega_{bc}\bigvee_{\Omega_{abc}}\Omega_{ac}$ lies entirely in one of the subgraphs\/ $\Omega_{ab}$, $\Omega_{bc}$, or\/ $\Omega_{ac}$ (with different cycles possibly lying in different subgraphs). 
\end{enumerate}
\end{prop}

\begin{proof}
Denote $h=\rr(H)$, $k=\rr(K)$. Then $\rr(H\cap K)=\frac{i(i-1)}{2}$ and from Theorem~\ref{thm:ranks} we conclude that $\Omega_{abc}$ is a bipartite graph with $2h$ vertices in one part and $2k$ vertices in the other, and that $\Omega_{abc}$ has exactly $i(i-1)$ edges. The condition $\rk(H\vee K)=\rk(H)+\rk(K)-i$ can be written in terms of the reduced ranks as $\rr(H\vee K)=h+k-(i-1)$. Also, from the discussion in subsection~\ref{ssec:tp} we know that $\rr(\T)\ge\rr(H\vee K)$. Thus, part (3) of Theorem~\ref{thm:ranks} gives us:
\begin{equation}\label{eq:ccO}
\text{\big(\# connected components of $\Omega_{abc}$\big)} \ge 2\rr(\T)\ge 2h+2k-2(i-1).
\end{equation}
Let $\Omega_{abc}$ have $p$ single-vertex components in the part corresponding to vertices from $V(\G_H)$, $q$ single-vertex components in the part corresponding to vertices from $V(\G_K)$, and $\ell$ components $C_j$, ($j=1,\dots,\ell$) each of which has at least one edge. Then inequality~\eqref{eq:ccO} implies:
\[
p+q+\ell\ge 2h+2k-2(i-1),
\]
or
\begin{equation}\label{eq:8}
(2h-p)+(2k-q)-\ell \le 2(i-1).
\end{equation}
Let $s_j$, $t_j$ denote the number of vertices of $C_j$ in each of the two parts of the bipartite graph $\Omega_{abc}$. Then we have the following equalities: $\sum_{j=1}^\ell s_j = 2h-p$, $\sum_{j=1}^\ell t_j = 2k-q$, and~\eqref{eq:8}, together with the condition $2\rr(H\cap K)=i(i-1)$, becomes:
\begin{align}
\sum_{j=1}^\ell (s_j + t_j - 1)&\le 2(i-1),\label{eq:9}\\
\sum_{j=1}^\ell \text{(\# edges of $C_j$)}&=i(i-1).\label{eq:10}
\end{align}
Notice that the quantity $s_j+t_j-1$ is the number of edges in a spanning tree of the component $C_j$. 
The following Lemma shows that the total number of edges in the left-hand side of~\eqref{eq:10} attains its maximum $i(i-1)$ under the constraint~\eqref{eq:9} if and only if $\ell=1$ and $C_1=K_{i,i-1}$. 

Let's call a component of a bipartite graph $\G$ having at least one edge, \emph{nontrivial}.
\begin{lem}\label{lem:forest}
Let $\G$ be a bipartite graph having $\ell$ nontrivial connected components $C_1,\dots,C_\ell$. Under the constraint:
\[
\sum_{j=1}^\ell (\text{\# edges of a spanning tree of $C_j$})\le 2m,
\]
for some $m\ge 2$, the maximal possible number of edges for\/ $\G$ is achieved when $\ell=1$ and $C_1=K_{m,m+1}$.
\end{lem}
\begin{proof}
(This is essentially the reasoning of Ivanov~\cite[(4.2)]{Iva} with few more details provided.) Let's show first that for any two components $C_1$, $C_2$ of $\G$ with the number of edges in their spanning trees $m_1$, $m_2$ respectively, a single component $C_0$ with a spanning tree having $m_1+m_2$ edges can have a bigger total number of edges than $C_1$ and $C_2$ together. Indeed, let $s_1,t_1$ be numbers of vertices in the two parts of $C_1$ and $s_2,t_2$ be numbers of vertices in the two parts of $C_2$ (we refer to the bipartite structure of $\G$ here). Then the maximal total number of edges in $C_1$ is $s_1t_1$ which is achieved when $C_1=K_{s_1,t_1}$ and, similarly, the maximal total number of edges in $C_2$ is $s_2t_2$.
Now let's consider a join $C_0$ of $C_1$ and $C_2$ at a pair of vertices either in one part or in the other. In the first case we will get a bipartite graph $C_0$ on $s_1+s_2-1$ vertices in one part and $t_1+t_2$ vertices in the other, and in the second case $C_0$ will have $s_1+s_2$ vertices in one part and $t_1+t_2-1$ ones in the other. In both cases, a spanning tree for $C_0$ will have $m_1+m_2$ edges, so that the total number of edges in all spanning trees remains invariant. Notice that the number of edges of the complete bipartite graph $K_{s_1+s_2-1,t_1+t_2}$ is
\[
(s_1+s_2-1)(t_1+t_2)=(s_1t_1+s_2t_2)+(s_1-1)t_2+(s_2-1)t_1 > s_1t_1+s_2t_2, \text{ if either $s_1\ge 2$ or $s_2\ge 2$.}
\]
Similarly, the number of edges of the complete bipartite graph $K_{s_1+s_2, t_1+t_2-1}$ is
\[
(s_1+s_2)(t_1+t_2-1)=(s_1t_1+s_2t_2)+(t_1-1)s_2+(t_2-1)s_1 > s_1t_1+s_2t_2, \text{ if either $t_1\ge 2$ or $t_2\ge 2$.}
\]
We conclude that joining two components $C_1$ and $C_2$ allows us to have a bigger total number of edges in $C_0$ than the sum of edges in $C_1$ and $C_2$, unless $s_1=s_2=t_1=t_2=1$, when joining two $K_{1,1}$'s produces a $K_{1,2}$ with the same number of edges. But since $2m>2$, we are going to deal with components having more than two edges, and we can proceed by induction, joining components together, and each time (after possibly joining two $K_{1,1}$'s the very first time) we increase the maximal possible number of edges while preserving the total  number of edges of all spanning trees. Hence we prove by induction that the maximal number of edges is achieved when there is only one nontrivial component, and it should be a complete bipartite graph $K_{s,t}$ with $s+t-1=2m$. Clearly, the number $st$ of edges of $K_{s,t}$ is maximized under the constraint $s+t-1=2m$ if and only if $s$ and $t$ are closest to being equal. Since $s$ and $t$ have opposite parity, we conclude that $s=m$ and $t=m+1$, or vice versa.
\end{proof}

Hence, part (1) is established. Notice also that the only solution to~\eqref{eq:9} and~\eqref{eq:10} implies the equality in~\eqref{eq:9}, and this is equivalent to having two equalities in~\eqref{eq:ccO}. Hence we may apply the last clause of part (3) of Theorem~\ref{thm:ranks}, which proves part (2).
\end{proof}

\bigskip
We will need the following graph theoretic construction. 

Let $\Omega=\Omega_{ab}\bigvee_{\Omega_{abc}}\Omega_{bc}\bigvee_{\Omega_{abc}}\Omega_{ac}=\Omega_u\sqcup\Omega_v$, $\Omega_a$, $\Omega_b$, $\Omega_c$ be the the Dicks graphs defined for given core graphs $\G_H$, $\G_K$, $\G_{H\cap K}$, and let $\tilde o$, $\tilde t$ be the embeddings of $\Omega_x\hookrightarrow \Omega_y$ ($x\in\{a,b,c\}$, $y\in\{u,v\}$) defined in Section~\ref{sec:dicks}. Let, as before, $A=\Omega_{ab}\bigvee_{\Omega_{abc}}\Omega_{ac}$, $B=\Omega_{ab}\bigvee_{\Omega_{abc}}\Omega_{bc}$, $C=\Omega_{ac}\bigvee_{\Omega_{abc}}\Omega_{bc}$. For any finite connected bipartite undirected graph $\Delta$ (with a fixed bipartite structure) define the following directed graph, which we will call the \emph{subgraph isomorphism graph} for $\Delta$, and denote it $\SIG(\Delta)$:

Vertices of $\SIG(\Delta)$:
\[
V\big(\!\SIG(\Delta)\big)=\{\text{subgraphs } \G\subset \Omega, \text{ such that } \G\subset A, \text{ or } \G\subset B, \text{ or } \G\subset C, \text{ and } \G\cong\Delta\},
\]
where $\cong$ is the isomorphism of bipartite graphs, i.e.\ it is required to send parts of bipartite structure of $\G$ (induced by that of $\Omega$) into the corresponding parts of $\Delta$.

Since $\Omega=\Omega_u\sqcup\Omega_v$, we may define the set of directed edges of $\SIG(\Delta)$ by specifying the stars of the `source' vertices $\G\subset\Omega_u$. Let $\G$ be a vertex of $\SIG(\Delta)$ such that $\G$, viewed as a subgraph of $\Omega$, lies in $A\cap\Omega_u$. Then $\G$ lies in the image under $\tilde o$ of some connected component $Q\subset \Omega_a$. Hence $\G'=\tilde t\circ (\tilde o|_Q)^{-1}(\G)$ is another vertex of $\SIG(\Delta)$, with $\G'\subset A\cap\Omega_v$. We connect $\G$ and $\G'$ in $\SIG(\Delta)$ with a directed edge labeled $a$ with the origin $\G$ and the terminus $\G'$. Similarly, if $\G\subset B\cap\Omega_u$, the star of $\G$ in $\SIG(\Delta)$ will have an outgoing $b$--edge, and if $\G\subset C\cap\Omega_u$, the star of $\G$ will have an outgoing $c$--edge, with their termini defined correspondingly. Thus, a vertex $\G$ of $\SIG(\Delta)$ may have valence $1$, $2$, or $3$, if $\G$ lies in only one of the subsets $A$, $B$, $C$, or in only two of them, or in all three, respectively. Clearly, the same is true for vertices $\G'\subset\Omega_v$.

In other words, edges of $\SIG(\Delta)$ are in $1{-}1$ correspondence with the subgraphs of $\Omega_a$, $\Omega_b$, $\Omega_c$ which are isomorphic to $\Delta$, with the restrictions of $\tilde o$, $\tilde t$ as the origin and the terminus maps.

Notice also that $\SIG(\Delta)$ admits a natural immersion into the topological pushout $\T$, since the vertices and edges of $\SIG(\Delta)$ are naturally mapped into the vertices and edges of $\T$, and this mapping is injective on stars.

Figure~\ref{fig:sig} shows the graph $\SIG(K_{1,1})$ for the Dicks graphs in Figure~\ref{fig:toppush}.

\begin{figure}[ht]
\begin{center}
\begin{tikzpicture}[scale=0.75]
\begin{scope}[xshift=-4.5cm]
\fill (0,0) circle (2.5pt); \fill (1,0) circle (2.5pt); \draw[thick] (0,0)--(1,0); 
\draw (0,-0.3) node {$8$}; \draw (1,-0.3) node {$2$};
\end{scope}
\begin{scope}[xshift=4.5cm]
\fill (0,0) circle (2.5pt); \fill (1,0) circle (2.5pt); \draw[thick] (0,0)--(1,0); 
\draw (0,-0.3) node {$7$}; \draw (1,-0.3) node {$3$};
\end{scope}
\begin{scope}[xshift=-2.5cm, yshift=3.46cm]
\fill (0,0) circle (2.5pt); \fill (1,0) circle (2.5pt); \draw[thick] (0,0)--(1,0); 
\draw (0,-0.3) node {$7$}; \draw (1,-0.3) node {$1$};
\end{scope}
\begin{scope}[xshift=2.5cm, yshift=3.46cm]
\fill (0,0) circle (2.5pt); \fill (1,0) circle (2.5pt); \draw[thick] (0,0)--(1,0); 
\draw (0,-0.3) node {$6$}; \draw (1,-0.3) node {$2$};
\end{scope}
\begin{scope}[xshift=2.5cm, yshift=-3.46cm]
\fill (0,0) circle (2.5pt); \fill (1,0) circle (2.5pt); \draw[thick] (0,0)--(1,0); 
\draw (0,-0.3) node {$6$}; \draw (1,-0.3) node {$4$};
\end{scope}
\begin{scope}[xshift=-2.5cm, yshift=-3.46cm]
\fill (0,0) circle (2.5pt); \fill (1,0) circle (2.5pt); \draw[thick] (0,0)--(1,0); 
\draw (0,-0.3) node {$5$}; \draw (1,-0.3) node {$3$};
\end{scope}

\draw[->-=0.56,color=mygreen,very thick] (-2.75,2.75) to (-4,0.75);
\draw[->-ii=0.56,color=blue,very thick] (-2.75,-3) to (-4,-0.75);
\draw[->-i=0.53,color=red,very thick] (5,-0.75) to (3.75,-3);
\draw[->-ii=0.56,color=blue,very thick] (5,0.75) to (3.75,2.75);
\draw[->-i=0.53,color=red,very thick] (-0.85,3.5) to (1.85,3.5);
\draw[->-=0.58,color=mygreen,very thick] (-0.85,-3.5) to (1.85,-3.5);

\begin{scope}[thick, xshift=0.25cm, yshift=2cm, xscale=0.3, yscale=0.25, red]
\fill (0,0) circle (6pt);
\fill (2,0) circle (6pt);
\fill (1,3) circle (6pt);
\draw (0,0)--(1,3);
\draw (2,0)--(1,3);
\draw [black] (-0.4,-0.9) node {\scriptsize(1,2)}; 
\draw [black] (2.4,-0.9) node {\scriptsize(3,4)}; 
\draw [black] (1,3.9) node {\scriptsize(7,6)}; 
\end{scope}

\begin{scope}[thick, xshift=5cm, yshift=-2.5cm, xscale=0.3, yscale=0.25, red]
\fill (0,0) circle (6pt);
\fill (2,0) circle (6pt);
\fill (1,3) circle (6pt);
\draw (0,0)--(1,3);
\draw (2,0)--(1,3);
\draw [black] (-0.4,-0.9) node {\scriptsize(1,2)}; 
\draw [black] (2.4,-0.9) node {\scriptsize(3,4)}; 
\draw [black] (1,3.9) node {\scriptsize(7,6)}; 
\end{scope}

\begin{scope}[thick, xshift=5cm, yshift=1.75cm, xscale=0.3, yscale=0.25, blue]
\fill (1,0) circle (6pt);
\fill (0,3) circle (6pt);
\fill (2,3) circle (6pt);
\draw (1,0)--(0,3);
\draw (2,3)--(1,0);
\draw [black] (1,-0.7) node {\scriptsize(3,2)}; 
\draw [black] (-0.4,3.8) node {\scriptsize(5,8)}; 
\draw [black] (2.4,3.8) node {\scriptsize(7,6)}; 
\end{scope}

\begin{scope}[thick, xshift=-5cm, yshift=-2.55cm, xscale=0.3, yscale=0.25, blue]
\fill (1,0) circle (6pt);
\fill (0,3) circle (6pt);
\fill (2,3) circle (6pt);
\draw (1,0)--(0,3);
\draw (2,3)--(1,0);
\draw [black] (1,-0.7) node {\scriptsize(3,2)}; 
\draw [black] (-0.4,3.8) node {\scriptsize(5,8)}; 
\draw [black] (2.4,3.8) node {\scriptsize(7,6)}; 
\end{scope}

\begin{scope}[thick, xshift=-4.25cm, yshift=1.8cm, xscale=0.3, yscale=0.25, mygreen]
\fill (0,0) circle (6pt);
\fill (0,3) circle (6pt);
\draw (0,0)--(0,3);
\draw [black] (0,-0.7) node {\scriptsize(1,2)}; 
\draw [black] (0,3.8) node {\scriptsize(7,8)}; 
\end{scope}

\begin{scope}[thick, xshift=-0.05cm, yshift=-2.75cm, xscale=0.3, yscale=0.25, mygreen]
\fill (2,0) circle (6pt);
\fill (2,3) circle (6pt);
\draw (2,0)--(2,3);
\draw [black] (2,-0.7) node {\scriptsize(3,4)}; 
\draw [black] (2,3.8) node {\scriptsize(5,6)}; 
\end{scope}

\end{tikzpicture}
\end{center}
\caption{The graph $\SIG(K_{1,1})$ for the Dicks graphs in Figure~\ref{fig:toppush}.\label{fig:sig}}
\end{figure}

\bigskip
An undirected graph $\G$ is called \emph{$k$--connected}, for $k\in\N$, if $\#V(\G)>k$ and $\G\setminus Y$ is connected for every $Y\subset V(\G)$ with $\#Y<k$. We will make use of the following global version of Menger's theorem, see~\cite[Th.\,3.3.6\,(i)]{Dies}.
\begin{menger}
A graph is $k$--connected if and only if it contains $k$ independent paths between any two vertices.\qed
\end{menger}
Recall that a \emph{path} in an undirected graph is a sequence $x_1e_1x_2e_2\dots x_ke_kx_{k+1}$ of pairwise distinct vertices $x_i$ and undirected edges $e_i$ such that for all $i$, vertices $x_i,x_{i+1}$ are incident to edge $e_i$. Two paths from $x$ to $y$ are \emph{independent} if they share no other vertices except $x$ and $y$.

Finally, we are ready to prove Theorem~\ref{thm:guzman4}. 

\begin{proof}[Proof of Theorem~\ref{thm:guzman4}]
Construct all the Dicks graphs $\Omega_\bullet$ for the core graphs $\G_H$, $\G_K$, $\G_{H\cap K}$. Then graphs $\Omega$ and $\Omega_{abc}$ satisfy conditions (1) and (2) of Proposition~\ref{prop:recast}. In particular, one connected component of $\Omega_{abc}$ is isomorphic to the complete bipartite graph $K_{i,i-1}$, and all other components are singleton vertices. Clearly, graph $K_{i,i-1}$ is $2$--connected, if $i\ge 3$. For the rest of the proof, let $\Delta$ denote $K_{i,i-1}$. (The proof will be valid for an arbitrary $2$--connected graph $\Delta$.) 

Consider the graph $\SIG(\Delta)$ built for the Dicks graphs constructed above. It has a single vertex of valence $3$, since, by part (1) of Proposition~\ref{prop:recast}, only one subgraph isomorphic to $\Delta$ exists in $\Omega_{abc}=A\cap B\cap C$. We claim that all other vertices of $\SIG(\Delta)$ have valence $2$.

Suppose the contrary, that there exists a vertex $\G$ of $\SIG(\Delta)$ which has valence $1$. This means that subgraph $\G\subset\Omega$ lies in only one of subgraphs $A$, $B$, $C$ of $\Omega$, let's say $\G\subset A=\Omega_{ab}\bigvee_{\Omega_{abc}}\Omega_{ac}$. Since $\G\not\subset B$ and $\G\not\subset C$, we conclude that $\G\cap (\Omega_{ab}\setminus\Omega_{abc})\ne\varnothing$ and $\G\cap (\Omega_{ac}\setminus\Omega_{abc})\ne\varnothing$.

If there exist vertices $p$, $q$ such that $p\in V(\G)\cap (\Omega_{ab}\setminus\Omega_{abc})$ and $q\in V(\G)\cap (\Omega_{ac}\setminus\Omega_{abc})$, then by Menger's theorem above, there exist two independent paths from $p$ to $q$, and their union is a cycle which does not lie entirely in either of the subgraphs $\Omega_{ab}$, $\Omega_{ac}$, $\Omega_{bc}$, thus contradicting condition~(2) of Proposition~\ref{prop:recast}.

Assume now that $V(\G)\cap (\Omega_{ab}\setminus\Omega_{abc})=\varnothing$ but there exists an edge $e\in E(\G)\cap (\Omega_{ab}\setminus\Omega_{abc})$. Let $s,t$ denote the vertices incident to the edge $e$. The last two conditions imply that $s,t\in V(\Omega_{abc})$. Form a new graph $\G'$ by subdividing the edge $e$ into a sequence $e_1pe_2$ of two edges $e_1, e_2$ and a new vertex $p$ such that the vertices $p$ and $s$ are incident to $e_1$ and the vertices $p$ and $t$ are incident to $e_2$. We claim that the graph $\G'$ obtained this way is also $2$--connected. First, observe that $s\ne t$ since $\Omega$ has no loops (being bipartite). If we remove the vertex $p$ from $\G'$ the result is the same as if we remove the edge $e$ from $\G$. Since $\G$ is $2$--connected, Menger's theorem guarantees the existence of another path $s{-}t$ in $\G$ which doesn't contain the edge $e$. Hence, $\G\setminus e$ is still connected, and so is $\G'\setminus p$. Also, the removal of any other vertex $p'\ne p$ from $\G'$ doesn't make the resulting graph disconnected. Indeed, for any two vertices $p_1$, $p_2$ of $\G$ there exist at least two independent paths between them, by Menger's theorem, and only one of them may contain $p'$. This means that $p_1$, $p_2$ are still connected via the other path in $\G\setminus p'$ and hence in $\G'\setminus p'$. Also, the vertex $p$ is connected to any other vertex of $\G'\setminus p'$ since $s\ne t$.

If $V(\G)\cap (\Omega_{ac}\setminus\Omega_{abc})\ne\varnothing$, pick a vertex $q$ in that subset. Otherwise, as before, there is some edge $e'\in E(\G)\cap (\Omega_{ac}\setminus\Omega_{abc})$, with the endpoints $s'\ne t'\in V(\Omega_{abc})$, and we perform the above operation of subdivision of edge again, applied to $e'$, thus obtaining another $2$--connected graph with $e'$ changed into $e'_1qe'_2$. (For simplicity, we will still denote this graph by $\G'$.)

Applying Menger's theorem again, we see that there exist two independent paths from $p$ to $q$ in $\G'$. Their union is a cycle in $\G'$, which has subpaths $se_1pe_2t$ and $s'e'_1qe'_2t'$, see Figure~\ref{fig:cycle}. Going back to the original graph $\G$ and replacing these subpaths with the subpaths $s\,e\,t$ and $s'e't'$, respectively, (the latter only if we performed the subdivision of edges twice), 
we get a cycle in the original graph $\G$ which does not lie entirely in either of the subgraphs $\Omega_{ab}$, $\Omega_{ac}$, $\Omega_{bc}$, thus contradicting condition~(2) of Proposition~\ref{prop:recast}. 

\begin{figure}[ht]
\begin{center}
\begin{tikzpicture}[scale=1.5]
\draw[xshift=1cm, rounded corners=13pt] (0,0) arc (0:120:2) arc (180:360:2) arc (60:180:2);
\draw [very thick, magenta] (-0.25,1.25) .. controls (-1,2) and (-2,1) .. (-0.75,0.25);
\fill[white] (-1.25,1.3) circle (1.5pt); \draw (-0.75,1.65) node {$e_1$}; \draw (-1.3,0.5) node {$e_2$};
\draw[thick] (-1.25,1.3) circle (1.5pt); \draw (-1.1,1.2) node {$p$};

\draw [very thick, orange] (0.25,1.25) .. controls (1,2) and (2,1) .. (0.75,0.25);
\fill[white] (1.25,1.3) circle (1.5pt); \draw (0.75,1.7) node {$e'_1$}; \draw (1.3,0.5) node {$e'_2$};
\draw[thick] (1.25,1.3) circle (1.5pt); \draw (1.1,1.2) node {$q$}; 

\fill (-0.25,1.25) circle (1.5pt); \draw (-0.3,1.1) node {$s$};
\fill (-0.75,0.25) circle (1.5pt); \draw (0.3,1.1) node {$s'$};
\fill (0.25,1.25) circle (1.5pt); \draw (-0.7,0.45) node {$t$};
\fill (0.75,0.25) circle (1.5pt); \draw (0.7,0.45) node {$t'$};

\draw [very thick, snake it] (-0.25,1.25) to [out=-30, in=210] (0.25,1.25);
\draw [very thick, snake it] (-0.75,0.25) to [out=-30, in=210] (0.75,0.25);

\draw (-1.7,1.5) node {$\Omega_{ab}$};
\draw (1.7,1.5) node {$\Omega_{ac}$};
\draw (0,0.65) node {$\Omega_{abc}$};

\end{tikzpicture}
\end{center}
\caption{A cycle in $\G'$ from the proof of Theorem~\ref{thm:guzman4}. 
\label{fig:cycle}}
\end{figure}
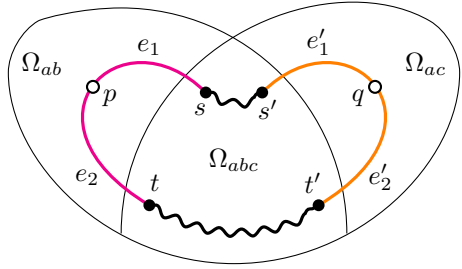

Therefore, all vertices of the graph $\SIG(\Delta)$ have valence $2$, except for a single vertex of valence $3$. But this is impossible, since in any graph the number of vertices of odd valence must be even (otherwise the count for the number of edges, $\#E^+(\G)=\frac12\sum_{v\in V(\G)}\val(v)$, would be a half-integer).

The obtained contradiction proves that the values of the ranks of $H$, $K$, $H\vee K$, and $H\cap K$ in Theorem~\ref{thm:guzman4} are non-realizable.
\end{proof}

{\sloppy
\begin{example}\label{ex:k23}
Interestingly, condition (1) alone in Proposition~\ref{prop:recast} does not make the Dicks graphs 
non-realizable, as the following example shows. Let subgroups $H,K\le \theta(F)\le F(a,b,c)$ be given by $H=\theta\big(\langle b,a^3,ab^{-1}a,ab^2a^{-1}\rangle\big)=\langle cb^{-1}, (ca^{-1})^3, ca^{-1}ba^{-1},ca^{-1}(cb^{-1})^2ac^{-1}\rangle$, $K=\theta\big(\langle a^{-1}b,b^{-2}ab^2\rangle\big)=
\langle ab^{-1},bc^{-1}ba^{-1}cb^{-1}cb^{-1}\rangle$. Then $\Omega_{abc}=K_{2,3}\cup\{\text{three vertices}\}$, see Figure~\ref{fig:k23}. The graph SIG($K_{2,3}$) has one vertex of valence~$3$ and three vertices of valence~$1$. Note also that $\text{(\# connected components of $\Omega_{abc}$)}=4>2=2\rr(\T)$, cf.~part~(3) of Theorem~\ref{thm:ranks}. 
\end{example}
}

\begin{figure}[!htp]
\begin{center}
\begin{tikzpicture}[thick, scale=1.0]
\begin{scope}[yshift=-0.5cm]
\draw[->-ii=0.6,color=blue] (1,-0.5) to [out=90, in=90,looseness=0.75] (2,-0.5);
\draw[->-=0.6,color=mygreen] (1,-0.5) to [out=-90, in=-90,looseness=0.75] (2,-0.5);
\draw[->-i=0.525,color=red] (1,-0.5) to [out=90, in=90,looseness=0.6] (4,-0.5);
\draw[->-=0.55,color=mygreen] (3,-0.5) to [out=90, in=90,looseness=0.6] (6,-0.5);
\draw[->-i=0.56,color=red] (3,-0.5) to [out=180, in=0,looseness=1] (2,-0.5);
\draw[->-ii=0.56,color=blue] (3,-0.5) to [out=0, in=180,looseness=1] (4,-0.5);
\draw[->-=0.56,color=mygreen] (5,-0.5) to [out=180, in=0,looseness=1] (4,-0.5);
\draw[->-i=0.55,color=red] (5,-0.5) to [out=90, in=90,looseness=0.75] (6,-0.5);
\draw[->-ii=0.6,color=blue] (5,-0.5) to [out=-90, in=-90,looseness=0.75] (6,-0.5);
\draw (1,-0.9) node {$1$};
\draw (2,-0.9) node {$2$};
\draw (3,-0.9) node {$3$};
\draw (4,-0.9) node {$4$};
\draw (5,-0.9) node {$5$};
\draw (6,-0.9) node {$6$};
\end{scope}

\begin{scope}[xshift=2.5cm]
\draw[->-=0.6,color=mygreen] (6,1) to [out=180, in=180,looseness=0.75] (6,2);
\draw[->-i=0.5,color=red] (6,1) to [out=0, in=0,looseness=0.75] (6,2);
\draw[->-ii=0.6,color=blue] (6,3) to [out=-90, in=90,looseness=0.75] (6,2);
\draw[->-=0.6,color=mygreen] (6,3) to [out=90, in=-90,looseness=0.75] (6,4);
\draw[->-ii=0.6,color=blue] (6,5) to [out=180, in=180,looseness=0.75] (6,4);
\draw[->-i=0.55,color=red] (6,5) to [out=0, in=0,looseness=0.75] (6,4);
\draw (6.4,5) node {$7$};
\draw (6.4,4) node {$8$};
\draw (6.4,3) node {$9$};
\draw (6.4,2) node {$10$};
\draw (6.4,1) node {$11$};
\draw (6,5) circle (2.5pt);
\draw (7,3) node {$\G_K$};
\end{scope}

\draw[->-ii=0.6,color=blue] (1,5)--(2,4);
\draw[->-ii=0.6,color=blue] (1,3)--(2,2);
\draw[->-ii=0.6,color=blue] (3,5)--(4,4);
\draw[->-ii=0.6,color=blue] (3,3)--(4,2);
\draw[->-ii=0.6,color=blue] (5,3)--(6,2);
\draw[->-ii=0.6,color=blue] (5,5) to [out=-90, in=180,looseness=0.65] (6,4);
\draw[->-=0.6,color=mygreen] (1,3) to (2,4);
\draw[->-=0.6,color=mygreen] (1,1) to (2,2);
\draw[->-=0.7,color=mygreen] (3,3) to (6,4);
\draw[->-=0.7,color=mygreen] (3,1) to (6,2);
\draw[->-=0.8,color=mygreen] (5,1) to (4,2);
\draw[->-=0.8,color=mygreen] (5,3) to (4,4);
\draw[->-i=0.75,color=red] (1,5)--(4,4);
\draw[->-i=0.75,color=red] (1,1)--(4,2);
\draw[->-i=0.75,color=red] (3,5)--(2,4);
\draw[->-i=0.75,color=red] (3,1)--(2,2);
\draw[->-i=0.55,color=red] (5,1)--(6,2);
\draw[->-i=0.55,color=red] (5,5) to [out=0, in=90,looseness=0.65] (6,4);

\foreach \x in {1,2,3,4,5,6}
{
	\fill (\x,-1) circle (1.5pt);
		
	\foreach \y in {1,2,3,4,5}
	{
			\fill (\x,\y) circle (1.5pt);
	}
}
\foreach \y in {1,2,3,4,5}
{
			\fill (8.5,\y) circle (1.5pt);
}		

\draw (1,-1) circle (2.5pt);
\draw (1,5) circle (2.5pt);

\begin{scope}[xshift=2.5cm, yshift=-0.5cm]
\draw[->-i=0.55,thick,color=red] (5.5,-0.5) to [out=90, in=90,looseness=0.75] (6.5,-0.5);
\draw[->-ii=0.6,thick,color=blue] (5.5,-0.5) to [out=0, in=180,looseness=0.75] (6.5,-0.5);
\draw[->-=0.6,thick,color=mygreen] (5.5,-0.5) to [out=-90, in=-90,looseness=0.75] (6.5,-0.5);
\fill (5.5,-0.5) circle (1.5pt);
\fill (6.5,-0.5) circle (1.5pt);
\draw (7.25,-0.5) node {$X\cong\T$};
\draw (5.5,-0.5) circle (2.5pt);
\end{scope}

\draw (0,-1) node {$\G_H$:};
\draw (0,3) node {$\G_{H\cap K}$:};

\begin{scope}[xshift=0.5cm]
\draw[->,semithick] (6.15,3) to (7.3,3); 
\draw (6.75,3.35) node {$\Pi_K$};
\draw[->,semithick] (3.5,0.45) to (3.5,-0.15); 
\draw (3.95,0.15) node {$\Pi_H$};
\draw[->,semithick] (6.05,-1) to (7.2,-1); 
\draw (6.6,-0.625) node {$p_H$};
\draw[->,semithick] (8,0.45) to (8,-0.35); 
\draw (8.5,0.05) node {$p_K$};
\end{scope}

\begin{scope}[xshift=1.5cm,yshift=2.5cm]
\draw (-2,-5) node {\textit{Legend:}};
\draw (0,-5) node {$a$--edges:};
\draw[->-i=0.58,color=red,thick] (0.85,-5) to (2,-5) [out=0, in=180];
\draw (3.45,-5) node {$b$--edges:};
\draw[->-ii=0.58,color=blue,thick] (4.3,-5) to (5.45,-5) [out=0, in=180];
\draw (6.75,-5) node {$c$--edges:};
\draw[->-=0.64,color=mygreen,thick] (7.60,-5) to (8.75,-5) [out=0, in=180];
\end{scope}

\begin{scope}[yshift=-5.5cm,xshift=-1cm,scale=0.5]
\fill (0,0) circle (3pt); \fill (1,0) circle (3pt); \fill (2,0) circle (3pt);
\fill (0,2) circle (3pt); \fill (1,2) circle (3pt); \fill (2,2) circle (3pt);
\draw (0,0)--(0,2)--(1,0)--(1,2)--(0,0);
\draw (0,0)--(2,2)--(2,0)--(0,2);
\draw (1,0)--(2,2); \draw (1,2)--(2,0);
\draw (-1,1) node {$\Omega_u$:};
\draw (0,-0.5) node {7};
\draw (1,-0.5) node {9};
\draw (2,-0.5) node {11};
\draw (0,2.5) node {1};
\draw (1,2.5) node {3};
\draw (2,2.5) node {5};
\end{scope}

\begin{scope}[yshift=-5.5cm,xshift=1.5cm,scale=0.5]
\fill (0.5,0) circle (3pt); \fill (1.5,0) circle (3pt);
\fill (0,2) circle (3pt); \fill (1,2) circle (3pt); \fill (2,2) circle (3pt);
\draw (0.5,0)--(0,2)--(1.5,0)--(1,2)--(0.5,0)--(2,2)--(1.5,0);
\draw (-1,1) node {$\Omega_v$:};
\draw (0.5,-0.5) node {8};
\draw (1.5,-0.5) node {10};
\draw (0,2.5) node {2};
\draw (1,2.5) node {4};
\draw (2,2.5) node {6};
\end{scope}

\begin{scope}[yshift=-5.5cm,xshift=4cm,scale=0.5,red]
\fill (0.5,0) circle (3pt); \fill (1.5,0) circle (3pt);
\fill (0,2) circle (3pt); \fill (1,2) circle (3pt); \fill (2,2) circle (3pt);
\draw (0.5,0)--(0,2)--(1.5,0)--(1,2)--(0.5,0)--(2,2)--(1.5,0);
\draw[black] (-0.75,1) node {$\Omega_a$:};
\draw[black] (0.25,-0.5) node {\scriptsize(7,8)};
\draw[black] (1.75,-0.5) node {\scriptsize(11,10)};
\draw[black] (-0.25,2.5) node {\scriptsize(1,4)};
\draw[black] (1,2.5) node {\scriptsize(3,2)};
\draw[black] (2.25,2.5) node {\scriptsize(5,6)};
\end{scope}

\begin{scope}[yshift=-5.5cm,xshift=6.5cm,scale=0.5,blue]
\fill (0.5,0) circle (3pt); \fill (1.5,0) circle (3pt);
\fill (0,2) circle (3pt); \fill (1,2) circle (3pt); \fill (2,2) circle (3pt);
\draw (0.5,0)--(0,2)--(1.5,0)--(1,2)--(0.5,0)--(2,2)--(1.5,0);
\draw[black] (-0.75,1) node {$\Omega_b$:};
\draw[black] (0.25,-0.5) node {\scriptsize(7,8)};
\draw[black] (1.75,-0.5) node {\scriptsize(9,10)};
\draw[black] (-0.25,2.5) node {\scriptsize(1,2)};
\draw[black] (1,2.5) node {\scriptsize(3,4)};
\draw[black] (2.25,2.5) node {\scriptsize(5,6)};
\end{scope}

\begin{scope}[yshift=-5.5cm,xshift=9cm,scale=0.5,mygreen]
\fill (0.5,0) circle (3pt); \fill (1.5,0) circle (3pt);
\fill (0,2) circle (3pt); \fill (1,2) circle (3pt); \fill (2,2) circle (3pt);
\draw (0.5,0)--(0,2)--(1.5,0)--(1,2)--(0.5,0)--(2,2)--(1.5,0);
\draw[black] (-0.75,1) node {$\Omega_c$:};
\draw[black] (0.25,-0.5) node {\scriptsize(9,8)};
\draw[black] (1.75,-0.5) node {\scriptsize(11,10)};
\draw[black] (-0.25,2.5) node {\scriptsize(1,2)};
\draw[black] (1,2.5) node {\scriptsize(3,6)};
\draw[black] (2.25,2.5) node {\scriptsize(5,4)};
\end{scope}

\begin{scope}[yshift=-7.75cm,xshift=0cm,scale=0.5]
\draw[magenta] (1,0)--(0,2); \draw[magenta] (1,0)--(1,2); \draw[magenta] (1,0)--(2,2); 
\fill (1,0) circle (3pt); 
\fill (0,2) circle (3pt); \fill (1,2) circle (3pt); \fill (2,2) circle (3pt);
\draw (-1,1) node {$\Omega_{u,ab}$:};
\draw (1,-0.5) node {7};
\draw (0,2.5) node {1};
\draw (1,2.5) node {3};
\draw (2,2.5) node {5};
\end{scope}

\begin{scope}[yshift=-7.75cm,xshift=2.5cm,scale=0.5]
\draw[orange] (1,0)--(0,2); \draw[orange] (1,0)--(1,2); \draw[orange] (1,0)--(2,2); 
\fill (1,0) circle (3pt); 
\fill (0,2) circle (3pt); \fill (1,2) circle (3pt); \fill (2,2) circle (3pt);
\draw (-1,1) node {$\Omega_{u,ac}$:};
\draw (1,-0.5) node {11};
\draw (0,2.5) node {1};
\draw (1,2.5) node {3};
\draw (2,2.5) node {5};
\end{scope}

\begin{scope}[yshift=-7.75cm,xshift=5cm,scale=0.5]
\draw[cyan] (1,0)--(0,2); \draw[cyan] (1,0)--(1,2); \draw[cyan] (1,0)--(2,2); 
\fill (1,0) circle (3pt); 
\fill (0,2) circle (3pt); \fill (1,2) circle (3pt); \fill (2,2) circle (3pt);
\draw (-1,1) node {$\Omega_{u,bc}$:};
\draw (1,-0.5) node {9};
\draw (0,2.5) node {1};
\draw (1,2.5) node {3};
\draw (2,2.5) node {5};
\end{scope}

\begin{scope}[yshift=-7.75cm,xshift=8.5cm,scale=0.5]
\fill (0.5,0) circle (3pt); \fill (1.5,0) circle (3pt);
\fill (0,2) circle (3pt); \fill (1,2) circle (3pt); \fill (2,2) circle (3pt);
\draw (0.5,0)--(0,2)--(1.5,0)--(1,2)--(0.5,0)--(2,2)--(1.5,0);
\draw (-1.5,1.75) node {$\Omega_{v,ab}$:};
\draw (-1.5,1) node {$\Omega_{v,ac}$:};
\draw (-1.5,0.25) node {$\Omega_{v,bc}$:};
\draw (3,1) node {$=\Omega_v$};
\draw (0.5,-0.5) node {8};
\draw (1.5,-0.5) node {10};
\draw (0,2.5) node {2};
\draw (1,2.5) node {4};
\draw (2,2.5) node {6};
\end{scope}

\begin{scope}[yshift=-10cm,xshift=2.5cm,scale=0.5]
\fill (0,2) circle (3pt); \fill (1,2) circle (3pt); \fill (2,2) circle (3pt);
\draw (1,0) node {$\varnothing$};
\draw (-1.5,1) node {$\Omega_{u,abc}$:};
\draw (0,2.5) node {1};
\draw (1,2.5) node {3};
\draw (2,2.5) node {5};
\end{scope}

\begin{scope}[yshift=-10cm,xshift=6cm,scale=0.5]
\fill (0.5,0) circle (3pt); \fill (1.5,0) circle (3pt);
\fill (0,2) circle (3pt); \fill (1,2) circle (3pt); \fill (2,2) circle (3pt);
\draw (0.5,0)--(0,2)--(1.5,0)--(1,2)--(0.5,0)--(2,2)--(1.5,0);
\draw (-1.5,1) node {$\Omega_{v,abc}$:};
\draw (0.5,-0.5) node {8};
\draw (1.5,-0.5) node {10};
\draw (0,2.5) node {2};
\draw (1,2.5) node {4};
\draw (2,2.5) node {6};
\end{scope}

\begin{scope}[yshift=-15cm,xshift=1cm,scale=0.75]
\draw [thin,rounded corners=5pt] (0,3.464) arc (60:-60:2) arc (240:60:2) arc (360:180:2) arc (120:60:2);
\draw [magenta] (-1.5,3.25)--(-0.67,2.25); \draw [magenta] (-1.5,3.25)--(0,2.25); \draw [magenta] (-1.5,3.25)--(0.67,2.25);
\draw [orange] (1.5,3.25)--(-0.67,2.25); \draw [orange] (1.5,3.25)--(0,2.25); \draw [orange] (1.5,3.25)--(0.67,2.25);
\draw [cyan] (0,0.5)--(-0.67,2.25); \draw [cyan] (0,0.5)--(0,2.25); \draw [cyan] (0,0.5)--(0.67,2.25);
\fill (-0.67,2.25) circle (2pt);
\fill (0,2.25) circle (2pt);
\fill (0.67,2.25) circle (2pt);
\fill (-1.5,3.25) circle (2pt);
\fill (1.5,3.25) circle (2pt);
\fill (0,0.5) circle (2pt);
\draw (-0.8,2) node {\scriptsize1};
\draw (-0.15,2) node {\scriptsize3};
\draw (0.75,2) node {\scriptsize5};
\draw (-1.7,3) node {\scriptsize7};
\draw (1.7,3) node {\scriptsize11};
\draw (-0.25,0.5) node {\scriptsize9};

\draw[->-i=0.52,color=red,very thick] (2.5,3.5) to [out=45, in=135] (7.5,3.5);
\draw[red] (5.1,4) node {$\Omega_a$};
\draw[->-ii=0.52,color=blue,very thick] (2.5,2) to [out=0, in=180] (7.5,2);
\draw[blue] (5.1,2.5) node {$\Omega_b$};
\draw[->-=0.52,color=mygreen,very thick] (2,1) to [out=-45, in=225] (8,1);
\draw[mygreen] (5.1,0.25) node {$\Omega_c$};

\begin{scope}[yshift=0cm,xshift=10cm]
\draw [thin,rounded corners=5pt] (0,3.464) arc (60:-60:2) arc (240:60:2) arc (360:180:2) arc (120:60:2);
\begin{scope}[yshift=0.35cm,yscale=0.8]
\fill (-0.5,2) circle (2pt);
\fill (0,2) circle (2pt);
\fill (0.5,2) circle (2pt);
\fill (-0.33,3) circle (2pt);
\fill (0.33,3) circle (2pt);
\draw (-0.33,3)--(-0.5,2)--(0.33,3)--(0.5,2)--(-0.33,3); \draw (-0.33,3)--(0,2)--(0.33,3);
\end{scope}
\draw (-0.5,1.725) node {\scriptsize2};
\draw (0,1.725) node {\scriptsize4};
\draw (0.5,1.725) node {\scriptsize6};
\draw (-0.33,3) node {\scriptsize8};
\draw (0.25,3) node {\scriptsize10};
\end{scope}

\draw (-3,2) node {$\T:$};

\end{scope}
\end{tikzpicture}
\end{center}
\caption{The core graphs, the Dicks graphs and the topological pushout for the subgroups $H$, $K$ from Example~\ref{ex:k23}. The graph $\Omega_{abc}$ is the union of $\Omega_{v,abc}=K_{2,3}$ and $\Omega_{u,abc}=\{\text{three singleton vertices $1,3,5$}\}$.\label{fig:k23}}
\end{figure}
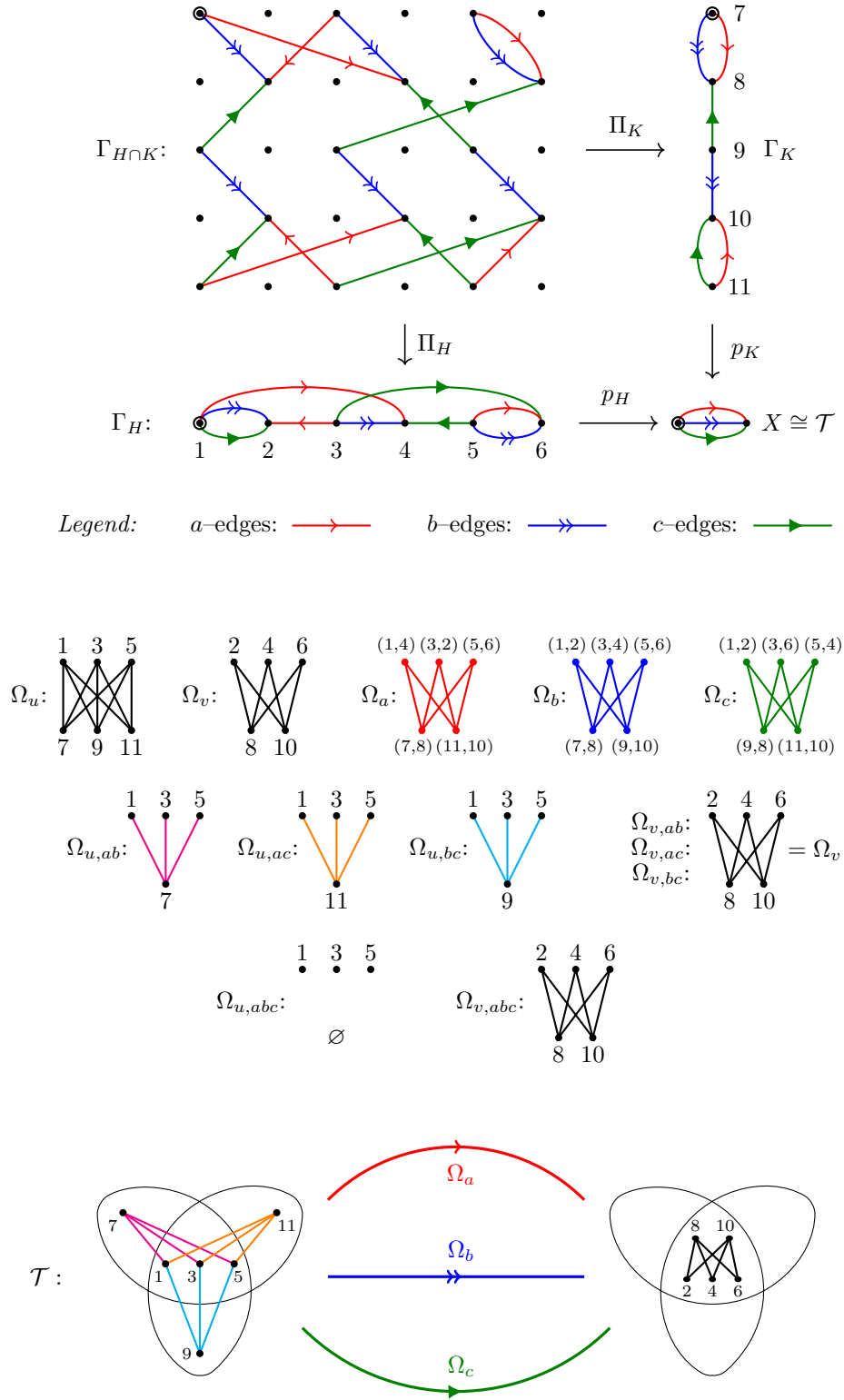

\bigskip
The case $i=3$ of Theorem~\ref{thm:guzman4} resolves the remaining open case $m=4$ of Guzman's ``Group-Theoretic Conjecture'' in the affirmative:
\begin{cor15}
Let $F$ be a free group. If two subgroups $H,K\le F$ both have ranks equal to\/ $4$, and\/ ${\rk (H\cap K)\ge 4}$, then\/ $\rk (H\vee K)\le 4$. 
\end{cor15}
\begin{proof}
Indeed, looking at the locus of known realizable values and the region of proved non-realizable values for $\rk(H)=\rk(K)=4$, see Figure~\ref{fig:guzman4}, we conclude that the GTC for $m=4$ holds true if and only if the tuple $\big(\!\rk(H\vee K), \rk(H\cap K)\big)=(5,4)$ is not realizable. But this is exactly what Theorem~\ref{thm:guzman4} says for $\rk(H)=\rk(K)=4$ and $i=3$.
\end{proof}

Invoking the implication theorem from~\cite{Gu}, we obtain a proof of the ``Geometric Conjecture'' for $k=6$:
\begin{cor16}
Let $M$ be a closed, orientable, hyperbolic $3$--manifold. If\/ $\pi_1(M)$ is\/ $6$--free then there exists a point $P$ in $M$ such that the set of all elements of\/ $\pi_1(M,P)$ represented by loops of length less than\/ $\log(11)$ is contained in a free subgroup of\/ $\pi_1(M)$ of rank at most $3$. \qed
\end{cor16}

\begin{figure}[ht]
\begin{center}
\begin{tikzpicture}[scale=1.0]
{\footnotesize
\draw (3.0cm,0.75cm) node {$\rk (H\cap K)$};
\draw (-0.3cm,-1.6cm) node [rotate=90] {$\rk H\vee K$};
}
{\scriptsize
\Ylinecolour{black!50}
\Yfillcolour{black!10}
\tgyoung(0cm,0cm,%
::::::::::::,%
::::::::::::,%
:::::::::;;;,%
:::::::;;;;;,%
:::::;?;;;;;;,%
::::;;;;;;;;,%
:::;;;;;;;;;,%
::;;;;;;;;;;%
)
\tgyoung(0cm,0cm,::0:1:2:3:4:5:6:7:8:9:<10>,%
:2;;;;;;;;;;;,%
:3;;;;;;;;,%
:4;;;;;;,%
:5;;;;,%
:6;;;,%
:7;;,%
:8;%
)

\Ylinecolour{black!25}
\Yfillcolour{black!50}
\Ynodecolour{white}
\tgyoung(0cm,0cm,::::::::::::::,%
:;;;;;;;;;;;,%
:;;;;;;;;,%
:;;;;;;,%
:;;;;,%
:;;;,%
:;;,%
:;%
)
\Ynodecolour{black}
\Yfillcolour{white}
\tgyoung(0cm,0cm,%
::::::::::::,%
::::::::::::,%
::::::::::::,%
:::::::::;;;,%
::::::;;;;;;,%
::::;;;;;;;;,%
:::;;;;;;;;;,%
::;;;;;;;;;;%
)
\begin{scope}[black,thick]
\draw (13pt,-91pt) -- ++(143pt,0) -- ++ (0,91pt) -- ++(-143pt,0) -- ++(0,-91pt);
\draw [red] (65pt,-91pt) -- ++(0,52pt) -- ++(91pt,0);
\end{scope}
}
\end{tikzpicture}
\caption{Guzman's GTC for $m=\rk(H)=\rk(K)=4$.\label{fig:guzman4}}
\end{center}
\end{figure}
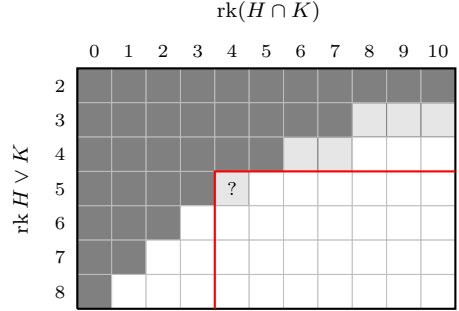

\bigskip
Now we are going to prove Theorem~\ref{thm:h2}.
\begin{proof}[Proof of Theorem~\ref{thm:h2}]
Let $k$, $v$ and $c$ be as in the statement of Theorem~\ref{thm:h2}. The existence of subgroups $H$, $K\le F$ with $\rk(H)=2$, $\rk(K)=k$, $\rk(H\vee K)=v$, and $\rk(H\cap K)=c$, such that $c\le k+2-v$, follows from Theorem~\ref{thm:real}. Let's prove that if $c>k+2-v$, such subgroups do not exist. 

Suppose the contrary, that subgroups $H$, $K$ with $\rk(H)=2$, $\rk(K)=k$, $\rk(H\vee K)=v$, $\rk(H\cap K)=c$, and $c>k+2-v$ do exist. Denote $i=k+2-v$ and let $d\ge1$ be such that $c=i+d$. From Theorem~\ref{thm:ranks} we see that $\Omega_{abc}$ is a bipartite graph with $2$ vertices in the $V(\G_H)$--part of $\Omega_{abc}$ and $2(k-1)$ vertices in the $V(\G_K)$--part (note that now $k$ denotes $\rk(K)$ so that $\rr(K)=k-1$), and that $\Omega_{abc}$ has $2(c-1)=2(i-1)+2d$ edges. Denote nontrivial components of $\Omega_{abc}$ as $C_1$, $C_2$, \dots, $C_\ell$ and call the edges of a spanning tree of $C_j$ the \emph{spanning edges} of $C_j$.  Arguing as in the proof of Proposition~\ref{prop:recast}, we get:
\begin{equation}\label{eq:11}
\sum_{j=1}^\ell \big(\text{\# spanning edges of $C_j$}\big)\le 2(i-1)
\end{equation}
and
\begin{equation}\label{eq:12}
\sum_{j=1}^\ell \big(\text{\# edges of $C_j$}\big)=2(i-1)+2d.
\end{equation}

Denote the two vertices in the $V(\G_H)$--part of $\Omega_{abc}$ as $z$ and $w$. We claim that $z$ and $w$ belong to the same nontrivial component $C_j$. Indeed, if $z$ and $w$ belong to different components, then these components are trees, and, in particular, the count in \eqref{eq:12} is equal to that of \eqref{eq:11}, which contradicts the condition $d\ge 1$. Hence, the component $C_j$ containing $z$ and $w$ is the only nontrivial component of $\Omega_{abc}$, i.e.\ $j=\ell=1$.

Denote 
\begin{align}
s&=\big(\text{\# spanning edges of $C_1$}\big),\label{eq:s}\\ 
m&=\big(\text{\# vertices of valence $2$ in the $V(\G_K)$--part of $\Omega_{abc}$}\big),\label{eq:m}\\
q&=\big(\text{\# singleton vertices in the $V(\G_K)$--part of $\Omega_{abc}$}\big),\label{eq:q}
\end{align}
see Figure~\ref{fig:Omega2}. 

\begin{figure}[ht]
\begin{center}
\begin{tikzpicture}
\fill (0,0) circle (2pt);
\fill (0.5,0) circle (2pt);
\fill (1,0) circle (2pt);
\fill (1.5,0) circle (2pt);
\fill (2,0) circle (2pt);
\fill (2.5,0) circle (2pt);
\fill (3,0) circle (2pt);
\fill (3.5,0) circle (2pt);
\fill (4,0) circle (2pt);
\fill (4.5,0) circle (2pt);
\fill (5,0) circle (2pt);
\fill (5.5,0) circle (2pt);
\fill (0.5,1) circle (2pt) node [above=5pt] {$z$};
\fill (2,1) circle (2pt) node [above=5pt] {$w$};
\draw [thick] (0,0)--(0.5,1)--(0.5,0)--(2,1)--(1,0)--(0.5,1)--(1.5,0)--(2,1)--(2,0)--(0.5,1);
\draw [thick] (2.5,0)--(2,1)--(3,0);
\draw [snake=brace,segment amplitude=2mm] (2.1,-0.15) -- (0.4,-0.15); \draw (1.25,-0.65) node {$m$}; 
\draw [snake=brace,segment amplitude=2mm] (5.6,-0.15) -- (3.4,-0.15); \draw (4.5,-0.65) node {$q$}; 
\draw (-1,0.5) node {$\Omega_{abc}:$};
\end{tikzpicture}
\caption{The graph $\Omega_{abc}$ when $\rk(H)=2$.\label{fig:Omega2}}
\end{center}
\end{figure}
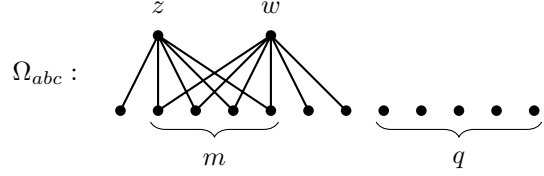

In particular, we see that $m\ge 2d+1$ (indeed, $m-1$ is the rank of $C_1$, i.\,e.\ the number of edges of $C_1$ minus the number of spanning edges of $C_1$, hence is at least $2d$), and that $\Omega_{abc}$ contains a subgraph $\Delta$ isomorphic to the complete bipartite graph $K_{2,m}$.

Consider the graph $\SIG(K_{2,m})$, as defined after the proof of Proposition~\ref{prop:recast}. Arguing as in the proof of Theorem~\ref{thm:guzman4}, we observe that $\SIG(K_{2,m})$ has a unique vertex of valence $3$, and hence must have another vertex of odd valence, that is of valence $1$. Let's call the subgraph of $\Omega$, corresponding to this valence $1$ vertex, $\Delta'$. Without loss of generality we can assume that $\Delta'$ lies in $A=\Omega_{ab}\bigvee_{\Omega_{abc}}\Omega_{ac}$, and hence that $\Delta'\cap \big(\Omega_{ab}\setminus\Omega_{abc}\big)\ne\varnothing$ and $\Delta'\cap \big(\Omega_{ac}\setminus\Omega_{abc}\big)\ne\varnothing$. We conclude that $\Delta'\cap\Omega_{abc}\ne\varnothing$, and, since the isomorphisms involved in the definition of $\SIG$ preserve the bipartite structure, the image of $m$ vertices of $K_{2,m}$ in $\Delta'$ (call this subset $M$) is a subset of the $q$ singleton vertices of the $V(\G_K)$--part of $\Omega_{abc}$. In particular, $m\le q$, see Figure~\ref{fig:val1}. 

\begin{figure}[ht]
\begin{center}
\begin{tikzpicture}[scale=1.0]
\draw[xshift=1cm, rounded corners=8.67pt] (0,0) arc (0:120:2) arc (180:360:2) arc (60:180:2);
\begin{scope}[yshift=5pt]
\fill (-0.75,0.25) circle (2pt); 
\fill (-0.25,0.25) circle (2pt);
\fill (0.25,0.25) circle (2pt);
\fill (0.75,0.25) circle (2pt); 
\fill (-1,1.3) circle (2pt); 
\fill (1,1.3) circle (2pt); 
\draw [thick] (-1,1.3)--(-0.75,0.25)--(1,1.3);
\draw [thick] (-1,1.3)--(-0.25,0.25)--(1,1.3);
\draw [thick] (-1,1.3)--(0.25,0.25)--(1,1.3);
\draw [thick] (-1,1.3)--(0.75,0.25)--(1,1.3);
\end{scope}
\draw (-1.6,1.3) node {$\scriptstyle\Omega_{ab}$};
\draw (1.6,1.3) node {$\scriptstyle\Omega_{ac}$};
\draw (0,1.35) node {$\scriptstyle\Omega_{abc}$};
\draw [snake=brace,segment amplitude=2mm] (0.8,0.3) -- (-0.8,0.3); \draw (0,-0.05) node {\scriptsize $m\le q$}; 
\end{tikzpicture}
\end{center}
\caption{The graph $\Delta'\cong K_{2,m}$ corresponding to the valence $1$ vertex of $\SIG(K_{2,m})$.
\label{fig:val1}}
\end{figure}
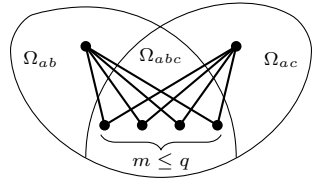

To estimate $\rr(\T)$ we recall equation~\eqref{eq:val} from Section~\ref{sec:dicks}:
\[
2\rr(\T)=\sum_{v \in V(\T)} \big(\!\val(v)-2\big).\tag{5}
\]
From Proposition~\ref{prop:toppush} we know that the vertices of $\T$ are connected components of $\Omega$. Vertices of $\T$ of valence $\ge 3$ correspond to certain subgraphs of $\Omega_{abc}$, with the exact relation between components of $\Omega_{abc}$ and the valence of the corresponding vertex of $\T$ given by the component connectivity graph $\G$, as constructed in the proof of Theorem~\ref{thm:ranks}. The vertices of $\G$ are connected components of $\Omega_{abc}$. Two vertices $p$, $q$ of $\G$ may be connected by up to three undirected edges (colored magenta, yellow, and cyan) in $\G$, if there exists an $\Omega_{abc}$-avoidant path connecting components $p$ and $q$ which lies entirely in $\Omega_{ab}$, $\Omega_{ac}$, or $\Omega_{bc}$, respectively. The right-hand sum of equation~\eqref{eq:val} equals the value of the function $\Sigma$ on $\G$, as defined in~\eqref{eq:sigma} of Section~\ref{sec:cn}. The main conclusion of Proposition~\ref{prop:cn} is that this value, and hence, the value $2\rr(\T)$, is bounded above by the number of vertices of $\G$.

The component connectivity graph $\G$ in the situation we are considering will have one vertex for the only nontrivial component $C_1$ of $\Omega_{abc}$ and $q$ vertices for the remaining singleton components. However, if we identify the $m$ vertices of $M$ in $\Delta'$ into a single vertex, thus forming a new graph $\G'$, we observe that the value of function $\Sigma$ on $\G$ and $\G'$ is the same. Indeed, these $m$ vertices belong to the same component in each of the subgraphs $A$, $B$, $C$ of $\Omega$, and hence their contribution to the number of $a$--edges, $b$--edges, and $c$--edges of $\T$ is the same as if they were a single vertex of $\Omega_{abc}$. Hence, we can use the graph $\G'$ for computing the quantity $2\rr(\T)$, and we conclude that the latter is bounded above by $n=\#V(\G')$. We now estimate $n$.

We have: $n=1+1+(q-m)$, where the ones correspond to the component $C_1$ and the subset $M$, which is one vertex of $\G'$. Since a spanning tree of $C_1$ contains only one vertex of valence $2$ in the $V(\G_K)$--part (for otherwise there would be a cycle in it), all other vertices of the spanning tree that lie in the $V(\G_K)$--part have valence $1$. Now from~\eqref{eq:q} and~\eqref{eq:s} above we deduce that:
\[
q=2(k-1)-(s-1).
\]
Also from~\eqref{eq:m} and~\eqref{eq:12} it follows that: 
\[
m=\big(\text{\# edges of $C_1$}\big)-\big(\text{\# spanning edges of $C_1$}\big) + 1=2(i-1)+2d-s+1.
\]
Hence, we get:
\[
n=2+q-m=2+2(k-1)-(s-1)-2(i-1)-2d+s-1=2(k-i-d+1),
\]
and we deduce from formula~\eqref{eq:val} above (and the discussion following it) that
\[
\rr(\T)\le k-i-d+1.
\]
Now recall that $v=\rk(H\vee K)=k+2-i$ and that $\rk(H\vee K)\le \rk(\T)=\rr(\T)+1$. We get:
\[
v=\rk(H\vee K)\le \rk(\T)=(k+2-i)-d = v-d,
\]
which yields a contradiction with the condition that $d\ge1$. Since $d$ was defined as $d=c-(k+2-v)$, this proves that $c+v\le k+2$. 
\end{proof}

\end{document}